\documentclass[12pt,reqno]{amsart}

\numberwithin{equation}{section}

\usepackage{amssymb}
\usepackage{amsmath}
\usepackage{amsbsy}
\usepackage{amscd}
\usepackage{amsthm}
\usepackage{amsfonts}
\usepackage{enumerate}

\usepackage{color}

\usepackage{ifpdf}
\ifpdf \usepackage[pdftex,pdfstartview=FitH,pdfpagemode=none,colorlinks,bookmarks,linkcolor=blue]{hyperref} \else  \usepackage[hypertex]{hyperref} \fi

\usepackage[nobysame,abbrev,alphabetic]{amsrefs}

\usepackage[nohug,heads=vee]{diagrams}
\diagramstyle[textstyle=\scriptstyle,labelstyle=\scriptstyle]

\newtheorem{theorem}{Theorem}[section]

\newtheorem{lemma}[theorem]{Lemma}
\newtheorem{corollary}[theorem]{Corollary}
\newtheorem{definition}[theorem]{Definition}

\newtheorem{proposition}[theorem]{Proposition}

\newtheorem{property}[theorem]{Property}

\newtheorem{remark}[theorem]{Remark}
\theoremstyle{definition}\newtheorem*{acknowledgments}{Acknowledgments}

\newcommand{\cA}{\mathcal{A}}
\newcommand{\cB}{\mathcal{B}}

\newcommand{\cE}{\mathcal{E}}

\newcommand{\cL}{\mathcal{L}}
\newcommand{\cM}{\mathcal{M}}

\newcommand{\cO}{\mathcal{O}}
\newcommand{\cP}{\mathcal{P}}

\newcommand{\cU}{\mathcal{U}}

\newcommand{\bC}{\mathbb{C}}

\newcommand{\bE}{\mathbb{E}}

\newcommand{\bP}{\mathbb{P}}
\newcommand{\bR}{\mathbb{R}}
\newcommand{\bZ}{\mathbb{Z}}
\newcommand{\bQ}{\mathbb{Q}}

\newcommand{\bN}{\mathbb{N}}

\newcommand{\bT}{\mathbb{T}}

\newcommand{\talpha}{{\tilde{\alpha}}}
\newcommand{\tX}{{\tilde{X}}}

\newcommand{\tx}{{\tilde{x}}}
\newcommand{\ty}{{\tilde{y}}}
\newcommand{\tmu}{{\tilde{\mu}}}
\newcommand{\txi}{{\tilde{\xi}}}

\newcommand{\hL}{{\hat{L}}}
\newcommand{\hbeta}{{\hat{\beta}}}

\newcommand{\rmm}{{\mathrm{m}}}

\newcommand{\dA}{{\dot{A}}}

\newcommand{\dG}{{\dot{G}}}

\newcommand{\dX}{{\dot{X}}}
\newcommand{\dmu}{{\dot{\mu}}}
\newcommand{\dtX}{{\dot{\tX}}}
\newcommand{\dtmu}{{\dot{\tmu}}}
\newcommand{\dtx}{{\dot{\tx}}}
\newcommand{\dGamma}{{\dot{\Gamma}}}
\newcommand{\dx}{{\dot{x}}}
\newcommand{\dv}{{\dot{v}}}

\newcommand{\gog}{\mathfrak{g}}
\newcommand{\goh}{\mathfrak{h}}
\newcommand{\gol}{\mathfrak{l}}

\newcommand{\gos}{\mathfrak{s}}

\newcommand{\gou}{\mathfrak{u}}
\newcommand{\gov}{\mathfrak{v}}
\newcommand{\gow}{\mathfrak{w}}

\newcommand{\bfe}{\mathbf{e}}
\newcommand{\bfm}{\mathbf{m}}
\newcommand{\bfn}{\mathbf{n}}
\newcommand{\bft}{\mathbf{t}}
\newcommand{\bfp}{\mathbf{p}}
\newcommand{\bfq}{\mathbf{q}}

\newcommand{\bfv}{\mathbf{v}}

\newcommand{\bfx}{\mathbf{x}}
\newcommand{\bfy}{\mathbf{y}}

\newcommand{\bfzero}{\mathbf{0}}
\newcommand{\bfone}{\mathbf{1}}
\newcommand{\bfxi}{{\boldsymbol{\xi}}}
\newcommand{\bfeta}{{\boldsymbol{\eta}}}

\newcommand{\SL}{\operatorname{SL}}

\newcommand{\GL}{\operatorname{GL}}
\newcommand{\SO}{\operatorname{SO}}

\newcommand{\supp}{\operatorname{supp}}

\newcommand{\rank}{\operatorname{rank}}
\newcommand{\dist}{\operatorname{dist}}

\newcommand{\Aut}{\operatorname{Aut}}

\newcommand{\Ad}{\operatorname{Ad}}

\newcommand{\Stab}{\operatorname{Stab}}

\newcommand{\di}{\mathrm{d}}

\newcommand{\Id}{\mathrm{Id}}

\newcommand{\ab}{\mathrm{ab}}
\newcommand{\T}{\mathrm{T}}

\newcommand\diag[1]{\operatorname{diag}\left(#1\right)}

\newcommand{\onto}{\xymatrix{\ar@{>>}[r]&}}
\newcommand{\da}[4]{\xymatrix{#1 \ar@<.5ex>[r]^{#2} \ar@<-.5ex>[r]_{#3} & #4}}

\newcounter{subconst}[subsection]

\newcounter{const}

\newcounter{CONST}

\begin{document}

\title{Multi-invariant measures and subsets on nilmanifolds}
\author[Z. Wang]{Zhiren Wang}
\address{\newline Yale University, New Haven, CT 06520, USA\newline \rm zhiren.wang@yale.edu}
\setcounter{page}{1}
\begin{abstract}Given a $\bZ^r$-action $\alpha$ on a nilmanifold $X$ by automorphisms and an ergodic $\alpha$-invariant probability measure $\mu$, we show that $\mu$ is the uniform measure on $X$, unless modulo finite index modification, one of the following obstructions occurs for an algebraic factor action: \begin{enumerate}\item The factor measure has zero entropy under every element of the action;
\item The factor action is virtually cyclic.\end{enumerate}

We also deduce a rigidity property for invariant closed subsets.
\end{abstract}
\maketitle
{\small\tableofcontents}

\section{Introduction}
\subsection{Main results}

A $\bZ^r$-action $\alpha$ on a topological space $\Omega$ is a group morphism $\alpha:\bfn\mapsto\alpha^\bfn$ from $\bZ^r$ to $\mathrm{Homeo}(\Omega)$,  the group of self-homeomorphisms of $\Omega$. The rank of the action is the torsion-free rank of the image of this morphism.
When the rank is $0$ or $1$, then the action is, up to torsion elements, generated by a single transformation, and we say in this case that $\alpha$ is {\bf virtually cyclic}.
 
By a nilmanifold we mean the quotient $X=G/\Gamma$, where $G$ is a simply connected nilpotent Lie group and $\Gamma$ is a cocompact lattice in $G$. Let $\rmm_X$ denote the unique left $G$-invariant probability measure on $X$. 

We focus on $\bZ^r$-actions on $X$ by automorphisms, that is, a group morphism $\alpha:\bfn\mapsto\alpha^\bfn$ from $\bZ^r$ to the automorphism group $\Aut(X)$ of $X$. An automorphism of $X$ is a homeomorphism of $X$ descending from a group automorphism of $G$ that preserves $\Gamma$. 

Our main result is: 
\begin{theorem}\label{MeasureThm} Let $\alpha:\bZ^r\curvearrowright X$ be a $\bZ^r$-action on a nilmanifold $X$ by automorphisms and $\mu$ be an ergodic $\alpha$-invariant probability measure on $X$. Then $\mu=\rmm_X$, unless there is a finite index subgroup $\Sigma\subset\bZ^r$, such that for every $\alpha|_\Sigma$-ergodic component $\mu'$ of $\mu$, there is a non-trivial algebraic factor action $\dot\alpha:\Sigma\curvearrowright\dX$ of the restriction $\alpha|_\Sigma$ of $\alpha$ to $\Sigma$, which satisfies at least one of the following:
\begin{enumerate}
\item The projection of $\mu'$ to $\dX$ has entropy $0$ for every $\dot\alpha^\bfn$,  $\bfn\in\Sigma$;
\item $\dot\alpha$ is virtually cyclic.
\end{enumerate} 
\end{theorem}

For definition of algebraic factor actions, see \S\ref{SecNilfiber}. It should be remarked that the second obstruction, namely the restriction $\alpha|_\Sigma$ to some finite index subgroup has a virtually cyclic algebraic factor, is equivalent to that $\alpha$ has such a factor (see Remark \ref{Rk2ResRmk}). The theorem is stated in its present form for conciseness.

When either of the two obstructions is present, no classification of invariant measures is available in the factor $\dX$, even in the simplest case where $\dX$ is a torus on which $\dot\alpha$ acts hyperbolically and totally irreducibly. In fact, the action being virtually cyclic is known to be a genuine obstruction that makes such classification impossible, in the sense that the action virtually becomes a Markov process and invariant measures can be quite arbitrary. On the other hand, it is a long-standing open problem and a generalization of Furstenberg Conjecture, to characterize zero entropy measures invariant under higher-rank actions (see \S\ref{SecBackground} below). In this sense, Theorem \ref{MeasureThm} is optimal within the scope of current technologies.

In certain special cases, Theorem \ref{MeasureThm} has been known by the works of Katok-Spatzier \cite{KS96}, Einsiedler-Lindenstrauss \cite{EL03} and Kalinin-Spatzier\cites{KS05}, see \S\ref{SecBackground}.

We also study the rigidity of $\alpha$-invariant subsets. For a nilmanifold $X=G/\Gamma$, there is a maximal torus factor $G/[G,G]\Gamma$, which we denote by $X_\ab$. Moreover, a $\bZ^r$-action by automorphisms of $X$ naturally projects to a factor $\bZ^r$-action on $X_\ab$ by toral automorphisms.

\begin{theorem}\label{TopoThm} Let $\alpha:\bZ^r\curvearrowright X$ be an action by nilmanifold automorphisms. Suppose $\alpha$ has no virtually cyclic algebraic factors.  

 If $A$ is an $\alpha$-invariant closed subset whose projection to the maximal torus factor $X_\ab$ is $X_\ab$ itself, then $A=X$.\end{theorem}
 
 A $\bZ^r$-action on a torus $\bT^d$ by automorphisms is {\bf totally irreducible} if the restriction to any finite index subgroup of $\bZ^r$ leaves invariant no proper non-trivial subtorus of $\bT^d$.
  
When the factor action of $\alpha$ on the maximal torus action $X_\ab$ is totally irreducible, one can say more about invariant measures and subsets on $X$.

\begin{corollary}\label{MeasureCor} Suppose $\alpha:\bZ^r\curvearrowright X$ is an action by nilmanifold automorphisms and it factors to a totally irreducible, non-virtually cyclic action on $X_{ab}$. If $\mu$ is an ergodic $\alpha$-invariant measure, then
\begin{itemize} 
\item either the projection of $\mu$ to $X_\ab$ has zero entropy under the projection of $\alpha^\bfn$ for all $\bfn\in\bZ^r$;  
\item or $\mu=m_X$.
\end{itemize}
\end{corollary}

We remark that behaviors of invariant subsets can be very different from those of invariant measures. In fact, after comparing Theorem \ref{TopoThm} with Theorem \ref{MeasureThm}, one might wonder if,  in Theorem \ref{TopoThm}, the weaker assumption that $A$ has full projection in every totally irreducible torus factor action would also force $A$ to be $X$. This is false in general, even when $X$ is a torus, unless $X_\ab$ is totally irreducible under the action.  But when the action on $X_\ab$ is totally irreducible and hyperbolic, we do have such rigidity:

\begin{corollary}\label{TopoCor} Let $\alpha:\bZ^r\curvearrowright X$ be  an action by nilmanifold automorphisms. Suppose the factor action on $X_\ab$ is totally irreducible, not virtually cyclic, and contains a hyperbolic toral automorphism. If $A\subset X$ is an $\alpha$-invariant closed subset, then
\begin{itemize} 
\item either $A$ projects to a finite subset of $X_\ab$;
\item or $A=X$.
\end{itemize}
\end{corollary}

In the cases studied by \cites{KS96,EL03,KS05}, in the absence of virtually cyclic factors, when an ergodic invariant measure $\mu$ has positive entropy, some finite index component of $\mu$ is invariant under left translations by some connected closed subgroup $H\subset G$ and $H$ has compact orbits at all points in the support of this component. Whereas, this fibration property fails in general, as the following theorem shows.

\begin{theorem}\label{Heisenberg}There exists a nilmanifold $X$, a $\bZ^2$-action $\alpha$ by automorphisms on $X$ without virtually cyclic algebraic factors, and an $\alpha$-invariant probability measure $\mu$, such that: 
\begin{enumerate}
\item There is an algebraic factor of $\alpha$, which takes place on a torus factor of $X$, such that the factor action is totally irreducible and not virtually cyclic, and the projection of $\mu$  to the factor is Lebesgue;
\item $\mu$ is ergodic under the restriction of $\alpha$ to every finite index subgroup $\Sigma\subset\bZ^2$;
\item For any finite index subgroup $\Sigma$, there does not exist a non-trivial $\alpha|_\Sigma$-invariant connected closed subgroup $H$ of $G$, such that $\mu$ can be desintegrated into $H$-invariant probability measures on compact $H$-orbits.\end{enumerate}
\end{theorem}

\subsection{Background}\label{SecBackground}

There has been a long history of studies of the rigidity of higher rank abelian algebraic actions, which we partially survey here.

In the seminal work \cite{F67} of Furstenberg on the multiplicative semigroup action generated by $\times 2$ and $\times 3$ on $\bR/\bZ$,  it was showed that any invariant closed subset is either a finite set of rational points or $\bR/\bZ$ itself. In the measure category, Furstenberg's conjecture asks if any ergodic invariant probability measure $\mu$ is either supported by a finite invariant set or the Lebesgue measure on $\bR/\bZ$.

While the conjecture remains open, it has been confirmed under the assumption that $\mu$ has positive entropy with respect to an element of the action, by Rudolph \cite{R90}. The complete positive entropy case had been previously settled by Lyons \cite{L88}. Though there has been various strengthenings as well as alternative proofs ( \cites{J92,F93,H95,JR95,P96,H12,HS12}, to list a few) to Rudolph's theorem, so far nothing is yet known in the zero entropy case.   

It is important that the action has rank $2$. Indeed, the action generated by $\times 2$ alone is isomorphic to the Bernoulli shift on $\{0,1\}^\bZ$ and enjoys no rigidity in either measure or subset settings. 

A natural model in higher dimension is a $\bZ^r$-action by toral automorphisms. In this setting, besides higher rank, the known rigidity results also need the positive entropy assumption. Actually, these requirements should hold for every algebraic factor action, since non-standard invariant measures and sets lift from a factor to the original space.

\begin{theorem}\label{KSEL}\cites{KS96,EL03} Suppose $\alpha:\bZ^r\curvearrowright\bT^d$ acts by toral automorphisms and has no virtually cyclic algebraic factors and $\mu$ is an ergodic $\alpha$-invariant probability measure. Then there is a finite index subgroup $\Sigma\subset\bZ^r$, such that for every $\alpha|_\Sigma$-ergodic component $\mu^i$ of $\mu$,  there is an $\alpha|_\Sigma$-invariant subgroup $T^i\subset\bT^d$ satisfying:\begin{itemize}
\item The projection of $\mu^i$ to $\bT^d/T^i$ has zero entropy with respect to the projection of every $\alpha^\bfn$, $\bfn\in\Sigma$;
\item $\mu^i$ is invariant under translation by $T^i$.
\end{itemize}\end{theorem}

In particular, if some ergodic component $\mu^i$ has positive entropy projection for every non-trivial algebraic factor action, then $T=\bT^d$ and $\mu=\rmm_{\bT^d}$.  Theorem \ref{MeasureThm} is a generalization of this fact.

Theorem \ref{KSEL} was first proved by Katok and Spatzier in \cite{KS96} under an additional assumption called totally non-symplecticity (TNS). The general case was obtained by Einsiedler and Lindenstrauss in \cite{EL03}. A detailed proof of this theorem, extended to certain slightly more general settings, is the subject of a forthcoming paper \cite{ELW}. 

The set-theoretic analogue of Theorem \ref{KSEL} is available only when the action is totally irreducible. Berend  showed that 
\begin{theorem}\label{Berend}\cite{B83} Let $\alpha:\bZ^r\curvearrowright\bT^d$ be an action by toral automorphisms. Suppose $\alpha$ is totally irreducible and not virtually cyclic, and contains a hyperbolic element,  then every invariant closed subset is either finite or $\bT^d$. \end{theorem}

Corollary \ref{TopoCor} generalizes Theorem \ref{Berend}.

When the action on $\bT^d$ is not totally irreducible, the situation is more complicated and there can be non-homogeneous invariant closed subsets. See \cites{M10, LW12}.

In \cite{B84} and \cite{EL03}, actions by automorphisms on solenoids, which are non-archimedean cousins of tori, were also studied.

The question of whether under Berend's assumptions every ergodic invariant measure is either atomic or Lebesgue, is an open problem which, through the analogy between tori and solenoids, roughly equivalent to Furstenberg's conjecture.

One moves next from tori to nilmanifolds. Nilmanifold automorphisms are natural genralizations of toral automorphisms. They are also important because, for example, the only known examples of hyperbolic diffeomorphisms of algebraic nature are given by nilmanifold automorphisms modulo finite covering. 

In the nilmanifolds setting, generalizing \cite{KS96}, Kalinin and Spatzier \cite{KS05} established an analogue of Theorem \ref{KSEL} under the following additional assumptions: (1) the action is TNS; (2) the derivative of every $\alpha^\bfn$ is diagonalizable over $\bC$; (3) either every derivative has only real eigenvalues or $\mu$ is mixing.


In this paper, we adapt the approach of Einsiedler and Lindenstrauss \cite{EL03} to establish Theorem \ref{MeasureThm}. Many definitions and lemmas will be borrowed from the nicely written lecture notes \cite{EL10} on Einsiedler-Katok-Lindenstrauss's work \cite{EKL06} towards Littlewood's conjecture.

An important difference between our settings in Theorem \ref{MeasureThm} and those in \cites{KS96,EL03,KS05} is that, under our assumptions an action can still carry virtually cyclic behaviors. A nilmanifold is a tower of bundles where each level is made of torus fibers. An action by automorphisms, after restricting to a finite index subgroup, induces actions by fiber maps on every level. The assumption that there is no virtually cyclic algebraic factor only guarantees that the induced actions on the bottom levels, which make up the maximal torus factor $X_\ab$, are of higher rank. However, on the upper levels, $\alpha$ may still induce virtually cyclic fiber actions. This issue is absent for actions on tori. And in \cite{KS05}, the TNS assumption prevents this from happening.

Due of this phenomenon, a general analogue of Theorem \ref{KSEL} on nilmanifolds, in the form that the measure (or possibly a finite index ergodic component) can be projected to an algebraic factor, such that the factor is dominated by a mixture of zero entropy and virtually cyclic behaviors and the measure is uniform along the fibers, is not possible. This kind of failure is indicated by Theorem \ref{Heisenberg}. 

Finally, we would like to remark that a parallel line of research (see \cite{KN11}*{\S Introduction 3} and references therein and more recent progresses \cites{FKS11,FKS13,RHW}) is the smooth rigidity of higher rank abelian actions by automorphisms, which aims to classify actions by diffeomorphisms that are topologically equivalent to a given action by automorphisms. This question can be thought of as the classification of smooth structures invariant under an action by automorphisms. In this regard, Theorem \ref{MeasureThm} might be viewed as a parallel result to our recent joint work \cite{RHW} with Federico Rodriguez Hertz, in which $C^\infty$ Anosov $\bZ^r$-actions without virtually cyclic factors by nilmanifold automorphisms are smoothly classified.
     
\subsection{Organization of paper} Section \ref{SecPrelim} consists of preliminaries. 

Section \ref{SecInductive} puts Theorem \ref{MeasureThm} as a special case of a more general Theorem \ref{MeasureInductive}, which studies the joining of $\alpha$ with another action. We then give an inductive argument that proves Theorem \ref{MeasureInductive} under a certain assumption (Property \ref{TransInvProp}), which says locally at every point, the measure is invariant by the left translation of some element of $G$. We also deduce topological properties from their measure-theoretic counterparts.

The verification of Property \ref{TransInvProp} is divided into four different cases, which will be taken care of separately by Sections \ref{SecBerendFiber}--\ref{SecEnt0Fiber}.

It will be showed in Section \ref{SecRatner} that Property \ref{TransInvProp} is indeed sufficient to complete the inductive argument.

Finally, in Section \ref{SecHeisen}, a non-homogeneous example as in Theorem \ref{Heisenberg} will be constructed on the 13-dimensional Heisenberg nilmanifold.

\begin{acknowledgments}I would like to thank Danijela Damjanovi\'c for suggesting the problem and stimulating conversations. 

An essential portion of this paper is based on Manfred Einsiedler and Elon Lindenstrauss's proof of Theorem \ref{KSEL}. I am deeply indebted to them for explaining their approach to me in details, as well as for guiding me through homogeneous dynamics over the years. 

Thanks are also due to Federico Rodriguez Hertz and Andrew Torok for helpful discussions.

The author was supported by NSF grant DMS-1201453 and an AMS-Simons travel grant.\end{acknowledgments}

\section{Preliminaries}\label{SecPrelim}

\subsection{Fibration of nilmanifolds}\label{SecNilfiber}
\begin{definition} Given a simply connected nilpotent Lie group $G$ and a cocompact lattice $\Gamma\subset G$, a {\bf rational subgroup} is a connected closed subgroup $H\subset G$ such that $H\cap\Gamma$ is a cocompact lattice in $H$.\end{definition}

It is a well known fact that $\exp^{-1}(\Gamma)$ is a lattice in the lie algebra $\gog$. Furthermore, a connected closed subgroup $H$ is rational if and only if its Lie algebra $\goh$ satisfies the condition that $\goh\cap\exp^{-1}(\Gamma)$ is a lattice in $\goh$. Equivalently, $\exp^{-1}(\Gamma)$ determines a $\bQ$-structure on $\gog$, and $\goh$ is a rational subspace with respect to this structure, see e.g. \cite{R72}. 

\begin{remark}\label{RationalFiberRmk}Suppose $H\lhd G$ is a rational normal subgroup. It is easy to show that the image $\Gamma/(H\cap\Gamma)$ of $\Gamma$ in the quotient group $G/H$ is a cocompact lattice. The compact nilmanifold $G/\Gamma$ decomposes as a bundle whose base is the nilmanifold $G/H\Gamma=(G/H)\Big/\big(\Gamma/(H\cap\Gamma)\big)$ and whose fibers are isomorphic to the nilmanifold $H/(H\cap\Gamma)$.\end{remark}

In this case, we will call $G/H\Gamma$ an {\bf algebraic factor} of $X$ and denote it by $X/H$.

Let $\alpha:\bZ^r\curvearrowright X$ be a $\bZ^r$-action by automorphisms of $X=G/\Gamma$, and identify $\alpha$ with its lift, which is an action on $G$ by group automorphisms. Then the lifted action $\alpha$ preserves $\Gamma$. Therefore, if $H$ is an $\alpha$-invariant rational subgroup, then $H\cap\Gamma$ is also $\alpha$-invariant. And hence $\alpha$ induces a natural actions on $H/(H\cap\Gamma)$. Moreover, if $H$ is normal, then $\alpha$ factors onto a $\bZ^r$-action on the quotient nilmanifold $X/H$. All such actions will be denoted by $\alpha$ indifferently in this paper.

\begin{definition}An {\bf algebraic factor} of $\alpha$ is the induced action by $\alpha$ on some $X/H$ where $H$ is an $\alpha$-invariant normal rational subgroup in $G$. We say this factor action is a {\bf torus factor action} if $X/H$ is a torus, or equivalently $G/H$ is abelian.\end{definition}

We also say that $X/H$ is an {\bf $\alpha$-equivariant algebraic factor} of $X$ in this case.

Before stating the following lemma, we recall that any finite index subgroup of $\bZ^r$ is isomorphic to $\bZ^r$ itself.

\begin{lemma}\label{IrrVertical}If $\alpha$ is a $\bZ^r$-action on a compact nilmanifold $X=G/\Gamma$ by automorphisms,  then there is a non-trivial abelian subgroup $Z$ of  $G$, such that:
\begin{enumerate}
\item $Z$ is in the center of in $G$, rational with respect to $\Gamma$, and invariant under the restriction of $\alpha$ to a finite index subgroup $\Sigma\subset\bZ^r$;
\item The induced and restricted action $\alpha|_\Sigma: \Sigma\curvearrowright Z/(Z\cap\Gamma)$ is totally irreducible.
\end{enumerate} 
\end{lemma}

\begin{proof} The center $G_0$ of $G$ is rational with respect to $\Gamma$ by \cite{R72}*{Prop. 2.17}, and $\alpha$-invariant because $\alpha$ acts on $G$ by automorphisms. 

If the torus $G_0/(G_0\cap \Gamma)$ is totally irreducible under $\alpha$ then we are done.

Otherwise, there must be a minimal non-trivial subtorus whose stabilizer under $\alpha$ has finite index in $\bZ^r$. The subtorus can be written as $Z/(Z\cap \Gamma)$ and we denote its stabilizer by $\Sigma$. Then $Z$ is $\alpha|_\Sigma$-invariant and rational. It is also normal as it is a subgroup of $G_0$. Moreover, the quotient torus $Z/(Z\cap \Gamma)$ is totally irreducible under $\alpha|_\Sigma$ by the minimality assumption.\end{proof}

\begin{remark}\label{EquivariantTower}Repeated applications of the Lemma shows that there exist a finite index subgroup $\Sigma$ and sequence of $\alpha|_\Sigma$-invariant normal rational subgroups $\{e\}=H_0\subsetneq H_1\subsetneq H_2\subsetneq\cdots\subsetneq H_k=G$, such $H_{i+1}/H_i$ is abelian for all $i$ and the induced $\Sigma$-action $\alpha|_\Sigma$ on $H_{i+1}/H_i(H_{i+1}\cap\Gamma)$ is totally irreducible. This gives rise to an $\alpha|_\Sigma$-equivariant fibration of $M$ as a tower of torus bundle extensions $M\mapsto M/H_1\mapsto\cdots M/H_k=\{\text{point}\}$, where each $M_i$ is a principle torus bundle on $M_k$ with fibers isomorphic to $H_{i+1}/H_i(H_{i+1}\cap\Gamma)$.\end{remark}

In the present paper we will primarily deal with $\bZ^r$-actions without virtually cyclic factors by nilmanifold automorphisms. The following lemma is another characterization of this property.  
\begin{lemma}\label{Rk2Equiv} Suppose $\alpha$ is a $\bZ^r$-action on a nilmanifold $X$ by automorphisms. Then the following are equivalent:
\begin{enumerate}
\item $\alpha$ has no virtually cyclic algebraic factor;
\item There is a subgroup $\Lambda\subset\bZ^r$ isomorphic to $\bZ^2$, such that for every non-trivial element $\bfn\in\Lambda$, $\alpha^\bfn$ is ergodic with respect to the standard volume $\rmm_X$.
\end{enumerate}
\end{lemma}

This is a well-known fact essentially due to Starkov \cite{S99}. See \cite{RHW}*{Lemma 2.9} for more explanation. 

\begin{remark}\label{Rk2ResRmk}If $\Sigma$ is of finite index in $\bZ^r$, one see from the lemma that  $\alpha$ has no virtually cyclic algebraic factor if and only if $\alpha|_\Sigma$ has no such factor.\end{remark}

\begin{corollary}\label{ProdErgo}If $\alpha$ is as in Lemma \ref{Rk2Equiv}, and $\beta:\bZ^r\curvearrowright (Y,\nu)$ is an arbitrary ergodic measure-preserving action, then:
\begin{enumerate}
\item $\rmm_X\times\nu$ is ergodic under the product $\bZ^r$-action $\talpha=\alpha\times\beta$ on $\tX=X\times Y$;
\item For any subgroup $\Sigma\subset\bZ^r$ of finite index, the ergodic decomposition of $\rmm_X\times\nu$ with respect to $\talpha|_\Sigma$ can be written as $\frac1N\sum_{i=1}^N\rmm_X\times\nu^i$, where the $\nu^i$'s are the $\beta|_\Sigma$-ergodic components of $\nu$.
\end{enumerate}
\end{corollary} 

\begin{proof}(1) By Lemma \ref{Rk2Equiv}, $\alpha^\bfn$ is an ergodic nilmanifold automorphism for some $\bfn$. It was proved by Parry \cite{P69} that ergodic nilmanifold automorphisms are mixing, and in particular, weakly mixing. 

Take the ergodic decomposition $\nu=\int\nu_y^\cE\di\nu(y)$ with respect to $\beta^\bfn$, where $\cE$ is the $\sigma$-algebra of $\beta^\bfn$-invariant sets. Then for each $y$, $\rmm_X\times\nu_y^\cE$ is ergodic with respect to $\alpha^\bfn\times \beta^\bfn$. This implies that any $\alpha\times\beta$-invariant set is $\pi_Y^{-1}(\cE)$-measurable modulo $(\rmm_X\times\nu)$-null sets. Since $\nu$ is $\beta$-ergodic, it follows that $\rmm_X\times\nu=\int(\rmm_X\times\nu_y^\cE)\di\nu(y)$ is $\talpha$-ergodic.

(2) Since $\nu$ is ergodic under $\beta$,  it has a finite ergodic decomposition $\nu=\frac1N\sum_{i=1}^N\nu^i$ with respect to $\beta|_\Sigma$. By Lemma \ref{Rk2Equiv}, the property of having no virtually cyclic factors passes from the action $\alpha$ to the restriction $\alpha|_\Sigma$. Therefore, by part (1), each $\rmm_X\times\nu^i$ is $\talpha|_\Sigma$-ergodic. Hence $\rmm_X\times\nu=\frac1N\sum_{i=1}^N(\rmm_X\times \nu^i)$ is the ergodic decomposition we want. \end{proof}

\subsection{Common eigenspaces under the action}\label{ComEigenSec}

Given a $\bZ^r$-action $\alpha$ by automorphisms of $X$, the differential $D\alpha$ gives a $\bZ^r$ action on the Lie algebra $\gog$ of $G$ by Lie algebra automorphisms. Abusing notation again, we will denote this action by $\alpha$ as well.

Recall that for all $\bfn$ and $w\in\gog$, $\alpha^\bfn.\exp w=\exp(\alpha^\bfn w)$ as $\alpha^\bfn$ is an automorphism.

Since the action is commutative, the Lie algebra automorphisms $\alpha^\bfn$ can all be upper triangularized simultaneously. More precisely, $\gog\otimes_\bR\bC$ can be decomposed as a direct sum $\bigoplus_kW_k$, where for all $W_k$ and $\alpha^\bfn$, $W_k$ is a generalized eigenspace of $\alpha^\bfn$, the eigenvalue of which is denoted by $\zeta_k^\bfn$ (see for instance \cite{RHW}*{\S 2.3}). The map $\bfn\mapsto\zeta_k^\bfn$ is a group morphism from $\bZ^r$ to $\bC^\times$.

Suppose a connected closed subgroup $H$ is $\alpha$-invariant, then its Lie algbera $\goh$ is also $\alpha$-invariant and $\goh\otimes_\bR\bC=\bigoplus_k\big((\goh\otimes_\bR\bC)\cap W_k\big)$.

\begin{remark}\label{QuotientEigenRmk}If $H_1\subsetneq H_2\subsetneq\cdots\subsetneq H_m=H$ is a sequence of $\alpha$-invariant connected closed subgroup $H$. Then $\goh_i$ are all $\alpha$-invariant, and $\alpha$ induces a linear $\bZ^r$-action on all the quotients $\goh_i/\goh_{i-1}$. Again, $\goh_i/\goh_{i-1}$ can be decomposed into generalized eigenspaces. It is a simple fact of linear algebra that, the eigenvalue functionals $\bfn\mapsto\zeta_k^\bfn$ of $\alpha$ on $\goh$ are made up by all eigenvalue functionals of $\alpha$ on $(\goh_1)\otimes_\bR\bC, (\goh_2/\goh_1)\otimes_\bR\bC,\cdots,(\goh/\goh_{m-1})\otimes_\bR\bC$, with multiplicities counted.\end{remark}

Therefore, using Remark \ref{EquivariantTower} to describe $X$ as a tower of iterated torus bundles, where each layer of tori fibers supports a totally irreducible action induced by $\alpha$, finding eigenvalue functionals for $\alpha:\bZ^r\curvearrowright X$ can be reduced to the the simpler task of finding eigenvalues of totally irreducible toral actions. This is the content of the next lemma.

\begin{lemma}\label{IrrField} Suppose $\alpha$ is a totally irreducible $\bZ^r$-action on the torus $\bT^d$ by automorphisms. Then: \begin{enumerate}
\item The complexified tangent space $\bC^d=\bR^d\otimes_\bR\bC$ splits into $d$ eigenspaces;
\item There exists a number field $K$ of degree $d$ and a group morphism $\zeta$ from $\bZ^r$ to the group of units $\cO_K^\times$ of $K$, such that the eigenvalue functionals of the action $\alpha$ are $\zeta_i=\sigma_i\circ\zeta$, $i=1,\cdots, d$. Here $\sigma_1,\cdots, \sigma_d$ are all the archimedean embeddings of $K$.
\end{enumerate}
\end{lemma}

In particular, in this case $\alpha^\bfn$ is an automorphism of finite order if and only if $\zeta^\bfn$ is a root of unity.

\begin{proof} The proposition is a special case of a well-known fact, a general form of which can be found in \cite{S95}*{\S 7\&29} and \cite{EL03}*{Prop 2.1}. 
\end{proof}

Together with Remark \ref{EquivariantTower}, Lemma \ref{IrrField} shows that given a $\bZ^r$ action by nilmanifold $G/\Gamma$, every eigenvalue functional of $\alpha:\bZ^r\curvearrowright\gog$ is a group morphism from $\bZ^r$ to the group of units of an archimeadeanly embedded number field.

\subsection{Coarse Lyapunov decomposition} Given a $\bZ^r$-action $\alpha$ by automorphisms on $X$, the tangent space $\gog$ of $X$ can be decomposed as a direct sum of different coarse Lyapunov subspaces. Two non-zero tangent vectors belong to the same coarse Lyapunov subspace component if and if they cannot be distinguished by the action, in the sense that no element of the action expands one of the vectors while contracting the other in the long run. The existence and some properties of this decomposition are collected in the proposition and definition below.

\begin{proposition}\label{LyaDecomp} If $\alpha$ is a $\bZ^r$-action on a nilmanifold $X=G/\Gamma$ by automorphisms. Then there exists finitely linear functionals $\chi\in(\bR^r)^*$, to each of which is associated a non-trivial subspace $\gov^\chi$ of the Lie algebra $\gog$ of $G$, such that: \begin{enumerate}
\item $\gog=\bigoplus\gov^\chi$;
\item For all $v\in\gov^\chi\backslash\{0\}$,
$$\lim_{|\bfn|\rightarrow\infty}\frac{\log\big|\alpha^\bfn v\big|-\chi(\bfn)}{|\bfn|}=0,$$
\item If $[\gov^\chi,\gov^{\chi'}]$ is not empty, then it is contained in $\gov^{\chi+\chi'}$.
\end{enumerate}
\end{proposition}

These facts can be proved by simple arguments from linear algebra. In fact, given the generalized eigenspace decomposition $\gog=\bigoplus_kW_k$ from \S\ref{ComEigenSec}.  The logarithm map $\bfn\mapsto\log|\zeta_k^\bfn|$ is a linear map from $\bZ^r$ to $\bR$. The subspace $\gov^\chi$ is given by
\begin{equation}\label{LyaEigenEq}\gov^\chi=\bigoplus_{\{k\ :\ \log|\zeta_k^\bfn|=\chi(\bfn), \forall\bfn\}}W_k.\end{equation} For details, see for instance \cite{RHW}.

The functionals $\chi$ and the subspaces $\gov^\chi$ are respectively called {\bf Lyapunov exponents} and {\bf Lyapunov subspaces}. It should be stressed that $\gov^\chi$ is in general not a Lie subalgebra. However, if one take the direct sum of all Lyapunov subspaces whose Lyapunov exponents are positively proportional to $\chi$, then these exponents form a semigroup under addition. And in particular, by  part (3) of the proposition, the direct sum is closed under Lie bracket and hence a Lie subalgebra.

\begin{definition}\label{CoarseLya}In the Lyapunov decomposition above,

\begin{enumerate}\item For a Lyapunov exponent $\chi$ that is present in the Lyapunov decomposition, the corresponding {\bf coarse Lyapunov subspace} is the Lie subalgebra
$$\gov^{[\chi]}=\bigoplus_{\chi'=c\chi, c>0}\gov^{\chi'};$$
\item To each coarse Lyapunov subspace $\gov^{[\chi]}$ is associated a closed connected subgroup $V^{[\chi]}=\exp\gov^{[\chi]}\subset G$, which will be called a {\bf coarse Lyapunov subgroup}.
\end{enumerate}\end{definition}

In particular, there is a finite set $\cL$ of homothety equivalence classes $[\chi]$, such that \begin{equation}\label{CoarseLyaDecomp}\gog=\bigoplus_{[\chi]\in\cL}\gov^{[\chi]}.\end{equation}

For any element $\bfn\in\bZ^r$, set \begin{equation}\gog_\bfn^u=\bigoplus_{\{\chi\in\cL, \chi(\bfn)>0\}}\gog^{[\chi]}, \gog_\bfn^s=\bigoplus_{\{\chi\in\cL, \chi(\bfn)<0\}}\gog^{[\chi]}.\end{equation}$\alpha^\bfn$ exponentially expands vectors in the unstable subspace $\gog_\bfn^u$ and exponentially contracts the stable subspace $\gog_\bfn^s$.  Both $\gog_\bfn^u$ and $\gog_\bfn^s$ are Lie subalgebras. Denote the corresponding Lie subgroups by $G_\bfn^u$ and $G_\bfn^s$.

\begin{lemma}\label{LyaFree}Both $G_\bfn^u$ and $G_\bfn^s$ act freely by left translations on $X$.\end{lemma}
\begin{proof}We prove the claim for $G_\bfn^s$. The other half follows by switching $\bfn$ and $-\bfn$.

We claim that there is a constant $a>0$ depending only on $G$ and $\Gamma$, such that for any $g\in\Gamma$, $\min_{\gamma\in\Gamma\backslash\{e\}}d(g\gamma g^{-1},e)>a$. To see this, observe that it suffices to show the inequality for all the $g$'s in a fundamental domain $F$ of $G/\Gamma$, which we can take to be precompact. For $g\in F$, the map $h\mapsto g^{-1}hg$ is uniformly bi-Lipschitz on a given neighborhood $B_1^G$ of $e\in G$. Hence $\bigcup_{g\in F}g^{-1}B_1^Gg$ is bounded and therefore contains only a finite subset $\Lambda\subset\Gamma$. Thus only when $\gamma\in\Lambda$, there could be some $g\in F$ such that $g\gamma g^{-1}$ falls in $B_1^G$. Again by uniform bi-Lipschitz continuity, $\bigcup_{g\in F}\bigcup_{\gamma\in\Lambda\backslash\{e\}}g\gamma g^{-1}$ is uniformly bounded away from the identity. This proves the claim.

For $v\in G_\bfn^s$, $\alpha^{k\bfn}v$ is contracted towards the identity exponentially fast as $k\rightarrow\infty$. But $\alpha^{k\bfn}v\in\alpha^{k\bfn}(g\Gamma g^{-1})=(\alpha^{k\bfn}g)\Gamma(\alpha^{k\bfn}g)^{-1}$ since $\Gamma$ is an $\alpha$-invariant lattice, and thus, by the claim above, $d(\alpha^{k\bfn}v,e)>a$, which is a contradiction. The lemma is hence proved.\end{proof}

The next fact relates the position of Lyapunov exponents to the rank of the action.
\begin{lemma}Suppose $\alpha$ is a totally irreducible $\bZ^r$-action on the torus $\bT^d$ by automorphisms. Then the action is virtually cyclic if and only if all Lyapunov exponents are proportional to each other.\end{lemma}
\begin{proof} A virtually cyclic action is generated by finitely many automorphisms of finite order and possibly another automorphism $\alpha^\bfn$. The Lyapunov exponents of the finite order elements are always identically zero. So every Lyapunov exponent $\chi$ is completely decided by $\chi(\bfn)$. Which implies all the $\chi$'s are proportional as linear functionals.

Conversely suppose all Lypunov exponents are proportional. That is, for any elements $\bfn,\bfm\in\bZ^r$ and Lyapunov exponents $\chi. \chi'$, $\frac{\chi(\bfn)}{\chi'(\bfn)}=\frac{\chi(\bfm)}{\chi'(\bfm)}$, or equivalently $\chi(\bfn)=c\chi(\bfm)$ for a constant $c$ that depends only on $\bfn$ and $\bfm$ but not on $\chi$ (one may assume $c\neq\infty$ by switch $\bfn$ and $\bfm$ if necessary).

Thus for any $\epsilon>0$, one can always find a sufficiently close approximation $c\approx\frac pq$, $p,q\in\bZ$ such that $|p\chi(\bfn)-q\chi(\bfm)|<\epsilon$. As there are only finitely many $\chi$, there is a pair $(p,q)$ making this true for all of them. Then the element $\tilde\bfn=p\bfn-q\bfm$ satisfies $|\chi(\tilde\bfn)|<\epsilon$ for all $\chi$. This is equivalent to that, in Lemma \ref{IrrField},  each algebraic conjugate of the corresponding eigenvalue $\zeta^{\tilde\bfn}\in K$ has absolute value between $e^{-\epsilon}$ and $e^\epsilon$. However, by Dirichlet's Unit Theorem, if $\epsilon$ is chosen to be small enough, then this implies that $\zeta^{\tilde\bfn}$ is a root of unity. 

By Lemma \ref{IrrField}, $\alpha^{\tilde\bfn}$ is of finite order.  Furthermore, this shows that the group of automorphisms generated by $\alpha^\bfn$ and $\alpha^\bfm$ is virtually cyclic. Since this is true for all pairs $(\bfn,\bfm)$, the entire acting group $\{\alpha^\bfn: \bfn\in\bZ^r\}$ is virtually cyclic. The lemma is proved.\end{proof}

\begin{corollary}\label{NonPropLya}Suppose $\alpha$ is a $\bZ^r$-action on a nilmanifold $X$ by automorphisms such that, for any finite index subgroup $\Sigma\subset\bZ^r$ and any $\alpha|_\Sigma$-equivariant algebraic factor $X_0=X/H$, the induced action by $\alpha$ on $X_0$ is not virtually cyclic. Then in the Lyapunov decomposition on $X$, not all Lyapunov exponents are proportional to each other.\end{corollary}
\begin{proof}As Remark \ref{EquivariantTower} explains, there is always a finite index subgroup $\Sigma$, such that one can find an $\alpha|_\Sigma$-equivariant torus factor $X_0$ on which $\alpha|_\Sigma$ acts totally irreducibly. By Remark \ref{QuotientEigenRmk} and \eqref{LyaEigenEq}, all Lyapunov exponents on $X_0$ are Lyapunov exponents on $X$ as well. The corollary follows immediately from the previous lemma.\end{proof}

\subsection{Leafwise measures and entropies} Take a $\bZ^r$-action $\alpha$ on $X$ by nilmanifold automorphisms. Suppose a closed connected group $V\subset G$ acts freely on $X$ by left translations, i.e. $\Stab_V(x)=\{e\}$ for all $x\in X$.

Take another continuous $\bZ^r$-action $\beta$ on a compact metric space $Y$. Let $\tX=X\times Y$ and $\tilde \alpha=\alpha\times \beta$ be the product action on $\tX$.  Denote by $\cB_Y$ the Borel $\sigma$-algebra of $Y$ and regard it as a subalgebra of $\cB_{\tX}$, the Borel $\sigma$-algebra of $\tX$. 

$V$ acts freely on $\tX$ by translation on the $X$-component. Let $\mu$ be a Borel probability measure on $\tX$. One can define leafwise measures along the $H$-orbits. These measures were previously used in various homogeneous dynamical settings (for instance \cites{KS96,  EK03, EL03, KS05, L06}); detailed explanations can be found in \cite{EL10}.

Denote by $\cM(V)$ the space of positive Radon (i.e. locally finite) measures on $V$, and define an equivalence relation on $\cM(V)$ by $\nu\simeq\nu'$ if and only if $\nu=c\nu'$ for some $c>0$. Let $\cM_1(V)$ be the quotient space $\cM(V)/\simeq$ of normalized measures.
If one further pose certain mild growth conditions, then one gets a compact metric space $\cM_1^0(V)\subset\cM_1(V)$. We refer the reader to \cite{EL10}*{8.4} for the norm that defines $\cM_1^0(V)$.

From now on, $B_\delta^H=\exp\{v\in\goh, |v|\leq\delta\}\subset H$ will denote be the compact ball of radius $\delta$ around identity in a Lie group $H$.

\begin{definition}\label{Subord}A Borel $\sigma$-algebra $\cA$ of $\tX$ is {\bf $V$-subordinate} if $\cA$ is countably generated and for $\mu$-a.e. $\tx\in\tX$, there exists $\delta>0$ such that the atom $[\tx]_\cA$ satisfies  $B_\delta^V.\tx\subset [\tx]_\cA\subset B_{\delta^{-1}}^V.\tx$.

Given a continuous map $T:\tX\mapsto\tX$, $\cA$ is {\bf $T$-increasing} if $T\cA\subset\cA$.\end{definition}

Recall that the conditional measure $\tmu_\tx^\cA$ is the uniquely defined measure system (up to a null measure set) such that for all $f\in L^1(\mu)$ and $\tmu$-a.e. $\tx$, $\bE_\tmu(f|\cA)(\tx)=\int f d\tmu_\tx^\cA$.

\begin{proposition}\label{LeafMeas}\cite{EL10}*{Thm. 6.3} Suppose $V$ acts freely by left translations on $\tX$, and $\mu$ is a probability measure on $\tX$. Then there is a set $\tX^*\subset\tX$ of full $\tmu$-measurea, and a Borel measurable map $\tx\mapsto\tmu_\tx^V$ from $\tX^*$ to $\cM_1^0(V)$, such that:
\begin{enumerate}
\item For $\tmu$-a.e. $\tx$, $e$ is in the support of $\tmu_\tx^V$;
\item If $v\in V$ and $\tx, v.\tx\in\tX^*$, then $\tmu_\tx^V$ is proportional to $(\tau_v)_*\tmu_{v.\tx}^V$, the pushforward of $\tmu_{v.\tx}^V$ under the right translation by $v$.
\item If $\cA$ is $V$-subordinate then $(\tmu_\tx^V).\tx$, the pushforward of  $\tmu_\tx^V$ under the map $h\mapsto h.\tx$,  is proportional to $\tmu_\tx^{\cA}$ on $[\tx]_\cA$.
\end{enumerate}
\end{proposition}

\begin{lemma}\label{LeafAction} Suppose $V$ is an $\alpha$-invariant subgroup of $G_\bfn^u$ for some $\bfn$, and $\tmu$ is an $\talpha$-invariant probability measure. Then $$\alpha^\bfm_*\tmu_\tx^V\simeq\tmu_{\talpha^\bfm.\tx}^V, \text{ for }\tmu\text{-a.e. }\tx, \forall\bfm\in\bZ^r.$$\end{lemma}

\begin{proof}
Take $\cA$ to be an $\alpha^\bfn$-increasing, $V$-subordinate  $\sigma$-algebra. Einsiedler and Lindenstrauss \cite{EL03} proved there do exist such nice $\sigma$-algebras, which we will explain later in \S\ref{SecSubord}.

Then $\talpha^{k\bfn}\cA$ is also $V$-subordinate. By Proposition \ref{LeafMeas}, via the identification $h\mapsto h.\tx$, $\tmu_\tx^V$ is proportional to $\tmu_\tx^{\talpha^{k\bfn}\cA}$ on $[\tx]_{\talpha^{k\bfn}\cA}$ for $\tmu$-a.e. $\tx\in\tX$. As $\tmu$ is $\talpha$-invariant, $\tmu_\tx^{\talpha^{k\bfn}\cA}=\talpha^{-\bfn}_*\tmu_{\talpha^\bfn.\tx}^{\talpha^{(k+1)\bfn}\cA}$.  As $\cA$ is $\alpha^\bfn$-increasing, it follows that $\tmu_\tx^V$ is proportional to $\alpha^{-\bfn}_*\tmu_{\talpha^\bfn.\tx}^V$ on $[\tx]_{\talpha^{k\bfn}\cA}$ for any $k>0$. 

As $\cA$ is $V$-subordinate and $V\subset G_\bfn^u$, for almost every $\tx$, there is $k\in\bN$ such that $[\tx]_{\talpha^{k\bfn}\cA}=\talpha^{k\bfn}.[\talpha^{-k\bfn}.\tx]_\cA$ contains arbitrarily large neigborhoods in $V.\tx$. To see this, observe that for any $\epsilon$ there is a subset $\Omega\subset\tX$ of measure $\tmu(\Omega)>1-\epsilon$, such that for all $\tx\in\Omega$, $B_\delta^V.\tx\subset [\tx]_\cA$ for some uniform $\delta>0$.To prove the claim for almost every $\tx\in\Omega$, it sufficient to take values of $k$ for which $\talpha^{-k\bfn}.\tx$ recurs to $\Omega$. By letting $\epsilon$ approach $0$, we see the claim is actually true for almost every $\tx\in\tX$.

So we have proved that $\alpha^\bfn_*\tmu_\tx^V\simeq\tmu_{\talpha^\bfn.\tx}^V$ for almost all points.

For any $\bfm\in\bZ^r$, one can always find a large $s$ such that $\chi(\bfm)+s\chi(\bfn)<0$ for all Lyapunov exponents $\chi$ from $G_\bfn^u$. Then the element $\bfn'$ also has negative Lyapunov exponents on $G_\bfn^u$, so $V\subset G_\bfn^u\subset G_{\bfn'}^u$. By the same proof,   $\alpha^{\bfn'}_*\tmu_\tx^V\simeq\tmu_{\talpha^{\bfn'}.\tx}^V$ almost everywhere. By writing $\talpha^\bfm$ as $\talpha^{\bfn'}\circ\talpha^{-s\bfn}$, the almost everywhere equivariance property for $\bfm$ follows from those for $\bfn$ and $\bfn'$.
\end{proof}

One of the main advantage of studying leafwise measures is that they charaterize quantitatively the source of entropies. In general smooth dynamics, this is the  Ledrappier-Young formula \cite{LY85}. In a homogeneous setup, things can be made more precise.

\begin{lemma}\label{EntContriAtX} For $\bfn\in\bZ^r$, an $\alpha$-invariant closed subgroup $V\subset G_\bfn^u$, and any $\talpha$-invariant probability measure $\tmu$ on $\tX$,
\begin{enumerate}
\item The limit $$D_\tmu(\talpha^\bfn,V)(\tx)=\lim_{k\rightarrow\infty}-\frac1k\log\tmu_{\tx}^V(\talpha^{-k\bfn}.B_1^V)$$
exists for $\tmu$-a.e. $\tx$ and is an $\talpha^\bfn$-invariant function;
\item For $\tmu$-a.e. $\tx$, $D_\tmu(\talpha^\bfn,V)(\tx)=0$ if and only if $\tmu_\tx^V$ is trivial.
\item For $\tmu$-a.e. $\tx$, $D_\tmu(\talpha^\bfn,V)(\tx)\leq h_{\tmu_\tx^\cE}(\talpha^\bfn|\cB_Y)$, where $\tmu_\tx^\cE$ stands for the $\alpha^\bfn$-ergodic component $\tmu_\tx^\cE$ of $\tmu$ at $\tx$. The equality holds when $V=G_\bfn^u$.
\end{enumerate}
\end{lemma}
 
 This is the conditional version of \cite{EL10}*{Thm. 7.6} and the proof is identical. 

\begin{definition}\label{EntContriDef}When $\tmu$, $\bfn$ and $V$ are as in Lemma \ref{EntContriAtX}, the {\bf entropy contibution} of $V$ is $$h_\tmu(\talpha^\bfn,V)=\int_\tX D_\tmu(\talpha^\bfn,V)(\tx)\di\tmu(\tx).$$ \end{definition}

Notice that, $D_\tmu(\talpha^\bfn, V)$ is also $\talpha^\bfm$-invariant for all $\bfm$, because of Lemma \ref{LeafAction} and the fact that $\alpha^\bfm$ commutes with $\alpha^\bfn$. So if $\tmu$ is $\talpha$-ergodic then $D_\tmu(\talpha^\bfn,V)(\tx)$ is constant and equal to $h_\tmu(\talpha^\bfn,V)$ almost everywhere.

\begin{proposition}\label{EntContri}$h_\tmu(\talpha^\bfn,V)$ has the following properties:
\begin{enumerate}
\item $h_\tmu(\talpha^\bfn,V)=0$ if and only if $\tmu_\tx^V$ is trivial almost everywhere;
\item $h_\tmu(\talpha^\bfn,V)\leq h_\tmu(\talpha^\bfn|\cB_Y)$. The equality holds when $V=G_\bfn^u$;
\item $h_\tmu(\talpha^\bfn,V)\leq\log\big|\det\alpha^\bfn|_\gov\big|$, where $\gov$ is the Lie algebra of $V$. The equality holds if and only if $\tmu$ is invariant under left translations by $V$.
\end{enumerate}
\end{proposition}
\begin{proof}By Proposition 7.22 of \cite{EL10}, which can be transplanted here without modification, $(\tmu_\tx^\cE)_\ty^V=\tmu_\ty^V$ for $\tmu$-almost every $\tx$ and $\tmu_\tx^\cE$-almosty every $y$. It follows that \begin{equation}D_\tmu(\talpha^\bfn,V)(\ty)=\int_\tX D_{\tmu_\tx^\cE}(\talpha^\bfn,V)(\ty)\di\tmu(\tx).\end{equation} And therefore,
\begin{equation}\label{EntContriEq1}\begin{split}
h_\tmu(\talpha^\bfn,V)
=&\int_\tX D_\tmu(\talpha^\bfn,V)(\ty)\di\tmu(\ty)\\
=&\int_{\tX\times\tX} D_{\tmu_\tx^\cE}(\talpha^\bfn,V)(\ty)\di(\tmu\times\tmu)(\tx,\ty)\\
=&\int_\tX h_{\tmu_\tx^\cE}(\talpha^\bfn,V)\di\tmu(\tx)
\end{split}\end{equation} 

Given this, part (1) is a direct corollary of Lemma \ref{EntContriAtX}.

Conditional entropies are known to spread linearly in ergodic decompositions: \begin{equation}\label{EntLinearEq}h_\tmu(\talpha^\bfn|\cB_Y)=\int_\tX h_{\tmu_\tx^\cE}(\talpha^\bfn|\cB_Y)\di\tmu(\tx).\end{equation} So we obtain part (2) by integrating Lemma \ref{EntContriAtX}.(3).

When $\tmu$ is $\alpha^\bfn$-ergodic, the last part of the proposition is an adaption of \cite{EL10}*{Theorem 7.9}, whose proof works verbatim in the current setting. In the general case, it suffices to apply \eqref{EntContriEq1}.\end{proof}

\begin{lemma}\label{EntPosForAll} Let $V\subset G$ be an $\alpha$-invariant closed subgroup. If $V$ is contained in both $G_\bfn^u$ and $G_\bfm^u$ for two different elements $\bfn,\bfm\in\bZ^r$, then $h_\tmu(\talpha^\bfm, V)>0$ if and only if $h_\tmu(\talpha^\bfn, V)>0$.\end{lemma}
\begin{proof}As the claim is symmetric, we treat the ``if'' direction here. It suffices to prove the almost everywhere property that $D_\tmu(\talpha^\bfm,V)(\tx)\geq\frac1s D_\tmu(\talpha^\bfn,V)(\tx)$ for some constant $s\in\bN$. For this purpose, one only needs that for all sufficiently large $k$, \begin{equation}\label{EntPosForAllEq1}\alpha^{-sk\bfm}.B_1^V\subset \alpha^{-k\bfn}.B_1^V.\end{equation}

Let $\theta=\chi(\bfm)$ and $\theta'=\chi'(\bfn)$ respectively be the smallest Lyapunov exponent of $\alpha^\bfm$ and the largest one of $\alpha^\bfn$ on $\gov$. Then both are strictly positive. Then for all $\epsilon>0$, $\alpha^\bfm$ contracts all vectors of $\gov$ at an exponential rate faster than $e^{-(\theta-\epsilon)k}$, and $\alpha^\bfn$ contracts at a slower rate than $e^{-(\theta'+\epsilon)k}$. Then there is some constant $C$ such that  $\alpha^{-sk\bfm}.B_1^V\subset B_{Ce^{-s(\theta-\epsilon)k}}^V$ and $\alpha^{-k\bfn}.B_1^V\supset B_{C^{-1}e^{-(\theta'+\epsilon)k}}^V$. In order to make \eqref{EntPosForAllEq1} valid for large values of $k$,  one only has to fix $\epsilon<\theta$ and some $s\in\bN$ such that $s(\theta-\epsilon)>\theta'+\epsilon$.  \end{proof}

\subsection{Subordinate $\sigma$-algebras}\label{SecSubord}
In this part we explain a proof, which was outlined by Einsiedler and Lindenstrauss in \cite{EL03}, of the fact that there do exist $V$-subordinate $\sigma$-algebras. 

Let $\tX$, $\talpha$, $G$ and $V$ be as in Section \ref{SecPrelim}. For a finite partition $\cP$ is a compact metric space $\Omega$ into Borel measurable sets, let $\partial_\epsilon\cP=\{x\in \Omega: B_\epsilon(x)\not\subset[\tx]_\cP\}$ for all $\epsilon>0$. In the case where $\Omega=\tX$, we define $\partial_\epsilon^V\cP=\{\tx\in \tX: B_\epsilon^V.\tx\not\subset[\tx]_\cP\}\subset\partial_\epsilon\cP$.

\begin{lemma}\label{SubordExist}For $\bfn\in\bZ^r$, an $\talpha^\bfn$-invariant probability measure $\tmu$ on $\tX$, an $\alpha^\bfn$-invariant closed subgroup $V\subset G_\bfn^u$ and any $\delta>0$, there exist on $\tX$ a finite partition $\cP$ of $\tX$ and a countably generated $\sigma$-algebra $\cA$, such that:\begin{enumerate}
\item Any atom in $\cP$ is of diameter $<\delta$, and $\tmu(\partial_\epsilon^V\cP)<C_\epsilon$ for all $\epsilon>0$;
\item $\cA$ is $\talpha^\bfn$-increasing and $V$-subordinate;
\item $[\tx]_\cA=[\tx]_{\bigvee_{k=0}^\infty\talpha^{k\bfn}\cP}\cap(V.\tx)$.
\end{enumerate}
\end{lemma}
\begin{proof} Indeed, for any $\delta>0$, there is a finite partition $\cP_X$ of $X$ into sets of diameter $<\frac\delta2$, and a constant $C>0$, such that $\big((\pi_X)_*\mu\big)(\partial_\epsilon\cP_X)<C\epsilon$ for all $\epsilon>0$ (\cite{EL10}*{Lemma. 7.27}). Then the same proof in \cite{EL10}*{7.32} shows for $(\pi_X)_*\mu$-a.e. $x\in X$, there is $\epsilon>0$, such that $B_\epsilon^{G_\bfn^u}.x\subset [x]_{\bigvee_{k=0}^\infty\alpha^{k\bfn}\cP_X}.$
When $\delta$ is sufficiently small, each member $P$ from the partition $\cP_X$ can be uniquely foliated by the action of $V$, such that every leaf has diameter less than $\delta$ inside some $V$-orbit. 

Define a $\sigma$-algebra $\hat\cP_X\succ\cP_X$ by taking the $V$-leaves inside each $P\in\cP_X$. Then $\hat\cP_X$ is countably generated as the leaves in each $P$ can be parametrized by points from a neighborhood of identity in $V\backslash G$, which is a second countable space.

One constructs a countably generated $\sigma$-algebra $\cA_X$  on $X$ by $\cA_X=\hat\cP_X\vee\left(\bigvee_{k=0}^\infty\alpha^{k\bfn}\cP_X\right)$. 

Let $\cP$ be the joining of $\cP_X$ with an arbitrary finite  parition $\cP_Y$ of $Y$ into sets of diameter less than  $\frac\delta2$. And define $\cA=\cA_X\times\cB_Y$.

Then $\cP$ is again finitely generated. Furthermore, any atom of $\cP$ is the product of an atom of $\cP_X$ and one of $\cP_Y$, and hence has diameter less than $\frac\delta2+\frac\delta2=\delta$.

$\talpha^\bfn\cA=\alpha^\bfn\cA_X\vee\beta^\bfn\cB_Y=\alpha^\bfn\cA_X\vee\cB_Y$. So in order to show part (2) it suffices to show $\cA_X$ is $\alpha^\bfn$-increasing, and $V$-subordinate with repsect to $(\pi_X)_*\mu$.

Because $\bfn$ is fixed, if $\delta$ is smooth enough, then $\alpha^\bfn\cP_X$ has sufficiently small atoms, and atoms of $\alpha^\bfn\hat\cP_X$ are the local $V$-leaves that foliate atoms of $\alpha^\bfn\cP_X$.
Notice $\hat\cP_X\vee\alpha^\bfn\cP_X$ is the $\sigma$-algebra whose atoms are $V$-leaves foliating atoms of $\cP_X\vee\alpha^\bfn\cP_X$. Thus $\alpha^\bfn\hat\cP_X\subset\hat\cP_X\vee\alpha^\bfn\cP_X\subset\cA$. Moreover, $\bigvee_{k=0}^\infty\alpha^{k\bfn}\cP_X $ is clearly $\alpha^\bfn$-increasing. It follows that $\cA_X$ is $\alpha^\bfn$-increasing.

By the choice of $\cP$,  for $(\pi_X)_*\mu$-a.e. $x\in X$, there is $\epsilon>0$, such that $[\tx]_{\bigvee_{k=0}^\infty\alpha^{k\bfn}\cP_X}$ contains $B_\epsilon^V.x$. In particular, $B_\epsilon^V.x$ is contained in $[x]_{\cP_X}$, and, since $\hat\cP_X$ is just the local $V$-foliation of $\cP_X$, also contained in $[\tx]_{\hat\cP_X}$.  So $B_\epsilon^V.x\subset[x]_{\cA_X}$. On the other hand, $[\tx]_{\cA_X}\subset[\tx]_{\hat \cP_X}\subset B_\delta^V.x$. Thus, $\cA_X$ is $V$-subordinate with respect to $(\pi_X)_*\mu$. This shows part (2).

Property (3) follows directly from  the discussion above. \end{proof}

\subsection{The suspension flow}\label{SecSuspension} 

It is a standard practice to extend a $\bZ^r$-action to an $\bR^r$-one using the suspension construction of Katok and Spatzier \cite{KS96}. We now set up notations regarding the suspension flow $\talpha_S$ of the action $\talpha$.

Let $\tX_S=\bR^r\times \tX/\sim$ where $\sim$ is the the equivalence relation given by \begin{equation}(\bfeta,x)\sim(\bfeta-\bfm,\talpha^\bfm.\tx), \forall\bfeta\in\bR^r,\bfm\in\bZ^r, x\in\tX.\end{equation} 
The space $\tX_S$ is a bundle over $\bT^r=\bR^r/\bZ^r$ where each fiber is isomorphic to $\tX$.

We write $\overline{(\bfeta,x)}\in\tX_S$ for the equivalence class containing $(\bfeta,x)$. For all point in $\tX_S$, there is a unique representative $(\bfeta,x)$ with $\bfeta\in[0,1)^r$.

 The suspension flow $\talpha_S:\bR^r\curvearrowright\tX_S$ is given by translations on the first component of $\bR^r\times\tX$. It is easy to check that this is compatible with the equivalence relation $\sim$. And for $\bfm\in\bZ^r$, $\talpha_S^\bfm.\overline{(\bfeta,\tx)}=\overline{(\bfeta,\talpha^\bfm.\tx)}$.
 
 For an $\talpha$-invariant measure $\tmu$ on $\tX$, define a probability measure $\tmu_S$ on $\tX_S$ by \begin{equation}\label{SuspMeasEq}\int_{\tX_S}f\di \tmu_S:=\int_{[0,1)^r}\int_\tX f\big(\overline{(\bfeta,\tx)}\big)\di\tmu(\tx)\di\eta, \forall f\in C(\tX_S).\end{equation} Then $\tmu_S$ is $\talpha_S$-invariant. If $\tmu$ is ergodic under $\talpha$, then so is $\tmu_S$ under $\talpha_S$.
 
For a subgroup $V\subset G$ that acts freely on $X$, it is not obvious whether one can make a canonical construction of a free $V$-action on $\tX_S$. Using the uniqueness of representative from $[0,1)^r\times\tX$, we define a free $V$-action on $\tX_S$ by
\begin{equation}\label{SuspVTransEq}v.\overline{(\bfeta,\tx)}=\overline{(\bfeta,v.\tx)}, \forall\bfeta\in[0,1)^r,\tx\in\tX.\end{equation} We stress that this action is in general discountinuous at $\tx_S\in\tX_S$ when $\tx_S$ projects to the boundary of the fundamental domain $[0,1)^r$.

When $\tmu$ is $\talpha$-invariant and $\tmu_S$ is defined by \ref{SuspMeasEq}, we can define the leafwise measure $(\tmu_S)_{\tx_S}^V$ by
\begin{equation}\label{SuspLeafEq}(\tmu_S)_{\overline{(\bfeta,\tx)}}^V:=\tmu_\tx^V,\ \bfeta\in[0,1)^r.\end{equation} The leafwise measure is defined on the set $\tX_S^*=\{\overline{(\bfeta,\tx)}:\bfeta\in[0,1)^r,\bfx\in X^*\}$, where $X^*$ is as in Proposition \ref{LeafMeas}. Notice $\tmu_S(\tX_S^*)=1$ and $(\tmu_S)_{\tx_S}^V$ is a measurable function in $\tx_S$.

For a general vector $\bfp\in\bR^r$ and $\tx_S=\overline{(\bfeta,\tx)}\in\tX_S$ where $\bfeta\in[0,1)^r$,  there is a unique vector $\bfm$ such that $\bfeta+\bfp-\bfm\in[0,1)^r$. Hence 
$\talpha_S^\bfp.\tx_S=\talpha_S^{\bfp-\bfm}.\talpha_S^\bfm.\overline{(\bfeta,\tx)}=\talpha_S^{\bfp-\bfm}.\overline{(\bfeta,\talpha^\bfm.\tx)}=\overline{(\bfeta+\bfp-\bfm,\talpha^\bfm.\tx)}$. It follows that $\talpha_S^\bfp$ acts along the $\tX$ fiber by $\alpha^\bfm$. In particular, by Proposition \ref{LeafMeas} we have
\begin{equation}\label{SuspActionEq}\alpha^\bfm_*(\tmu_S)_{\tx_S}^V\simeq(\tmu_S)_{\talpha_S^\bfp.\tx_S}^V.\end{equation}

It also follows from Lemma \ref{LeafAction} that, if $v\in V$ and $\tx_S, v.\tx_S\in\tX_S^*$, then \begin{equation}\label{SuspLeafTransEq}(\tmu_S)_{\tx_S}^V\simeq(\tau_v)_*(\tmu_S)_{v.\tx_S}^V.\end{equation}

Notice $\bfm$ depends not only on $\bfp$ but also on the initial position $\bfeta$, however it always approximates $\bfp$ in the sense that $\bfp-\bfm\in(-1,1)^r$. Especially, for each Lyapunov exponent $\chi$, $\chi(\bfm)$ is of bounded distance away from $t\chi(\bfp)$. Hence if $\chi(\bfp)>0$ (resp. $<0$), then $\talpha_S^{t\bfp}$ expands (resp. contracts) exponentially fast along the $V^{[\chi]}$ leaves as $k\rightarrow\infty$.

\subsection{The isometric direction}\label{IsomSec}

When $\chi(\bfp)=0$, $\talpha_S^{\bfp}$ behaves along $V^{[\chi]}$ direction with zero Lyapunov exponent. In this case, it may or may not acts approximately as a sequence of isometries.

\begin{definition}\label{IsomDef}For an element $v\in G$ and $\bfp\in\bR^r$. We say $v$ is {\bf isometric} under $\talpha_S^\bfp$ if for all $t\in\bR$, $\bfeta\in[0,1)^r$ and $\bfm\in\bZ^r$ with $\bfeta+t\bfp-\bfm\in[0,1)^r$,  $\alpha^\bfm.v$ is uniformly bounded.\end{definition}

Clearly, if $g\in V^{[\chi]}$ and is isometric under $\talpha_S^\bfp$, then $\chi(\bfp)=0$.

Let $\gov_0^{[\chi]}\subset \gov^{[\chi]}$ be the collection of all vectors on which all $\alpha^\bfm$'s that are as in Definition \ref{IsomDef} act in a uniformly bounded fashion. Since $\alpha^\bfm$ acts as a Lie algebra automorphism, $\gov_0^{[\chi]}$ is a Lie subalgebra. Therefore the isometric elements in $V^{[\chi]}$ under $\talpha_S^\bfp$ form a closed connected subgroup $V_0^{[\chi]}\exp\gov_0^{[\chi]}$, called the {\bf isometric subgroup} of $V^{[\chi]}$ under $\talpha_S^\bfp$. By commutativity of the action $\alpha$,  $\gov_0^{[\chi]}$ and $V_0^{[\chi]}$ are $\alpha$-invariant.

\begin{definition}\label{LeafIsomDef}We say an $\talpha$-invariant measure $\tmu$ {\bf has isometric support along} a coarse Lyapunov subgroup $V^{[\chi]}$, if for all $\bfp\in\ker\chi$ and $\tmu$-a.e $\tx\in\tX$, $\supp\tmu_\tx^{V^{[\chi]}}$ is contained in the $\talpha_S^\bfp$-isometric subgroup $V_0^{[\chi]}$ of $V^{[\chi]}$.\end{definition}

\begin{lemma}\label{IsomCpct}There exists a compact subset of $\Aut(V_0^{[\chi]})$ that contains all $\alpha^\bfm|_{V_0^{[\chi]}}$ if $t$, $\bfeta$ and $\bfm$ are as in Definition \ref{IsomDef}.\end{lemma}
\begin{proof}It is equivalent to find a compact set of $\Aut(\gov_0^{[\chi]})$ that contains all $\alpha^\bfm|_{\gov_0^{[\chi]}}$. By definition $\alpha^\bfm|_{\gov_0^{[\chi]}}$ is uniformly bounded as a matrix. To show it comes from a fixed compact set of $\Aut(\gov_0^{[\chi]})$ it suffices to prove that the inverse $\alpha^{-\bfm}|_{\gov_0^{[\chi]}}$ is also uniformly bounded. Note that for some integer vectors $\bfxi, \bfxi'\in\{-1, 0, 1\}^r$, the triple $(-t, \bfxi-\bfeta, \bfxi'-\bfm)$ also satisfies the conditions in Definition \ref{IsomDef}. So $\alpha^{\bfxi'-\bfm}|_{\gov_0^{[\chi]}}$ is uniformly bounded, which is equivalent to the uniform boundedness of $\alpha^{-\bfm}|_{\gov_0^{[\chi]}}$ as $\bfxi$ is chosen from a finite set.  \end{proof}

 Next, we give an alternative description of $\gov_0^{[\chi]}$.

Consider the Lie algebra $\gov^{[\chi]}=\bigoplus_{\chi'\in[\chi]}V^{\chi'}$, by \eqref{LyaEigenEq}, it can be decomposed as $\bigoplus_{j=1}^JW_j$, where each $W_j$ is a common generalized eigenspace for all $\alpha^\bfm$'s, with eigenvalues $\zeta_j^\bfm$. Consider the linear transform $U_j^\bfm$ of $W_j$ given by $\zeta_j^{-\bfm}\cdot\alpha^\bfm|_{W_j}$. Then $\{U_j^\bfm:\bfm\in\bZ^r\}$ is an abelian group and each $U_j^\bfm$ is unipotent. Hence there is a basis of $W_j$ that simultaneously upper-triangularize all $U_j^\bfm$'s.

On the group of upper-triangular matrices, which is a nilpotent Lie group, the logarithm map is a well-defined diffeomorphism, and the map $h\mapsto h^a$ is well defined for all $a\in\bR$.
As the $U_j^\bfm$'s commute with each other, for all $\bfq\in\bR^r$ we can uniquely define $U_j^\bfq=\prod_{i=1}^r(U_j^{\bfe_i})^{q_i}$ where $\bfq=\sum_{i=1}^rq_i\bfe_i$ is the standard coordinate decomposition of $\bfq$. When $\bfq=\bfm$ is an integer vector, this coincides with the already defined $U_j^\bfm$. Furthermore, set $U^\bfq=\bigoplus_{j=1}^JU_j^\bfq$. Then $U^\bfq$ is also a unipotent linear transform and $\bfq\mapsto U^\bfq$ is a group morphism from $\bR^r$.

Write $Z^\bfm=\bigoplus_{j=1}^J\zeta_j^\bfm\Id|_{W_j}$ and $U^\bfq=\bigoplus_{j=1}^J U_j^\bfq$. Then $\alpha^\bfm|_{\gov^{[\chi]}}=Z^\bfm U^\bfm$ and all the $Z^\bfm$'s and $U^\bfq$'s commute with each other.

\begin{lemma}\label{IsomFixed}$\gov_0^{[\chi]}$ is the maximal eigenspace $\ker(U^\bfp-\Id)$ of $U^\bfp$ in the generalized eigenspace $\gov^{[\chi]}$.\end{lemma}

\begin{proof} (1) We prove first $\ker(U^\bfp-\Id)\subset\gov_0^{[\chi]}$.

If $v\in\gov^{[\chi]}$ is fixed by $U^\bfp$. Let $t$, $\bfeta$ and $\bfm$ be as in Definition \ref{IsomDef}, then $\alpha^\bfm.v=Z^\bfm.v$. In order to show $\alpha^\bfm.v$ is uniformly bounded for all such $\bfm$, it suffices to show the $Z^\bfm$'s are uniformly bounded as matrices.

Actually, we claim there is a compact subsect of $\GL(\gov^{[\chi]})$ that contains all such $Z^\bfm$'s. By the construction of $Z^\bfm$, it suffices to show for each $j$, $|\zeta_j^\bfm|$ is uniformly bounded from $0$ and $\infty$ if $\bfm$ is as in Definition \ref{IsomDef}.

As $W_j\subset\gov^{[\chi]}$, $\log|\zeta_j^\bfm|$ is given by $\chi'(\bfm)$ for some Lyapunov functional $\chi'$ proportional to $\chi$. Therefore, $\chi'(\bfm)=\chi'(t\bfp)+\chi'(\bfm-t\bfp)=\chi'(\bfm-t\bfp)$ because $\bfp\in\ker\chi$. Now since there are only finitely many choices of $\chi'$ and $\bfm-t\bfp\in(-1,1)^r$, it follows that $\log|\zeta_j^\bfm|$ is uniformly bounded, which is what we want.

(2) We show next $\gov^{[\chi]}\subset\ker(U^\bfp-\Id)$. Assume $v\in\gov^{[\chi]}$ is not fixed by $U^\bfp$. As $U^\bfp$ is a unipotent matrix, $U^{t\bfp}.v$ is unbounded. Let $\bfeta\in[0,1)^r$ and $\bfm\in\bZ^r$ be such that $\bfeta+t\bfp-\bfm\in[0,1)^r$. 

We have already shown that $Z^\bfm$ takes value in a compact subset of $\GL(\gov^{[\chi]})$. Since $\bfm-t\bfp\in\bfeta-[0,1)^r\subset(-1,1)^r$ and $\bfq\mapsto U^\bfq$ is a group morphism, $U^{\bfm-t\bfp}$ also takes value in a compact subset of $\GL(\gov^{[\chi]})$. These compact sets are independent of the triple $(t, \eta, \bfm)$. 

Therefore since $U^{t\bfp}.v$ is unbounded and $Z^\bfm U^{\bfm-t\bfp}$ comes from a compact set of invertible matrices, $\alpha^\bfm.v=Z^\bfm U^\bfm.v=Z^\bfm U^{\bfm-t\bfp}U^{t\bfp}.v$ is unbounded when the triple $(t,\bfeta,\bfm)$ varies. That is, $v\notin\gov_0^{[\chi]}$. This shows the claim.\end{proof}

\begin{corollary}\label{IsomInv}For all $\bfm'\in\bZ^r$ and $\bfp'\in\bR^r$, $\gov_0^{[\chi]}$ is invariant under $Z^{\bfm'}$ and $U^{\bfp'}$.\end{corollary}
\begin{proof}This follows from the previous lemma and the fact that $Z^{\bfm'}$ and $U^{\bfp'}$ both commute with $U^\bfp$.\end{proof}

\section{The inductive argument}\label{SecInductive}
\subsection{Statement of inductive claim}\label{SecStatement}

In the rest of the paper,  let $X=G/\Gamma$ be a compact nilmanifold and $Y$ be a compact metric space, respectively equipped with $\bZ^r$-actions $\alpha$ and $\beta$, where $\alpha$ acts by automorphisms of $X$. Denote by $\talpha=\alpha\times\beta$ the product action on $\tX=X\times Y$.  If $X/H$ is an algebraic factor of $X$, denote $\tX/H=(X/H)\times Y$. When $H$ is $\alpha$-invariant, $\talpha$ factors naturally onto a $\bZ^r$-action on $\tX/H$, which we still denote by $\talpha$.

The proof of Theorem \ref{MeasureThm} is carried out by induction. More precisely, we aim to prove the following slightly stronger statement.

\begin{theorem}\label{MeasureInductive} Let $\tmu$ be an ergodic $\talpha$-invariant probability measure on $\tX$. Suppose that, for all finite index subgroups $\Sigma\subset\bZ^r$ and all $\alpha|_\Sigma$-invariant rational proper normal subgroups $H$ of $G$, the following conditions hold:
\begin{enumerate}
\item The algebraic factor action $\alpha|_\Sigma$ on $X/H$ is not virtually cyclic;
\item For any $\talpha|_\Sigma$-ergodic component $\tmu'$ of $\tmu$, there exists at least one $\bfn\in\Sigma$, such that the induced isomorphism $\talpha^\bfn$ on $\tX/H$ has positive conditional entropy $h_{\pi_{\tX/H}\tmu'}(\talpha^\bfn|\cB_Y)$ with respect to the natural projection $\pi_{\tX/H}\tmu'$ of $\tmu'$. 
\end{enumerate}
 Then $\tmu$ is the product measure between the Lebesgue measure $\rmm_X$ and an ergodic $\beta$-invariant probability measure on $Y$.
\end{theorem}

\subsection{The proof}\label{SecOutline}

We give the proof of  Theorem \ref{MeasureInductive}, while acknowledging several propositions that will be established in later sections.
\begin{proof}[Proof of  Theorem \ref{MeasureInductive}] The theorem is proved inductively, based on the dimension of the nilmanifold $X$. 

If $\dim X=0$, then $X$ is a point and $\tX=Y$, so the theorem holds trivially.

Suppose now the theorem was true for all configurations with $\dim X<d$. Let $\dim X=d$, by Lemma \ref{IrrVertical}, there is a non-trivial rational central subgroup $Z \subset G$ such that $\alpha$, when restricted to a finite index subgroup $\Lambda\subset\bZ^r$, preserves $Z$ and induces a totally irreducible action on the quotient torus $Z/(Z\cap\Gamma)$. 

$\tmu$ can be written as the average of at most $|\bZ^r/\Lambda|$ ergodic components $\tmu_j$. These components are permuted transitively among themselves by the $\bZ^r$ action. So it suffices to prove $\tmu_1=m_X\times\nu_1$ for some $\nu_1$. 

{\bf We will keep the following settings throughout the rest of the paper. }

\begin{itemize}

\item Without loss of generality,  assume $\Lambda=\bZ^r$, and thus $\tmu_1=\tmu$.

\item Let $\dG=G/Z$, $\dGamma=\Gamma/(Z\cap\Gamma)$, $\dX=\dG/\dGamma=X/Z$ and $\dtX=\dtX/Z=\dX\times Y$. 
\item Write by $\dtmu$ the projection of $\tmu$ to $\dX$. The induced action on $(\dtX,\dtmu)$ is still denoted by $\talpha=\alpha\times\beta$. Regard the Borel $\sigma$-algebra $\cB_\dtX$ of $\dtX$ as a subalgebra of the Borel $\sigma$-algebra $\cB_\tX$ of $\tX$. 

\end{itemize}

For any finite index subgroup $\Sigma$, let $\tmu=\frac1N\sum_{i=1}^N\tmu^i$ be the $\talpha|_\Sigma$-ergodic decomposition of $\tmu$. Then $\dtmu=\frac1N\sum_{i=1}^N\dtmu^i$, while for each $i$, the projection $\dtmu^i$ of $\tmu^i$ is also $\talpha|_\Sigma$-ergodic. Hence every $\talpha|_\Sigma$ ergodic component of $\dtmu$ must be one of the $\dtmu^i$'s. Because any algebraic factor of $(\talpha|_\Sigma,\dtX,\dtmu^i)$ is also one of $(\talpha,\tX,\tmu^i)$. The action $\talpha|_\Sigma$ on $(\dtX,\dtmu)$ satisfies the conditions of the theorem, too. By the inductive hypothesis, \begin{equation}\label{InductiveEq}\dtmu=\rmm_\dX\times\nu,\end{equation} where $\nu=(\pi_Y)_*\tmu$ is an ergodic $\beta$-invariant probability measure on $Y$.

Given the hypothesis above, the following questions can be asked:
\begin{enumerate}[A.]
\item Is the action $\alpha:\bZ^r\curvearrowright Z$ virtually cyclic?
\item Does $h_\tmu(\talpha^\bfn|\cB_\dtX)=0$ hold for all $\bfn\in\bZ^r$?
\item Does $\tmu$ has isometric support along every coarse Lyapunov subsgroup $V^{[\chi]}$?
\end{enumerate}

Any combination of answers to these questions is covered by at least one of the following four cases:
\begin{center}
  \begin{tabular}{ c | c | c | c}
      & A & B & C\\ \hline
    I & No & No&\\ \hline
    II  & Yes &  \\ \hline
    III  &   &  &No\\ \hline
    IV & &Yes&Yes
  \end{tabular}
\end{center}

The following property will be established in Sections \ref{SecBerendFiber}, \ref{SecRk1Fiber}, \ref{SecShear}, \ref{SecEnt0Fiber} respectively for each of these cases:

\begin{property}\label{TransInvProp}{\rm (Translational Invariance) }\it There exists a non-zero class $[\chi]$ of Lyapunov exponent functional, and a non-trivial connected closed subgroup $V\subset V^{[\chi]}$, such that for $\tmu$-a.e. $\tx\in\tX$, there is a non-zero element $v\in V$ such that the leafwise measure $\tmu_\tx^V$ satisfies $(\tau_v)_*\tmu_\tx^V\simeq\tmu_\tx^V$, where $\tau_v$ denotes the right translation by $v$.\footnote{In fact, this would force $(\tau_v)_*\tmu_\tx^{V^{[\chi]}}\simeq\tmu_\tx^{V^{[\chi]}}$ and thus one can take $V=V^{[\chi]}$.}\end{property}

Here it makes sense to talk about $\tmu_\tx^V$, because by Lemma \ref{LyaFree}, $V\subset V^{[\chi]}$ acts freely on $X$ and therefore on $\tX$ as well.

Once Property \ref{TransInvProp} is verified, by Proposition \ref{TIMeas} there exists a finite index subgroup $\Sigma\subset\bZ^r$, such that in the ergodic decomposition $\tmu=\frac1N\sum_{i=1}^N\tmu^i$ with respect to the restriction $\alpha|_\Sigma$, each ergodic component $\tmu^i$ is invariant under left translations by some $\alpha|_\Sigma$-invariant normal rational subgroup $L^i$.

Each point in $\tX$ has a compact $L^i$-orbit. The quotient space modulo the orbit equivalence is $\tX/L^i=X/L^i\times Y$. 

For any finite index subgroup $\Sigma'\subset\Sigma$, any algebraic factor of the induced action $\talpha|_{\Sigma'}$ on $\tX/L^i$, is also an algebraic factor of the action $\talpha|_{\Sigma'}$ on $\tX$, and thus is not virtually cyclic. 

Furthermore, any $\talpha|_{\Sigma'}$-ergodic component  $(\hat\tmu^i)^j$ of the projection $\hat\tmu^i$ of $\tmu^i$ to $\tX/L^i$ is the projection of some $\talpha|_{\Sigma'}$-ergodic component of $\tmu^i$. Therefore by assumption (2) of Theorem \ref{MeasureInductive}, the projection of  $(\hat\tmu^i)^j$ to any $\talpha|_{\Sigma'}$-equivariant algebraic factor has positive entropy for at least one $\talpha^\bfn$.

Therefore both assumptions in Theorem \ref{MeasureInductive} remain valid for $\hat\tmu^i$. Since $L^i$ is non-trivial, $\dim(X/L^i)<\dim X$ and the inductive hypothesis applies. So $\hat\tmu^i=\rmm_{X/L^i}\times\nu^i$ where $\nu^i$ is some ergodic $\beta$-invariant measure on $Y$. Because $\tmu^i$ is $L^i$-invariant, it must be equal to $\rmm_X\times\nu^i$.

Since this can be carried out for every $\tmu^i$, we conclude that $\tmu=\frac1N\sum_{i=1}^N\rmm_X\times\nu^i=\rmm_X\times\nu$ where $\nu$ is the average of the $\nu^i$'s. This completes the proof of Theorem \ref{MeasureInductive}.
\end{proof}

Theorem \ref{MeasureThm} is a direct consequence of Theorem \ref{MeasureInductive}
\begin{proof}[Proof of Theorem \ref{MeasureThm}] In Theorem \ref{MeasureInductive}, let $Y$ be a single point. We obtain that in the setting of Theorem \ref{MeasureThm}, if $\mu\neq\rmm_X$ then for some finite index subgroup $\Sigma\subset\bZ^r$ and some $\alpha|_\Sigma$-ergodic component $\mu'$, one of the two obstructions happens on an $\alpha|_\Sigma$-equivariant algebraic factor $X/H$. But because of finite index, there are only finitely many $\alpha|_\Sigma$-ergodic components, and each of them can be written as $\alpha^\bfm_*\mu'$ for some $\bfm\in\bZ^r$. It is easy to check that for each $\alpha^\bfm_*\mu'$, one of the obstructions is present on the $\alpha|_\Sigma$-equivariant algebraic factor $X/\alpha^\bfm(H)$. This completes the proof.\end{proof}

\begin{corollary}\label{FullXabCor}If $\alpha:\bZ^r\curvearrowright X$ has no virtually cyclic algebraic factors, and $\mu$ is an $\alpha$-invariant measure that projects to $\rmm_{X_\ab}$ on the maximal torus factor $X_{ab}$, then $\mu=\rmm_X$.\end{corollary}

No ergodicity is required here.

\begin{proof}For any subgroup $\Sigma\subset\bZ^r$ of finite index, by Remark \ref{Rk2ResRmk} $\alpha|_\Sigma$ has no virtually cyclic algebraic factor action either. This is also true for the projected action $\alpha|_\Sigma$ on $X_\ab$. Therefore by Lemma \ref{Rk2Equiv}, $\rmm_{X_\ab}$ is ergodic under $\alpha|_\Sigma$.

Let $\mu'$ be an $\alpha|_\Sigma$-ergodic component of $\mu$,  then $\mu'$ projects to an $\alpha|_\Sigma$-ergodic component of $\rmm_{X_\ab}$, which has to be $\rmm_{X_\ab}$ itself.

Moreover, for any non-trivial $\alpha|_\Sigma$-equivariant factor $\dX$ of $X$, its maximal torus factor $\dX_\ab$ is a non-trivial algebraic factor of $X_\ab$. The projection of $\mu'$ to $\dX_\ab$ factors through $\rmm_{X_\ab}$, and is hence equal to $\rmm_{\dX_\ab}$. Since the action $\alpha|_\Sigma$ on $\dX_\ab$ is not virtually cyclic, this guarantees positive entropy for some $\alpha^\bfn$ with respect to the projection of $\mu'$ to $\dX_\ab$. The positive entropy then lifts to the projection of $\mu'$ to $\dX$. 

Therefore neither of the obstructions in Theorem \ref{MeasureThm} occurs. So $\mu=\rmm_X$. \end{proof}

\begin{proof}[Proof of Corollary \ref{MeasureCor}] For any non-trivial $\alpha$-equivariant algebraic factor $\dX$ of $X$. The maximal torus factor $\dX_\ab$ is a non-trivial $\alpha$-equivariant factor of $X_\ab$. By total irreducibility of $X_\ab$, $\dX_\ab=X_\ab$. Thus the induced action on $\dX_\ab$, and hence the one on $\dX$ as well, is not virtually cyclic.

Furthermore, the projection $\mu_\ab$ of the $\alpha$-ergodic measure $\mu$ to $X_\ab$ remains ergodic under the factor action. When $\mu_\ab$ has positive entropy under some $\alpha^\bfn$, Einsiedler and Lindenstrauss \cite{EL03} proved that $\mu_\ab=\rmm_{X_\ab}$. By Corollary \ref{FullXabCor}, $\mu=\rmm_X$.\end{proof}

\subsection{Topological rigidity}\label{SecTopo}

We then deduce the topological Theorem \ref{TopoThm} from measure rigidity.

\begin{proof}As the action $\alpha$ has no virtually cyclic factor, by Lemma \ref{Rk2Equiv}, the Lebesgue measure $\rmm_{X_\ab}$ is ergodic under the induced action $\alpha:\bZ^r\curvearrowright X_\ab$. Take a generic point $z\in X_\ab$ such that the Birkhoff ergodic averages $\frac1{(2N+1)^r}\sum_{\bfn\in\{-N,\cdots,N\}^r}\delta_{\alpha^\bfn.z}$ equidistribute towards $\rmm_{X_\ab}$. Let $x\in A$ be a point that projects to $z$, and $\mu$ be a weak$^*$-limit of a subsequence of $\frac1{(2N+1)^r}\sum_{\bfn\in\{-N,\cdots,N\}^r}\delta_{\alpha^\bfn.\tx}$. Then $\mu$ is $\alpha$-invariant and projects to $\rmm_{X_\ab}$. By Corollary \ref{FullXabCor}, $\mu=\rmm_X$. We conclude that $A=X$ because $\mu$ is supported on $A$.\end{proof}

\begin{proof}[Proof of Corollary \ref{TopoCor}] If the projection of $A$ in $X_\ab$ is infinite, then by Berend's theorem (Theorem \ref{Berend}), the projection is $X_\ab$. Hence by Theorem \ref{TopoThm}, $A=X$.\end{proof}
\section{Case of higher rank fibers with positive entropy}\label{SecBerendFiber}

In this section, we prove that \begin{proposition}\label{PropBerendFiber}In the settings of \S\ref{SecOutline}, if the $\bZ^r$-action $\alpha$ on the normal abelian subgroup $Z$ has at least rank $2$ and $h_\tmu(\talpha^\bfn|\cB_\dtX)>0$ for some $\bfn\in\bZ^r$, then Property \ref{TransInvProp} holds.
\end{proposition}

We quote the following theorem of Einsiedler and Lindenstrauss.

\begin{theorem}\label{IrrTorus}\cite{EL03} Suppose $\alpha$ is a totally irreducible, not virtually cyclic $\bZ^r$-action on a torus $\bT^d$ and $\talpha=\alpha\times\beta$ is the product of $\alpha$ with a $\bZ^r$ action $\beta$ on a compact metric space $Y$. Use $V^{[\chi]}\subset\bR^d$ to denote the coarse Lyapunov subgroups with respect to $\alpha$.

Suppose $\tmu$ is a $\talpha$-invariant measure on $\bT^d\times Y$. \begin{enumerate}
\item If $\chi$ and $\chi'$ are two different coarse Lyapunov exponents of $\alpha$, then 
$$\label{EntropyRatio}\frac{h_\tmu(\talpha,V^{[\chi]})}{h_\tmu(\talpha,V^{[\chi']})}=\frac{h_{\rmm_{\bT^d}}(\alpha,V^{[\chi]})}{h_{\rmm_{\bT^d}}(\alpha,V^{[\chi']})}.$$
\item  Suppose for some $\bfn\in\bZ^r$, $h_\tmu(\talpha|\cB_Y)>0$. Then for any non-zero Lyapunov exponent $\chi$ of $\alpha$, at $\tmu$-almost every $\tx\in\bT^d\times Y$,  $\tmu_\tx^{V^{[\chi]}}\simeq(\tau_v)_*\tmu_\tx^{V^{[\chi]}}$ some non-zero vector $v\in V^{[\chi]}$. 
\end{enumerate}\end{theorem}

In this case, as shown in \cite{EL03}, the invariance of the leafwise measure actually forces $\tmu$ to be invariant under translation by any element in $\bR^d$, in other words, $\tmu$ is uniform in all $\bT^d$ fibers.

It should be remarked that in the proof of the above theorem in \cite{EL03}, the assumption that the underlying space is a direct product is not used. And the theorem generalizes to skew products with $\bT^d$ fibers.

\begin{proposition}\label{IrrTorusBundle} Suppose $\alpha$ is a totally irreducible, not virtually cylic $\bZ^r$-action on a torus $\bT^d$. Let $Y$ be a compact metric spaces and $\Omega$ be a principal $\bT^d$-bundle over $Y$.  For $x\in\Omega$, $v\in\bR^d$, denote by $x+v$ the translation of $x$ by $v$ along the fiber.

Assume $\talpha$ is a continuous $\bZ^r$-action on $\Omega$ such that $$\talpha^\bfn(x+v)=\talpha^\bfn.x+\alpha^\bfn v$$ for all $x\in\Omega$, $v\in\bR^d$, $\bfn\in\bZ^r$. Then the same conclusions from Theorem \ref{IrrTorus} hold.
\end{proposition}
\begin{proof}The proof in \cite{EL03} works here without change, using the partition $\cP$ in Lemma \ref{SubordExist}.\end{proof}

\begin{proof}[Proof of Proposition \ref{IrrTorus}]The space $X$ is a torus bundle over $\dX$ with fibers isomorphic to $Z/(Z\cap\Gamma)\cong\bT^d$, where $Z\cong\bR^d$. Moreover, if $v\in Z$ then $\talpha^\bfn(v.\tx)=(\alpha^\bfn v).(\talpha^\bfn.\tx)$. Hence Property \ref{TransInvProp} holds by Proposition \ref{IrrTorusBundle}, as we assume $h_\tmu(\talpha^\bfn|\cB_\dtX)>0$ for some $\bfn$.\end{proof}

\section{Case of lower rank fibers}\label{SecRk1Fiber}

The next case to treat is: \begin{proposition}\label{PropRk1Fiber}In the settings of \S \ref{SecOutline}, if the induced action $\alpha$ on $Z$ is virtually cyclic, then Property \ref{TransInvProp} holds.
\end{proposition}

\subsection{Leafwise measures along fibers} The quotient $T:=Z/(Z\cap\Gamma)$ is a group isomorphic to some $\bT^d$ as $Z$ is a connected abelian Lie group. 

Note that $T$ acts on $X$, and the action is free.  To see this, it suffices to show that for all $g\in G$, the stabilizer of $gx_0\in X$ inside $Z$ is $Z\cap\Gamma$, where $x_0$ denotes the origin of $G/\Gamma$; or equivalently $Z\cap g\Gamma g^{-1}=Z\cap\Gamma$. This follows easily from the condition that $Z$ is in the center of $G$. 

Hence $T$ acts on $\tX$ freely as well, and each orbit is a fiber of the projection $\pi$ to $\dtX$.

Therefore we can talk about the leafwise measure $\tmu_\tx^\T$. In this case, it is more convenient to define it through the conditional measure with respect to the Borel $\sigma$-algebra $\cB_{\dtX}$.

Let $\tmu_\tx^{\cB_\dtX}$ be the conditional measure of $\tmu$ with respect to $\cB_\dtX$, which is a priori defined only for $\tmu$-almost every $\tx$. It is a probability measure supported on $[\tx]_{\cB_\dtX}=T.\tx$ and \begin{equation}\label{FiberCondEq1}\tmu_\tx^{\cB_\dtX}=\tmu_{\tx'}^{\cB_\dtX}\text{ if }\tx'\in T.\tx.\end{equation} Imposing equation \eqref{FiberCondEq1}, one can actually define  $\tmu_\tx^{\cB_{\dtX}}$ for $\dtmu$-almost every $\dtx$ and every $\tx\in\pi^{-1}(\dtx)$. Note that after this extension, it is no longer true that almost every $\tx$ belongs to the support of $\tmu_\tx^{\cB_\dtX}$. The map $\tx\mapsto \tmu_\tx^{\cB_\dtX}$ is $\cB_\dtX$-measurable.

So for almost every $\dtx\in\dtX$ we can define $\tmu_\dtx^{\cB_{\dtX}}=\tmu_\tx^{\cB_\dtX}$ using any $\tx\in\pi^{-1}(\dtx)$. Then \begin{equation}\label{TFiberDesEq}\tmu=\int_{x\in\tX}\tmu_\tx^{\cB_\dtX}\di\tmu(x)
=\int_{x\in\tX}\tmu_{\pi(\tx)}^{\cB_\dtX}\di\tmu(x)
=\int_{\dx\in\dtX}\tmu_\dx^{\cB_\dtX}\di\dtmu(\dx).\end{equation}

For $\tx\in\tX$, define $\tmu_\tx^\T$ to be the pullback of $\tmu_\tx^{\cB_{\dtX}}$ from $T.\tx$ to $T$ with respect to the homeomorphism $t\mapsto t.\tx$. Then \eqref{FiberCondEq1} implies \begin{equation}\label{FiberCondEq2}\tmu_\tx^\T=(\tau_t)_*\tmu_{t.\tx}^\T,\end{equation} for all points in almost every fiber. The support of  $\tmu_\tx^\T$ does not have to contain identity. $\tx\mapsto \tmu_\tx^\T$ is a Borel measurable map from $\tX$ to the space of probability measures on $T$.

Since the action $\talpha$ sends fibers to fibers, $\tmu_{\talpha^\bfn.\tx}^{\cB_\dtX}=\talpha^\bfn_*\tmu_\tx^{\cB_\dtX}$. Because of this and the fact hat  $\talpha^\bfn(t.\tx)=(\alpha^\bfn t).(\talpha^\bfn.\tx)$ for all $t\in T$, we have, similarly to Lemma \ref{LeafAction}, \begin{equation}\label{FiberCondEq3}\tmu_{\talpha^\bfn.\tx}^\T=\alpha^\bfn_*\tmu_\tx^\T.\end{equation}

\subsection{Invariance along the stable foliation} Let $\Xi\subset\bZ^r$ be the subgroup consisting of all $\bfn\in\bZ^r$ such that $\alpha^\bfn$ acts on $Z$, hence on $T$, trivially. The hypothesis that the action induced by $\alpha$ on $Z$ is virtually cyclic implies that $\rank\Xi\geq (r-1)$.

Note that $\dX$ is not trivial. Otherwise, the action $\alpha$ on $X=Z/(Z\cap\Gamma)$ is virtually cyclic. This contradicts the condition in Theorem \ref{MeasureInductive} that all algebraic factors of the action $\alpha:\bZ^r\curvearrowright X$ has rank $2$ or higher.

The projected action on $\dX$ has no virtually cyclic algebraic factors. Therefore by Corollary \ref{NonPropLya}, its Lyapunov exponents are not all proportional. In particular, there is at least one coarse Lyapunov exponent $\chi$ of $\alpha$ on $\dG$ whose kernel does not contain $\Xi$. By Remark \ref{QuotientEigenRmk} and equation \eqref{LyaEigenEq}, this is also a Lyapunov exponent of $\alpha$ on $G$. 

Therefore, we may fix a Lyapunov subgroup $V^{[\chi]}\subset G$, and some $\bfn$, such that:\begin{itemize}
\item $V^{[\chi]}$ is non-trivial and has non-trivial projection with respect to $\dG$;
\item $\chi(\bfn)<0$, or equivalently $V^{[\chi]}\in G_\bfn^s$;
\item $\alpha^\bfn$ acts trivially on $T$.
\end{itemize}

\begin{proof}[Proof of Proposition \ref{PropRk1Fiber}]
Fix $\epsilon\in (0,\frac13)$, by Luzin's Theorem, there is a compact subset $A_0\subset\tX$ with $\tmu(A_0)>1-\epsilon^2$ such that $\tmu_\tx^\T$ is uniformly continuous in $\tx$ on $A_0$. Let $\dA=\pi(A_0)$, then $\dtmu(\dA)>1-\epsilon^2$, and thanks to \eqref{FiberCondEq2},  $\tmu_\tx^\T$ is uniformly continuous on $A=\pi^{-1}(\dA)$, which consists of entire fibers.

By the maximal ergodic theorem, $\dtmu\big(\{\dtx: \sup_{n\geq 1}\frac1n\sum_{k=0}^{n-1}{\bf 1}_{\dA^c}(\talpha^{k\bfn}.\dtx)>\epsilon\}\big)<\epsilon$. Denote by $\dot B$ the complement of this set.

Recall that by inductive hypothesis, $\dtmu$ is invariant under left translation by $\dG$. Fix any $v\in V^{[\chi]}$ with non-zero projection $\dv\in\dG$, then $\dtmu(\dv^{-1}.\dot B)=\dtmu(\dot B)>1-\epsilon$.

Consider $\dtx\in\dv^{-1}.\dot B\cap \dot B$ and $\tx\in\pi^{-1}(\dtx)$. Let $\tx'=v.\tx$, then $\pi(\tx')=\dv.\dtx\in \dot B$. Thus by choice of $\dot B$, there are respectively two subsequences $S$, $S'$, both of asymtoptic density $>1-\epsilon$, in $\bN$, such that $\talpha^{k\bfn}\dtx$ (resp. $\talpha^{k\bfn}\dtx'$) belongs to $\dA$ if $k\in S$ (resp. $S'$). Since $2(1-\epsilon)>1$, the intersection $S\cap S'$ is an infinite sequence. Then for $k\in S\cap S'$, the points $\talpha^{k\bfn}.\tx$ and $\talpha^{k\bfn}.\tx'$ are both in $A$; and the distance between them decays exponentially fast as $k$ grows since $\alpha$ contracts $V$ exponentially fast. Hence $\dist(\tmu_{\talpha^{k\bfn}.\tx}^\T,\tmu_{\talpha^{k\bfn}.\tx'}^\T)\rightarrow 0$. But as $\alpha^\bfn$ acts trivially on $T$, using equation \eqref{FiberCondEq3}, it follows that $\tmu_\tx^\T=\tmu_{\tx'}^\T$. By the construction of $\tmu_\tx^\T$ as well as the fact that $V$ commutes with $Z$ in $G$, this implies $\tmu_{\tx'}^{\cB_{\dtX}}=(\iota_v)_*\tmu_\tx^{\cB_{\dtX}}$, where $\iota_v$ denotes the left-translation by $v$. Or equivalently \begin{equation}\label{CondTransInvEq}\tmu_{\dv.\dtx}^{\cB_{\dtX}}=(\iota_v)_*\tmu_\dtx^{\cB_{\dtX}}. \end{equation}

\eqref{CondTransInvEq} holds for all $\dtx$ from $\dv^{-1}.B\cap B$. Since $\dtmu(\dv^{-1}.B\cap B)>1-2\epsilon$, by letting $\epsilon\rightarrow 0$, this actually holds for $\dtmu$-almost all $\dtx$.

Hence by \eqref{TFiberDesEq},
\begin{equation}\begin{split}\tmu=&\int_{\dtx\in\dtX}\tmu_\dtx^{\cB_{\dtX}}\di\dtmu(\dtx)=\int_{\dtx\in\dtX}\tmu_{\dv.\dtx}^{\cB_{\dtX}}\di\dtmu(\dtx)\\
=&(\iota_v)_*\int_{\dtx\in\dtX}\tmu_(\dtx)^{\cB_{\dtX}}\di\dtmu(\dtx)=(\iota_v)_*\tmu,
\end{split}\end{equation}
where we used the fact that $\dtmu$ is invariant under $\dtx\mapsto\dv.\dtx$.

So we proved $\tmu$ is invariant under the left translation by the non-trivial element $v\in V^{[\chi]}$, which implies Property \ref{TransInvProp} with respect to $V=V^{[\chi]}$ by Proposition \ref{LeafMeas}.(2).\end{proof}

\section{Case of non-isometric leafwise measures}\label{SecShear}

We now show that if $\tmu$ does not have isometric support along a coarse Lyapunov subgroup $V^{[\chi]}$ (see Definition \ref{LeafIsomDef}), then one gets translational invariance along $V^{[\chi]}_0$. 

\begin{proposition}\label{ShearProp}In the settings of \S\ref{SecOutline},  assume that for some coarse Lyapunov subgroup $V^{[\chi]}$, $\tmu$ does not have isometric support along $V^{[\chi]}$.
 
Then Property \ref{TransInvProp} holds. Actually, for $\tmu$-almost every $\tx\in\tX$, there exists a non-zero element $v\in V_0^{[\chi]}$ such that $\tmu_\tx^{V_0^{[\chi]}}=(\tau_v)_*\tmu_\tx^{V_0^{[\chi]}}$.
\end{proposition}

In the rest of this section, we assume the conditions from the proposition, and simply write $V$, $V_0$ respectively for $V^{[\chi]}$ and the $\alpha_S^\bfp$-isometric subgroup $V^{[\chi]}_0\subset V^{[\chi]}$. Let $\gov$, $\gov_0$ be the corresponding Lie algebras.

In order to prove the proposition, we will work with the suspension flow $\talpha_S\curvearrowright\tX_S$. Construct $\tmu_S$ as in \eqref{SuspMeasEq}, then it is $\talpha_S$-invariant and ergodic as $\tmu$ is $\talpha$-ergodic. 
 
From now on, we assume for the sake of contradiction that there is a positive portion of points $\tx\in\tX$ at which $\tmu_\tx^{V_0}$ has no translational invariance property as in Proposition \ref{ShearProp}. Equivalently, in the suspension $(\tX_S,\tmu_S)$, there is a set of positive measure on which $(\tmu_S)_{\tx_S}^{V_0}$ is not invariant under the right translation by any non-zero element from $V_0$. By \eqref{SuspActionEq} and Corollary \ref{IsomInv}, this set is $\talpha_S$-invariant modulo a null set. Therefore by ergodicity, for  $\tmu_S$-almost every $\tx_S$, $(\tmu_S)_{\tx_S}^{V_0}$ does not have this invariance property.

\subsection{Mass out of isometric leaves} Whenever $\tmu_\tx^{V}$ is defined, let \begin{equation}l(\tx)=\inf\{l>0: \tmu_\tx^{V}\big(\{v\in V\backslash V_0, \|\log v\|<l\}\big)>0\}.\end{equation}

By Lemma \ref{LeafAction} and Corollary \ref{IsomInv}, 
\begin{equation}\begin{split}
&\tmu_{\talpha^\bfn.\tx}^{V}\big(\{v\in V\backslash V_0, \|\log v\|<l\}\big)\\
= &\alpha^\bfn_*\tmu_\tx^{V}\big(\{v\in V\backslash V_0, \|\log v\|<l\}\big)\\
\geq &\tmu_\tx^{V}\big(\{v\in V\backslash V_0, \|\log v\|<\|\alpha^\bfn|_{V}\|^{-1}\cdot l\}\big).
\end{split}\end{equation}
Thus,  $l(\talpha^\bfn.\tx)\geq \|\alpha^\bfn|_{V}\|^{-1}\cdot l(\tx)$ for all $\bfn$. Thus $\{\tx: l(\tx)<\infty\}\subset \tX$ is $\talpha$-invariant. Moreover, by assumptions in Proposition \ref{ShearProp}, the subset $\{\tx: l(\tx)<\infty\}$ has positive measure.  Hence by ergodicity,
$l(\tx)<\infty$ for $\tmu$-almost every $\tx$.

\begin{lemma}\label{MassTouch}$l(\tx)=0$ for $\tmu$-almost every $\tx$. \end{lemma}
\begin{proof}Suppose this is not true, then for some $M>0$, $\tX_M:=\{\tx: l(\tx)>M\}$ has positive measure. As $\chi\neq 0$, we can fix some $\bfn$ with $\chi(\bfn)>0$. Reversing the equation above, we get $l(\talpha^{-k\bfn}.\tx)\leq\|\alpha^{-k\bfn}|_{V}\|\cdot l(\tx)$. However, as $\chi(\bfn)>0$, $\|\alpha^{-k\bfn}|_{V}\|$ decays to $0$ as $k\rightarrow\infty$. Hence for every $\tx$, $\talpha^{-k\bfn}.\tx$ eventually leaves $\tX_M$ and never returns, which cannot happen because of Poincar\'e's Recurrence Theorem. This proves the claim.\end{proof}

We still fix $\bfn\in\bZ^r$ with $\chi(\bfn)>0$. By Lemma \ref{SubordExist} one can find a $V$-subordinate $\talpha^\bfn$-increasing $\sigma$-algebra $\cA$ of $\tX$. Actually, the proposition allows to assume that for almost every $\tx$, $[\tx]_\cA\subset B_1^V.\tx$. 

Denote by $\cA_S$ the image of the product $\sigma$-algebra $\cB_{[0,1)^r}\times\cA$ under the bijection $(\bfeta,\tx)\mapsto \overline{(\bfeta,\tx)}$. Then $\cA_S$ is a $V$-subordinate $\talpha_S^\bfn$-increasing $\sigma$-algebra of $\tX_S$.

Furthermore, notice $[\tx_S]_{\cA_S}=\big\{\overline{(\bfeta,\tx')}:\tx'\in[\tx]_\cA\big\}$ in this case. Therefore, since $\tmu_\tx^V.\tx$ is proportional to $\tmu_\tx^\cA$ on $[\tx]_\cA$, $(\tmu_S)_{\tx_S}^V.\tx_S$ is proportional to $(\tmu_S)_{\tx_S}^{\cA_S}$ on $[\tx_S]_{\cA_S}$.

By Lemma \ref{MassTouch} and equation \eqref{SuspLeafEq}, for $\tmu_S$-almost every $\tx_S$ and any $l>0$, $(\tmu_S)_{\tx_S}^V\big(B_l^V\backslash V_0\big)>0$. Since for almost every $\tx_S$, $[\tx_S]_{\cA_S}$ is a bounded open  neighborhood of $\tx_S$ in $V.\tx_S$, it follows that \begin{equation}(\tmu_S)_{\tx_S}^{\cA_S}\big([\tx_S]_{\cA_S}\backslash(V_0.\tx_S)\big)>0,\text{ for }\tmu_S\text{-a.e. }\tx_S.\end{equation}

It follows that for all $\epsilon_1>0$, there exists a constant $\delta_1>0$ and a subset $\Omega_1\subset\tX_S$ with $\tmu_S(\Omega_1)>1-\epsilon_1$, such that for every $\tx_S\in\Omega_1$, \begin{equation}\label{ShearGeneEq1}(\tmu_S)_{\tx_S}^{\cA_S}\big([\tx_S]_{\cA_S}\backslash(V_0.\tx_S)\big)>\delta_1.\end{equation}

\subsection{Generic points} As the leafwise measure $(\tmu_S)_{\tx_S}^{V_0}$ is a measurable function in $\tx_S$, by Luzin's Theorem, for all $\epsilon_2>0$, there is a compact set $K\subset\Omega$ on which $(\tmu_S)_{\tx_S}^{V_0}$ varies uniformly continuously, with $\tmu_S(K)>1-\frac{\epsilon_2^3}2$.

In addition to this, we may assume for every $\tx_S\in K$, \begin{equation}\label{ShearNoTransEq}(\tmu_S)_{\tx_S}^{V_0}\not\simeq(\tau_v)_*(\tmu_S)_{\tx_S}^{V_0}, \forall v\in V_0,\end{equation} as this is supposed to be true for almost every $\tx_S\in\tX_S$.

Define a subset $\Omega_2\subset\tX_S$ by \begin{equation}\label{RecurLuzinEq}\Omega_2=\Big\{\tx_S: \inf_{T\geq 1}\frac1T\int_{t=0}^{T}\bfone_K(\talpha_S^{t\bfp}.\tx_S)>1-\epsilon_2\Big\}\cap K,\end{equation} where $\bfp$ is as in Proposition \ref{ShearProp}.
By the maximal ergodic theorem, the $\tmu_S$-measure of the first set on the right hand side is greater than $1-C\cdot\frac{\tmu_S(K^c)}{\epsilon}\geq \frac{\epsilon_2^2}2$ where $C\geq 1$ is some absolute constant. Thus $\tmu_S(\Omega_2)>\tmu_S(K)-\frac{C\epsilon_2^2}2>1-\frac{\epsilon_2^3}2-\frac{C\epsilon_2^2}2>1-C\epsilon_2^2$.

Since $\cA_S$ is $\talpha_S^\bfn$-increasing, $\cA_S\supset\talpha_S^{-\bfn}.\cA_S\supset\talpha_S^{-2\bfn}.\cA_S\supset\cdots$. And the sequence $\big\{\bE_{\tmu_S}(\bfone_{\Omega_2^c}|\talpha_S^{-k\bfn}.\cA_S)\big\}_{k=0}^\infty$ forms a martingale. By the martingale maximal inequality,
\begin{equation}\tmu_S\Big(\Big\{\tx_S:\sup_{k\in\bN}\bE_{\tmu_S}(\bfone_{\Omega_2^c}|\talpha_S^{-k\bfn}.\cA_S)\geq\epsilon_2\Big\}\Big)
\leq\frac{\tmu_S(\Omega_2^c)}{\epsilon_2}<C\epsilon_2.\end{equation}
In other words, if we denote
 \begin{equation}\Omega_3=\Big\{\tx_S\in\tX_S:\inf_{k\in\bN}(\tmu_S)_{\tx_S}^{\talpha_S^{-k\bfn}.\cA_S}(\Omega_2)>1-\epsilon_2\Big\},\end{equation} then $\tmu_S(\Omega_3)>1-C\epsilon_2$.
 
We fix $\epsilon_1>0$, which decides the value of $\delta_1$. Choose $0<\epsilon_2<\min(\frac12, \delta_1,\frac{1-\epsilon_1}{2C})$. Since for any $k>0$, $\tmu_S(\talpha_S^{-k\bfn}.\Omega_1)=\tmu_S(\Omega_1)>1-\epsilon_1$,  $$\tmu_S\big((\talpha_S^{-k\bfn}.\Omega_1)\cap\Omega_2\cap\Omega_3\big)>1-\epsilon_1-C\epsilon_2^2-C\epsilon_2>1-\epsilon_1-2C\epsilon_2>0,$$ and in particular there is a sequence of points $\tx_{S,k}\in(\talpha_S^{-k\bfn}.\Omega_1)\cap\Omega_2\cap\Omega_3$.

Consider the probability measure $(\tmu_S)_{\tx_{S,k}}^{\talpha_S^{-k\bfn}.\cA_S}$ on $[\tx_{S,k}]_{\talpha_S^{-k\bfn}.\cA_S}$. With respect to this probability, at least a $\delta_1$-portion of points are out of $V_0.\tx_{S,k}$ by the construction of $\Omega_1$, and at most an $\epsilon_2$-portion is out of $\Omega_2$ as $\tx_S\in\Omega31$. Because $\epsilon_2<\delta_1$, there is $\ty_{S,k}\in [\tx_{S,k}]_{\talpha_S^{-k\bfn}.\cA_S}$ such that $\ty_{S,k}\notin V_0.\tx_{S,k}$ but $\ty_{S,k}\in\Omega_2$.

Because $[\tx_{S,k}]_{\talpha_S^{-k\bfn}.\cA_S}\subset(\alpha^{-k\bfn}.B_1^V).\tx_{S,k}$ and  $\alpha^\bfn$ exponentially expands $V$, we actually proved:
\begin{lemma}\label{ShearPair}There are pairs of points $\tx_{S,k}, \ty_{S,k}\in\Omega_2$ such that $\ty_{S,k}\in (B_{a^{-k}}^V\backslash V_0).\tx_{S,k}$ for some $a>1$.\end{lemma}

\subsection{The H-principle}\label{SecHprin} We are interested in how two points from the same $V$-foliation drift apart from each other under the dynamics of $\talpha_S^\bfp$. Such drifting is controled by Ratner's H-principle as described in the general lemma below.

\begin{lemma}\label{HPrinciple} Suppose $\{U^t\}$ is a unipotent one-parameter subgroup of $\SL(d,\bR)$. Then for all $\epsilon>0$, there is a constant $\kappa\in(0,1)$, such that for all $v\in \bR^d$ with $0<|v|\ll 1$, if $v$ is not fixed by $\{U^t\}$ then there exists $T>0$ and a subset $B\subset [0,T]$ with Lebesgue measure greater than $(1-\epsilon)T$, such that for all $t\in B$, $U^t.v$ can be written as $w+w^\bot$, where:\begin{enumerate}
\item $w$ is a fixed vector of the one-parameter group $\{U^t\}$;
\item $|w|\in[\kappa,\kappa^{-1}]$; and
\item $|w^\bot|\lesssim_d|v|^{\frac1d}$.\end{enumerate}\end{lemma}

The lemma is pretty standard , for completeness we sktech the proof here, following the treatment in \cite{E06}. 

\begin{proof}By properly choosing a basis $U^1$ splitts into a direct sum $\oplus_{l=1}^LU_l^1$ on a splitting $\bR^d=\bigoplus_{l=1}^LF_l$ where each $U_l^1$ is a Jordan block with eigenvalue $1$. 

For each $l$, fix a basis $e_{l,0},\cdots,e_{l,d_l-1}$ of $F_l$ with $U^1.e_{l,i}=e_{l,i}+e_{l,i-1}$ unless $i=0$, in which case $U^1.e_{l,0}=e_{l,0}$.

Decompose $v=\sum_{l,i}v_{l,i}e_{l,i}$. Then $U^t.v=\sum_{l,i}(\sum_{j=i}^{ d_l-1}c_{j-i}v_{l,j}t^{j-i})e_{l,i}$ where $c_n=\frac1{n!}$. Denote the coefficient $\sum_{j=i}^{ d_l-1}c_{j-i}v_{l,j}t^{j-i}$ of $e_{l,i}$ by $f_{l,i}(t)$.

For sufficiently short $v\neq 0$, $|v_{l,i}|$ is short for all $1\leq l\leq L$ and $0\leq i\leq d_l-1$, and at least one $v_{l,i}$ is non-zero. Hence $|c_jv_{l,j}t^j|\ll 1$ at $t=0$. Let $T$ be the first positive value at which some $c_jv_{l,j}t^j$ has absolute value $1$.

It is a priori possible that $T=\infty$. In this case $v_{l,j}=0$ for all pairs $0\leq i<j\leq d_l-1$, which in turn means $v_{l,j}=0$ unless $j=0$. That implies $U^1.v=v$ and $U^t.v=v$ for all $t$, which has been excluded by assumption. Hence $T<\infty$.

Set $w(t)=\sum_lf_{l,0}(t)e_{l,0}$ and $w^\bot(t)=\sum_l\sum_{j=1}^{d_l}f_{l,j}(t)e_{l,j}$. Then $U^t.v=w(t)+w^\bot(t)$. Furthermore, as $e_{l,0}$ is fixed by $\{U^t\}$, so is $w(t)$. 

One estimates first the size of $f_{l,i}(t)$ for $0<i\leq d_l-1$.  For $i\leq j\leq d_l-1$, if $v_{l,j}\neq 0$ then $|c_jv_{l,j}T^j|\leq 1$. Therefore $T\leq c_j^{-\frac 1j}|v_{l,j}|^{-\frac 1j}$. Thus, if $i\geq 1$ and $t\in[0,T]$, then $|c_{j-i}v_{l,j}t^{j-i}|\lesssim_d |v_{l,j}|^{\frac ij}\leq |v_{l,j}|^{\frac 1j}\leq |v|^{\frac 1d}$. So $|f_{l,i}(t)|\lesssim_d|v|^{\frac1d}$,  for $0<i\leq d_l-1$ and $t\in[0,T]$. And therefore, $|w^\bot(t)|\lesssim_d|v|^{\frac1d}$ for $t\in[0,T]$.

It remains to show $1\lesssim_{d,\epsilon}|w(t)|\lesssim_{d,\epsilon}1$ for most values $t\in [0,T]$. Recall that y the choice of $T$, there is an index $l$ such that $|c_jv_{l,j}T^j|\leq 1$ for all $0\leq j\leq d_l-1$ and the equality holds for at least one $j$. The right part of the desired inequality follows trivially for all $t\in[0,T]$.

In order to obtain the left part, for $\kappa>0$ consider $B=\{t\in[0,T]: |f_{l,0}(t)|>\kappa\}$. We need that for some $\kappa=\kappa(\epsilon)>0$, the lebesugue measure of $B$ is at least $(1-\epsilon)T$. 

If the Lebesgue measure of $B$ is bounded by $(1-\epsilon)T$ then there are at least $2d_l-1$ disjoint subintervals in $[0,T]$ of length $\frac{\epsilon T}{2d_l}$ that intersect $B^c$. Order them in the usual way and pick all the ones who get odd numberings. Find a point from the intersection of $B^c$ with each of the $d_l$ subintervals picked. Then the chosen points $t_0<t_1<\cdots<t_{d_l-1}$ satisfy \begin{equation}\label{TestPtsGapEq}t_j-t_{j-1}>\frac{\epsilon T}{2d_l}.\end{equation}

Denote $s_j=\frac{t_j}T$, and write the Vandermonde matrix $A=(a_{jk})_{j,k=0}^{d_l-1}$ where $a_{jk}=s_j^k$. Let $b$ be a row vector with $b_j=c_jv_{l,j}T^j$, then $Ab=\big(f_{l,0}(t_0),\cdots,f_{l,0}(t_{d_l-1})\big)$. So all entries in $Ab$ have absolute values bounded by $\kappa$ since $t_j\notin B$. In other words $\frac{|Ab|}{|b|}\leq\kappa$ for at least one non-zero vector $b$. Furthermore, since $s_j\leq 1$, $\|A\|\lesssim_d 1$. It follows that $|\det A|\lesssim_d\kappa$. 

On the other hand, \begin{equation}\det A=\prod_{0\leq j< j'\leq d_l-1}(s_{j'}-s_j).\end{equation} By \eqref{TestPtsGapEq}, $s_{j'}-s_j>\frac\epsilon{2d_l}$ and thus $\det A\gtrsim_d\epsilon^{\frac{d_l(d_l-1)}2}$. Comparing this with the previous paragraph gives $\kappa\gtrsim_d \epsilon^{\frac{d_l(d_l-1)}2}$. So for $\kappa= C\epsilon^{\frac{d(d-1)}2}$ where $C$ is some constant depending only on $d$, the measure of  $B$ is at least $(1-\epsilon)T$. This completes the proof.\end{proof}

\subsection{The shearing argument} Assume two points $\tx_S, \ty_S\in\Omega_2\subset\tX_S$ are such that $\ty_S=(\exp v).\tx_S$, where $v\in\gov\backslash\gov_0$ and $|v|$ is very small. 

For $\bfp$ as in Proposition \ref{ShearProp}, we compare the difference between $\talpha_S^{t\bfp}.\tx_S$ and $\talpha_S^{t\bfp}.\ty_S$ for properly chosen values of $t$.

One can write $\tx_S=\overline{(\bfeta,\tx)}$ and $\ty_S=\overline{(\bfeta,(\exp v).\tx)}$ for some $\eta\in[0,1)^r$. For each $t$ there is $\bfm_t$ such that $\bfeta+t\bfp-\bfm_t\in[0,1)^r$. Then  $\talpha_S^{t\bfp}$ acts by $\talpha^{\bfm_t}$ on the $\tX$ fiber containing $\tx_S$. And thus $\talpha_S^{t\bfp}.\tx_S$ and $\talpha_S^{t\bfp}.\ty_S$ differ by $\alpha^{\bfm_t}.\exp(v)=\exp(\alpha^{\bfm_t}.v)$.

$\alpha^{\bfm_t}|_\gov=Z^{\bfm_t}U^{\bfm_t-t\bfp}U^{t\bfp}$ as in \S\ref{IsomSec}, and
\begin{equation}\label{ShearDiffEq}\talpha_S^{t\bfp}.\ty_S=\exp(Z^{\bfm_t}U^{\bfm_t-t\bfp}U^{t\bfp}.v).(\talpha_S^{t\bfp}.\tx_S)\end{equation} 

We obtain that:

\begin{lemma}\label{ShearLemma}There exist a compact subset $D\in\gov_0\backslash\{0\}$ and a constant $C>0$, both of which are independent of $\tx_S$ and $\ty_S$, such that given the pair $\tx_S$, $\ty_S$, there exists $t>0$, satisfying that:
\begin{enumerate}
\item $\talpha_S^{t\bfp}.\tx_S\in K$, $\talpha_S^{t\bfp}.\ty_S\in K$;
\item $\talpha_S^{t\bfp}.\ty_S=\exp(\alpha^{\bfm_t}w+u').(\talpha_S^{t\bfp}.\tx_S)$, where $w\in D$, and $|u'|\leq C|v|^\frac1d$;
\end{enumerate}\end{lemma}

\begin{proof}Fix $0<\epsilon<1-2\epsilon_2$, and some $\kappa<1$ that is decided by $\epsilon$, Lemma \ref{HPrinciple} gives a value $T>0$ such that for at least a $(1-\epsilon)$-portion $B$ of $t\in[0,T]$, $U^{t\bfp}.v$ can be written as $w+w'$ where $w\in\gov_0$ with $|w|\in[\kappa,\kappa^{-1}]$, and  $w'\in O_d(|v|^{\frac1d})$.  $T>1$ if $|v|$ is sufficiently small.

On the other hand, because $\tx_S,\ty_S\in\Omega_2$, by \eqref{RecurLuzinEq} there is a $(1-\epsilon_2)$-portion $B_{\tx_S}$ of $t\in[0,T]$ for which $\talpha_S^{t\bfp}.\tx_S\in K$, and the same is true for a $(1-\epsilon_2)$-portion $B_{\ty_S}$ if one replaces $\tx_S$ by $\ty_S$. As $\epsilon+2\epsilon_2<1$, it follows that $B\cap B_{\tx_S}\cap B_{\ty_S}$ is non-empty.

Take a value $t$ from this intersection. Then part (1) of the lemma is satisfied.

Finally, let $u'=Z^{\bfm_t}U^{\bfm_t-t\bfp}w'$. Since $Z^{\bfm_t}U^{\bfm_t-t\bfp}$ and $w$ take value from compact sets respectively in $\GL(\gov)$ (by the proof of Lemma \ref{IsomInv}) and $\gov_0\backslash\{0\}$, and $w'\lesssim_d|v|^{\frac1d}$, part (2) follows.
\end{proof}

\begin{proof}[Proof of Proposition \ref{ShearProp}] Now by Lemma \ref{ShearPair}, we can apply Lemma \ref{ShearLemma} to a sequence of pairs $(\tx_{S,k}, \ty_{S,k})$ with $\ty_{S,k}=(\exp v_k).\tx_{S,k}$ where $v_k\in\gov\backslash\gov_0$ and $|v_k|\rightarrow 0$ as $k$ grows. Let $t_k$, $w_k$ and $u'_k$ be as given by the lemma. 

Then $u'_k\rightarrow 0$ because $|u'_k|\leq C|v_k|^{\frac1d}$. For all $k$, $\alpha^{\bfm_{t_k}}|_{\gov_0}$ is from some given compact set $D_*\subset\Aut(\gov_0)$ by Lemma \ref{IsomCpct}, and $u_k\in D$ for all $k$. By passing to a subsequence we can assume they have limits $A\in D_*$ and $u\in D$.

On the other hand, $\talpha_S^{t_k\bfp}.\tx_{S,k}, \talpha_S^{t_k\bfp}.\ty_{S,k}\in K$ for all $k$, again by taking a subsequence, we assume there are limits $\tx'_S$ and $\ty'_S$. 

Therefore because $\talpha_S^{t_k\bfp}.\ty_{S,k}=\exp(\alpha^{\bfm_{t_k}}|_{\gov_0}w+u'_k).(\talpha_S^{t_k\bfp}.\tx_{S,k})$, by taking limit we have
\begin{equation}\ty'_S=(\exp u).\tx'_S.\end{equation}          

It follows from this, \eqref{SuspLeafEq} and Proposition \ref{LeafMeas} that
\begin{equation}\label{ShearLeafEq0}(\tmu_S)_{\tx'_S}^{V_0}\simeq(\tau_{\exp u})_*(\tmu_S)_{\ty'_S}^{V_0}.\end{equation}

By \eqref{SuspActionEq}, $\alpha^{-\bfm_{t_k}}_*(\tmu_S)_{\talpha_S^{t_k\bfp}.\tx_{S,k}}=(\tmu_S)_{\tx_{S,k}}^{V_0}$ and $\alpha^{-\bfm_{t_k}}_*(\tmu_S)_{\talpha_S^{t_k\bfp}.\ty_{S,k}}=(\tmu_S)_{\ty_{S,k}}^{V_0}$. Recall that $\tx_{S,k},\ty_{S,k}\in\Omega_2\subset K$ and $(\tmu_S)_{\tx_S}^{V_0}$ varies uniformly continuously for $\tx_S\in K$. Thus because $|v_k|\rightarrow 0$, \begin{equation}\dist_{\cM_1^0(V_0)}\Big(\alpha^{-\bfm_{t_k}}_*(\tmu_S)_{\talpha_S^{t_k\bfp}.\tx_{S,k}}^{V_0}, \alpha^{-\bfm_{t_k}}_*(\tmu_S)_{\talpha_S^{t_k\bfp}.\ty_{S,k}}^{V_0}\Big)\rightarrow 0.\end{equation}

Because the group automorphism $\alpha^{\bfm_{t_k}}|_{V_0}$ takes value from a compact subset of $\Aut(V_0)$ for all $t$, it follows that \begin{equation}\label{ShearLeafEq1}\dist_{\cM_1^0(V_0)}\Big((\tmu_S)_{\talpha_S^{t_k\bfp}.\tx_{S,k}}^{V_0}, (\tmu_S)_{\talpha_S^{t_k\bfp}.\ty_{S,k}}^{V_0}\Big)\rightarrow 0.\end{equation}

Moreover, since $\talpha_S^{t_k\bfp}.\tx_{S,k}, \talpha_S^{t_k\bfp}.\ty_{S,k}$ are also in $K$ and converge respectively to $\tx'_S$ and $\ty'_S$, $(\tmu_S)_{\tx'_S}^{V_0}=\lim_{k\rightarrow\infty}(\tmu_S)_{\talpha_S^{t_k\bfp}.\tx_{S,k}}^{V_0}$, and a similar equality holds at $\ty'_S$. Therefore we deduce from \eqref{ShearLeafEq1}
that \begin{equation}\label{ShearLeafEq2}(\tmu_S)_{\tx'_S}^{V_0}=(\tmu_S)_{\ty'_S}^{V_0}.\end{equation}

Comparing \eqref{ShearLeafEq2} with \eqref{ShearLeafEq0}, we see $(\tmu_S)_{\tx'_S}^{V_0}$ is invariant in the $\simeq$ sense by the right translation by $\exp u$. As $\tx'_S\in K$ and $u\in\gov_0\backslash\{0\}$, this contradicts  \eqref{ShearNoTransEq}.

We conclude that the main assumption in this section, namely that $\tmu_\tx^{V_0}$ has no translational invariance property for a $\tmu$-positive portion of $\tx\in\tX$, cannot hold. This establishes Proposition \ref{ShearProp}.\end{proof}

\section{Case of zero entropy fibers}\label{SecEnt0Fiber}

The last case we need to treat is:
 
\begin{proposition}\label{Ent0Fiber}In the settings of \S\ref{SecOutline}, assume in addition that:
\begin{enumerate}
\item $h_\tmu(\talpha^\bfn|\cB_\dtX)=0$ for all $\bfn\in\bZ^r$;
\item $\tmu$ has isometric support along all coarse Lyapunov subgroups.
\end{enumerate}
Then Property \ref{TransInvProp} holds.
\end{proposition}

In the rest of this section, we will work under the assumptions of the proposition.

\subsection{The product structure}\label{SecProd} Making use of the isometric assumption on each coarse Lyapunov subgroup, we may get a product-lemma of Einsiedler-Katok type \cite{EK03}. Especially, this will show that the total entropy of any group element can be decomposed as the sum of entropy contributions from all unstable coarse Lyapunov subgroups.

\begin{definition}\label{Admissible} A pair $(\gov,\gou)$, where $\gov=\gov^{[\chi]}$ is a coarse Lypunov subspace, and $\gou$ is a direct sum of coarse Lyapunov subspaces, is called admissible if:
\begin{enumerate}
\item $\gou$ is a Lie subalgebra;
\item $\gou$ is normalized by $\gov$;
\item There exists $\bfp\in\ker\chi$ such that $\chi'(\bfp)>0$ for all non-trivial coarse Lyapunov subgroups $\gov^{[\chi']}\subset\gou$.
\end{enumerate} \end{definition}

Let $[\chi]$ be the Lyapunov exponent class of $V=V^{[\chi]}$. Since $\gov$ normalizes $\gou$, $VU=UV$ is also a Lie subgroup. 

\begin{lemma}\label{ProdLemma}Suppose $(\gov,\gou)$ is an admissible pair and $\tmu$ is a $\talpha$-invariant probability measure on $\tX$ that has isometric support along $V$.

Then there is a subset $\tX'\subset\tX$ with $\tmu(\tX')=1$, such that:\begin{enumerate}
\item If $u\in U$, and $\tx$ and $u.\tx$ both belong to $\tX'$, then $\tmu_\tx^V\simeq\tmu_{u.\tx}^V$;
\item If $\tx\in\tX'$, then $\tmu_\tx^{VU}\simeq\tmu_\tx^V\tmu_\tx^U$, by which we mean the pushforward of $\tmu_\tx^V\times \tmu_\tx^U$ by $(v,u)\mapsto vu$.
\item If $\bfn\in\bZ^r$ satisfies $V, U\subset G_\bfn^u$, then $h_\tmu(\talpha^\bfn|VU)=h_\tmu(\talpha^\bfn,V)+h_\tmu(\talpha^\bfn,U)$.
\end{enumerate}\end{lemma}

\begin{proof}(1) Passing to the suspension $\tX_S$, let $\tmu_S$ be defined by \eqref{SuspMeasEq}. We show that there is a full measure subset $\tX'_S\subset\tX_S$, such that if  $\tx_S, u.\tx_S$ both lie in $\tX'_S$ for some $u\in U$, then $(\tmu_S)_{\tx_S}^{V_0}=(\tmu_S)_{u.\tx_S}^{V_0}$. The proof is similar to that of Proposition \ref{PropRk1Fiber}.

Fix $\epsilon>0$, there is a compact Luzin set $K\subset\tX_S$ such that $(\tmu_S)_{\tx_S}^V$ varies uniformly continuously on $K$ and $\tmu_S(K)>1-\epsilon^2$. Let $\Omega=\big\{\tx_S: \inf_{T\geq 1}\frac1T\int_{t=0}^{T}\bfone_K(\talpha_S^{-t\bfp}.\tx_S)>1-\epsilon\big\}$. By the isometric support hypothesis on $\tmu_\tx^V$ for $\tmu$-almost every $\tx$,  we may modify $\Omega$ by removing a null set, such that for every $\tx_S\in\Omega$, $(\tmu_S)_{\tx_S}^V$ is supported on the $\talpha_S^\bfp$-isometric subgroup $V_0\subset V$.  

$\tmu_S\Big(\big\{\tx_S: \sup_{T\geq 1}\frac1T\int_{t=0}^{T}\bfone_{K^c}(\talpha_S^{-t\bfp}.\tx_S)\geq \epsilon\big\}\Big)\leq C\cdot\frac{\tmu_S(K^c)}\epsilon<C\epsilon$ for an absolute constant $C$ by the maximal ergodic theorem. So $\tmu_S(\Omega)\geq 1-C\epsilon$.

Suppose $\tx_S$ and $u.\tx_S$ are both in $\Omega$, then for all $T$, there is a $1-2C\epsilon$ portion of $t\in[0,T]$ such that both $\talpha_S^{-t\bfp}.\tx_S$ and $\talpha_S^{-t\bfp}.(u.\tx_S)$ are in $K$. In particular, there is some $t\geq (1-2C\epsilon)T$ for which this is true. If $\epsilon<\frac1{2C}$, by letting $T\rightarrow\infty$ there are arbitrarily large values of $t$ that satisfy this requirement.

Choose a sequence of such $t_k$'s that grow to $\infty$. By discussions in \S\ref{SecSuspension}, for each $t_k$, there is an integer vector $\bfm_{-t_k}\in-t_k\bfp+(-1,1)^r$, such that $\talpha^{-t_k\bfp}$ sends the $\tX$ fiber containing $\tx_S$ to the one containing $\talpha_S^{-t_k\bfp}.\tx_S$ by $\talpha^{\bfm_{-t_k}}$. Especially,  $\talpha_S^{-t_k\bfp}.(u.\tx_S)=(\alpha^{\bfm_{-t_k}}.u).(\talpha_S^{-t_k\bfp}.\tx_S)$.  It follows from the remark at the end of \S\ref{SecSuspension} that, for all coarse Lyapunov exponents $\chi$ on $\gou$, $\chi(\bfm_{-t_k})$ is of bounded distance from $-\bft_k\chi(\bfp)$, and thus tends to $-\infty$ as $\chi(\bfp)<0$. So for $u\in U$, $\alpha^{\bfm_{-t_k}}.u\rightarrow e$ as $k\rightarrow\infty$, and \begin{equation}\label{ProdLemmaEq1}\dist\big(\talpha_S^{-t_k\bfp}.\tx_S,\talpha_S^{-t_k\bfp}.(u.\tx_S)\big)\rightarrow0.\end{equation}

By \eqref{SuspActionEq},
\begin{equation}\label{ProdLemmaEq2}\alpha^{\bfm_{-t_k}}_*(\tmu_S)_{\tx_S}^V\simeq(\tmu_S)_{\talpha_S^{-t\bfp}.\tx_S}^V;\end{equation}  
\begin{equation}\label{ProdLemmaEq3}\alpha^{\bfm_{-t_k}}_*(\tmu_S)_{u.\tx_S}^V\simeq(\tmu_S)_{\talpha_S^{-t\bfp}.(u.\tx_S)}^V.\end{equation}

Since $\talpha_S^{-t\bfp}.\tx_S$ and $\talpha_S^{-t\bfp}.(u.\tx_S)$ belong to the set $K$, where the leafwise measure along $V$ is a uniform continuous function, \eqref{ProdLemmaEq1}, \eqref{ProdLemmaEq2} and \eqref{ProdLemmaEq3} together imply that
\begin{equation}\label{ProdLemmaEq4}\dist_{\cM_1^0(V)}\big(\alpha^{\bfm_{-t_k}}_*(\tmu_S)_{\tx_S}^V,\alpha^{\bfm_{-t_k}}_*(\tmu_S)_{u.\tx_S}^V\big)\rightarrow 0.\end{equation}

Whereas, $(\tmu_S)_{\tx_S}^V$ and $(\tmu_S)_{\tx_S}^V$ are supported on $V_0$. Similar to what happens in the proof of Proposition \ref{ShearProp}, due to Lemma \ref{IsomCpct}, $\alpha^{\bfm_{-t_k}}|_{V_0}$ is from a fixed compact set of group automorphisms and may even be assumed to converge to some limit $A\subset\Aut(V_0)$. Then \eqref{ProdLemmaEq4} shows \begin{equation}\label{ProdLemmaEq5}(\tmu_S)_{\tx_S}^V\simeq(\tmu_S)_{u.\tx_S}^V,\text{ whenever }u\in U, \tx_S, u.\tx_S\in\Omega.\end{equation}

By taking a sequence $\epsilon_n\rightarrow 0$, there is a sequence $\{\Omega_n\}$ with $\tmu_S(\Omega_n)\rightarrow 1$ satisfying \eqref{ProdLemmaEq5}.   $\tX'_S=\bigcup_{n=1}^\infty\Omega_n$ gives us the subset we want. 

By the construction \eqref{SuspMeasEq}, for almost every $\bfeta\in[0,1)^r$, the intersection of $\tX'_S$ with the $\tX$ fiber $\{\overline{(\bfeta,\tx)}:\tx\in\tX\}$ has full measure with respect to $\tmu$. Part (1) of the lemma follows immediately, thanks to \eqref{SuspLeafEq}.

(2) The second part of the Lemma is \cite{EL10}*{Cor. 8.8}. 

(3) Define a neighborhood $B$ of identity in $VU$ by $B=B_1^V\cdot B_1^U$. Then \begin{equation}\alpha^{-k\bfn}.B=\big(\alpha^{-k\bfn}.B_1^V\big)\big(\alpha^{-k\bfn}.B_1^U\big).\end{equation}
Thanks to the product structure in (2), it follows that \begin{equation}\label{ProdLemmaEq31}\lim_{k\rightarrow\infty}-\frac1k\log\tmu_{\tx}^{VU}(\talpha^{-k\bfn}.B)=h_\tmu(\talpha^\bfn,V)+h_\tmu(\talpha^\bfn,U).\end{equation} Observe that because $\alpha^\bfn$ has positive Lyapunov exponents on $VU$, there exists a constant $s\in\bN$ such that $\talpha^{-s\bfn}.B\subset B_1^{VU}\subset \talpha^{s\bfn}.B$ and thus \begin{equation}\label{ProdLemmaEq32}\begin{split}\lim_{k\rightarrow\infty}-\frac1k\log\tmu_{\tx}^{VU}(\talpha^{-k\bfn}.B)=&\lim_{k\rightarrow\infty}-\frac1k\log\tmu_{\tx}^{VU}\Big(\talpha^{-k\bfn}.B_1^{VU}\Big)\\
=&h_\tmu(\talpha^\bfn, VU).\end{split}\end{equation}
Part (3) is deduced from \eqref{ProdLemmaEq31}, \eqref{ProdLemmaEq32}.
\end{proof}

In order to apply such product properties, we have the next lemma which produces admissible pairs.

\begin{lemma}\label{AdmiDecomp}Suppose $\gow\subset\gog$ is a Lie subalgebra that can be written as the direct sum of coarse Lyapunov subalgebras. If there is $\bfp\in\bR^r$ such that $\chi'(\bfp)>0$ for all $\gog^{[\chi']}\subset\gow$, then $\gow$ can be decomposed as $\gov\oplus\gou$ where $(\gov,\gou)$ is an admissible pair.\end{lemma}

\begin{proof} For each $[\chi']$ appearing in the decomposition of $\gow$, $\bfp$ lies on the half-space $\{\bfp:\chi'(\bfp)>0\}$. The intersection of all these half-spaces is a open convex simplicial cone with finitely many faces, and $\bfp$ is in the interior of this cone. Fix an arbitrary face of the cone, which is itself a cone inside $\ker\chi$ for some $\gog^{[\chi]}\subset\gow$. Let $\bfp'$ be an arbitrary vector from the interior of this face, then the line segment $\bfp\rightarrow\bfp'$ does not intersect any other $\ker\chi'$. Hence for $\gog^{[\chi']}\subset\gow$, $\chi'(\bfp')$ is positive as $\chi(\bfp)$ is, unless $\chi'=\chi$. Let $\gov=\gov^{[\chi]}$ and $\gou=\bigoplus_{\substack{\gog^{[\chi']}\subset\gow\\ \chi'\neq\chi}}\gog^{[\chi']}$.

Clearly $\gow=\gov\oplus\gou$. We now check that $(\gov,\gou)$ is admissible. Suppose $\gov^{[\chi_1]}$ and $\gov^{[\chi_2]}$ are respectively two coarse Lypunov subspaces inside $\gow$ and $\gou$. Then $\chi_1(\bfp')\geq 0$ and $\chi_2(\bfp')>0$. By Proposition \ref{LyaDecomp}.(3), $[\gov^{[\chi_1]},\gov^{[\chi_2]}]$, if non-trivial, is contained in $\gov^{[\chi_1+\chi_2]}$. Note $\chi_1+\chi_2$ has positive value at $\bfp'$, and  $[\gov^{[\chi_1]},\gov^{[\chi_2]}]\subset\gow$. Thus $[\gov^{[\chi_1]},\gov^{[\chi_2]}]$ must be inside $\gou$. This verifies both (1) and (2) in Definition \ref{Admissible}. Part (3) holds with respect to $\bfp'$.\end{proof}

\begin{corollary}\label{BigProdLemma}For all $\bfn\in\bZ^r$, there is an ordering $V^{[\chi_1]}, \cdots, V^{[\chi_l]}$ of all coarse Lypunov subgroups inside $G_\bfn^u$,  with $G=V^{[\chi_1]}V^{[\chi_2]}\cdots V^{[\chi_l]}$, such that:
\begin{enumerate}
\item  For each $j$, $V^{[\chi_{j+1}]}\cdots V^{[\chi_l]}$ is a Lie subgroup and normalizes $V^{[\chi_j]}$;
\item  There is a subset $\tX'\subset\tX$ with $\tmu(\tX')=1$, such that for all $\tx\in\tX'$, $\tmu_\tx^{G_\bfn^u}$ is proportional to $\tmu_\tx^{V^{[\chi_1]}}\tmu_\tx^{V^{[\chi_2]}}\cdots\tmu_\tx^{V^{[\chi_l]}}$, by which we mean the pushforward of $\tmu_\tx^{V^{[\chi_1]}}\times\cdots\times\tmu_\tx^{V^{[\chi_l]}}$ by $(v_1,\cdots,v_l)\mapsto v_1\cdots v_l$;
\item $h_\tmu(\talpha^\bfn|\cB_Y)=\sum_{i=1}^l h_\tmu(\talpha^\bfn, V^{[\chi_i]})$.
\end{enumerate} \end{corollary}
\begin{proof} Start with $\gou_0=\gog_\bfn^u$ and repeat applying Lemma \ref{AdmiDecomp}. We get admissible splittings $\gou_0=\gov^{[\chi_1]}\oplus\gou_1$, $\gou_1=\gov^{[\chi_2]}\oplus \gou_2$, $\cdots$. Let $U_j=\exp\gou_j$. Then $U_j=V^{[\chi_{j+1}]}\cdots V^{[\chi_l]}$ and normalizes $V^{[\chi_j]}$, which proves (1). By Lemma \ref{ProdLemma}, $\tmu_\tx^{U_j}\propto\tmu_\tx^{V^{[\chi_{j+1}]}}\tmu_\tx^{U_{j+1}}$. Part (2) follows immediately by iterating this. Given Proposition \ref{EntContri}.(2), Part (3) is a repeated application of Lemma \ref{ProdLemma}.(3).
\end{proof}

\subsection{Fibrating the action}

In this part we construct a new measure-preserving system, which is essentially a fibration of $(\tX, \talpha, \cB_\tX)$ with respect to values of the map $\tx\mapsto\tmu_\tx^V$ for some coarse Lyapunov subgroup $V$.

By assumption in Theorem \ref{MeasureInductive} $\tX$ has positive entropy for some $\bfn$ unless $X$ is trivial.  By Corollary \ref{BigProdLemma}, this in particular implies there is at least one non-trivial coarse Lyapunov subgroup $V^{[\chi]}$ with positive entropy contribution $h_\tmu(\talpha^\bfn,V^{[\chi]})$ for at least one $\bfn\in\bZ^r$. 

Suppose $[\chi']$ is another homothety class of exponents appearing in the coarse Lyapunov decomposition such that $\chi'$ is not negatively proportional to $\chi$. Then $\ker\chi'\cap\ker\chi$ is a proper subspace in $\ker\chi$. Since there are only finitely many choices of $[\chi']$, there is $\bfp\in\ker\chi$ such that $\chi'(\bfp)\neq 0$ for all $[\chi']$'s that are non-proportional to $[\chi]$.

Let $\gou=\bigoplus_{[\chi']:\chi'(\bfp)>0}\gov^{[\chi']}$ and $\gos=\bigoplus_{[\chi']:\chi'(\bfp)<0}\gov^{[\chi']}$. Then by Proposition \ref{LyaDecomp}.(3), $\gou$ and $\gos$ are Lie subalgebras. 

\begin{lemma}\label{AlmostFreeze}Given $\bfp\in\bR^r$ as above, let $\theta>0$ be any constant less than $\min \chi'(\frac{\bfp}{\|\bfp\|})$ where the minimum is taken over all $\chi'$'s with $V^{\chi'}$ appearing in the Lyapunov decomposition and $\chi'(\bfp)>0$. 

Then for any $\epsilon>0$, one can pick $\bfm\in\bZ^r$ such that \begin{enumerate}
\item $0<\chi'(\bfm)<\epsilon|\bfm|$ for all non-trival Lyapunov subpace $\gov^{\chi'}\subset V^{[\chi]}$;
\item $\chi'(\bfm)>\theta|\bfm|$ for all non-trival Lyapunov subspace $\gov^{\chi'}\subset \gou$;
\item $\chi'(\bfm)<0$ for all non-trival Lyapunov subspace  $\gov^{\chi'}\subset \gos$.
\end{enumerate}
\end{lemma}
\begin{proof}It is enough to pick $\bfm$ sufficiently close to the line $\bR\bfp$, on the correct side of $\ker\chi$. \end{proof}

We fix the choice of $\bfm$ from now on in this section.

Thanks to the choice of $\bfp$, any Lyapunov subspace $V^{\chi'}$ is either in $\gou$ or $\gos$, or such that $\chi'$ is proportional to $\chi$. It is thus easy to see that $\gog_\bfm^u=\gov^{[\chi]}\oplus\gou$. Furthermore, Lemma \ref{LyaDecomp}.(3) implies that $[\gov^{[\chi]},\gou]\subset\gou$. 

For simplicity, write $\gov=\gov^{[\chi]}$, $V=V^{[\chi]}$ and $U=\exp\gou$. Let $V_0\subset V$ be the $\talpha_S^\bfp$-isometric subgroup and $\gov_0$ be the corresponding Lie subalgebra. 

Then the unstable subgroup $G_\bfm^u=VU$ and  $V $ normalizes $U$. Moreover $(\gov,\gou)$ is an admissible pair via $\bfp$.  By Lemma \ref{ProdLemma}, there is a subset $\tX'\subset\tX$ with $\tmu(\tX')=1$, such that if $u\in U$ and $\tx, u.\tx\in\tX'$, then $\tmu_\tx^V=\tmu_{u.\tx}^V$.

By replacing $\tX'$ with $\bigcap_{\bfn\in\bZ^r}\talpha^\bfn_*\tX'$ one may assume $\tX'$ is $\talpha$-invariant.

Let $\cU'$ be the smallest $\sigma$-algebra that renders the map $\tx\mapsto\tmu_\tx^V$ from $\tX$ to $\cM_1^0(V)$ measurable. Keep in mind that, points in $\cM_1^0(V)$ are equivalence classes of positive proportional measures.

\begin{corollary}\label{LeafTribu}There is a countably generated, $\talpha$-invariant $\sigma$-algebra $\cU$ on $\tX$, such that:\begin{enumerate}
\item $\cU$ is equivalent to $\cU'$ modulo $\mu$-null sets;
\item Each $\cU$-atom is a union of full $U$-leaves. 
\end{enumerate} \end{corollary}

\begin{proof}
$\cU'$ is countably generated since $\tx\mapsto\tmu_\tx^V$ maps into the compact metric space $\cM_1^0(V)$. By Lemma \ref{LeafAction}, $\cU'$ is $\talpha$-invariant. 

By the lemma, for all $\tx$ and $A\in\cU'$, $V.\tx\cap\tX'$ completely lies in either $A$ or $A^c$. 

One can take a countable generating family $A'_i$ of $\cU'$, and modify each $A'_i$ to get a new set $A_i=V.(A'_i\cap\tX')$. Then by the property above and the fact that $\tX'$ has full measure, $\tmu(A_i)=\tmu(A'_i\cap\tX')=\tmu(A'_i)$. In addition, this modification preserves the set-theoretic relations between the sets. 

It suffices to take $\cU$ to be the $\sigma$-algebra generated by the $A_i$'s.
\end{proof}

Let $\beta_\cU:\bZ^r\curvearrowright Y_\cU$ be a continuous realization of the factor with respect to $\cU$, that is, a compact metric space equipped with a continuous $\bZ^r$-action such that $(Y_\cU,\beta_\cU, \cB_{Y_{\cU}})$ is isomorphic to the dynamical system $(\tX, \talpha, \cU)$.  Denote by $\phi:\tX\mapsto Y_\cU$ the corresponding conjugacy.

Write $\tX_\cU=\tX\times Y_\cU=X\times Y\times Y_\cU$ and $\talpha_\cU=\talpha\times\beta_\cU=\alpha\times\beta\times\beta_\cU$. Remark that, while working with the action $\tX_\cU$ we are still within the product scheme introduced at the beginning of \S\ref{SecInductive}, with $(Y,\beta)$ being replaced with $(Y\times Y_\cU,\beta\times\beta_\cU)$.

The map $\Phi:\tX\mapsto\tX_\cU$, $\Phi(\tx)=\big(\tx,\phi(\tx)\big)$ interwines the actions $\talpha$ and $\talpha_\cU$. Let $\tmu_\cU$ be the pushforward of $\tmu$ by $\Phi$, which is invariant and ergodic under $\talpha_\cU$. Constructions in \S\ref{SecPrelim} such as leafwise measures all work for $\tmu_\cU$ as well.

Notice that $\Phi$ is injective and is hence just an isomorphism between $\tX$ and $\Phi(\tX)$.

Our first observation is that, since $\cU$ does not break $U$-leaves. The leafwise structures along $U$ are preserved under the map $\Phi$.

\begin{lemma}\label{SameLeaf}For $\tmu$-almost every $\tx$, $\tmu_\tx^U=(\tmu_\cU)_{\Phi(\tx)}^U.$\end{lemma}

\begin{proof}Let $\cA$ be a $U$-subordinate $\sigma$-algebra on $\tX$, which exists by Lemma \ref{SubordExist}. Let $\Phi_*(\cA)$ be the $\sigma$-algebra generated by the images of $\cA$-measurable sets and the complement $\tX_\cU\backslash\Phi(\tX)$.

If two $\tmu$-generic points $\tx,\ty$ are in the same $\cA$-atom, then $\ty=u.\tx$ for some $u$ from a bounded neighborhood of identity in $U$, and hence they are in the same $\cU$-atom. Thus $\phi(\tx)=\phi(\ty)$ and \begin{equation}\label{SameLeafEq1}\Phi(\ty)=\big(u.\tx, \phi(\tx)\big)=u.\Phi(\tx).\end{equation} Therefore, for $\tmu$-almost every $\tx$, the image of $[\tx]_\cA$ by $\Phi$, which is just $[\Phi(\tx)]_{\Phi_*\cA}$, is still a bounded neighborhood in $U$-leaves. Thus $\Phi_*\cA$ is $U$-subordinate with respect to $\tmu_\cU$. (Recall that by Definition \ref{Subord}, being subordinate is an almost everywhere property with respect to a measure.) 

With this, the lemma is deduced from Proposition \ref{LeafMeas}.(3) and relation \eqref{SameLeafEq1}.
\end{proof}

One can also comment about the leafwise structure of $\tmu$ along $V$.

\begin{lemma}$\tmu_\cU$ has isometric support along $V$.\end{lemma}
\begin{proof}Let $\cA$ be a $V$-subordinate $\sigma$-algebra on $\tX$, instead. Then $\cA_\cU=\cA\times\cB_{Y_\cU}$ is a $V$-subordinate, $\talpha_\cU$-increasing $\sigma$-algebra on $\tX_\cU$. As $\Phi=\Id\times\phi$, $\Phi^{-1}(\cA_\cU)=\cA\vee\phi^{-1}(\cB_{Y_\cU})=\cA\vee\cU$. So for $\tmu$-almost every $\tx$,  $(\tmu_\cU)_{\Phi(\tx)}^{\cA_\cU}=(\Phi_*\tmu)_{\Phi(\tx)}^{\cA_\cU}=\Phi_*(\tmu_\tx^{\cA\vee\cU})$.

Because $\tmu$ has isometric support along $V$, $\supp\tmu_\tx^{\cA\vee\cU}\subset\supp\tmu_\tx^\cA\subset V_0.\tx$. And therefore $\supp (\tmu_\cU)_{\Phi(\tx)}^{\cA_\cU}\subset \Phi(V_0.\tx)$. Thus, every point in $\supp (\tmu_\cU)_{\Phi(\tx)}^{\cA_\cU}$ can be written as $\big(\ty,\phi(\ty)\big)$ where $\ty\in V_0.\tx$. Since $\cA_\cU$ refines $\cB_{Y_\cU}$, $\phi(\ty)$ must coincide with $\phi(\tx)$ in $Y_\cU$. Thus $\big(\ty,\phi(\ty)\big)=\big(v.\tx, \phi(\tx)\big)$ with $v\in V_0$, and therefore $\supp (\tmu_\cU)_{\Phi(\tx)}^{\cA_\cU}\subset V_0.\Phi(\tx)$. 

So for $\tmu_\cU$-almost every point $\tx_\cU$,  $\supp (\tmu_\cU)_{\tx_\cU}^{\cA_\cU}\subset V_0.\tx_\cU$ as $\tx_\cU$ can almost surely be written as $\Phi(\tx)$ for some $\tmu$-generic point $\tx$. By Proposition \ref{LeafMeas}.(3), $(\tmu_\cU)_{\tx_\cU}^V\big|_{[\tx_\cU]_{\cA_\cU}}$ is almost surely supported on $V_0.\tx_\cU$. By the same recurrence argument involving coarser and coarser $V$-subordinate $\sigma$-algebras $\alpha^{k\bfn}.\cA$ as in the proof of Lemma \ref{LeafAction}, one can deduce that $(\tmu_\cU)_{\tx_\cU}^V$ is completely supported on $V_0.\tx_\cU$ for almost every $\tx_\cU$.\end{proof}

\subsection{Entropy comparison}

Consider all non-trivial algebraic factor actions $\alpha|_\Sigma:\Sigma\curvearrowright X_0$ of the action $\alpha|_\Sigma$ on $\dX$, where $\Sigma$ is any finite index subgroup of $\bZ^r$.  Let $\tX_0=X_0\times Y$ and $\alpha|_\Sigma$, $\talpha|_\Sigma$ be defined accordingly on $X_0$ and $\tX_0$.  Let $\tmu_0$ be the projection of $\dtmu$ to $\tX_0$, which also descends from the measure $\tmu$ on $\tX$.

Without loss of generality, we may assume $\bfm\in\Sigma$ by replacing it with a positive integer multiple of itself if necessary, this would not effect the properties required in Lemma \ref{AlmostFreeze}.

The aim of this part is to compare entropies of $\talpha^\bfm$ and $\talpha_\cU^\bfm$. For this purpose we need $\talpha|_\Sigma$-ergodic measures to work with.

Since $\Sigma$ has finite index, $\tmu$ has a finite ergodic decomposition $\tmu=\frac1N\sum_{i=1}^N\tmu^i$ where the $\tmu^i$'s are distinct, ergodic under $\talpha|_\Sigma$ and transitively permuted among themselves by $\talpha$. Denote $\tmu_\cU^i=\Phi_*\tmu^i$. Let $\dtmu^i$ and $\dtmu_0^i$ respectively be the projections of $\tmu^i$ to $\dtX$ and $\tX_0$, and let $\tmu_{0,\cU}^i$ be the projection of $\tmu_\cU^i$ to $\tX_{0,\cU}=\tX_0\times Y_\cU=X_0\times Y\times Y_\cU$.

\begin{lemma}\label{DXUnifI}For each $i$, there is an ergodic $\beta|_\Sigma$-invariant probability measure $\nu^i$ on $Y$ such that $\dtmu^i=\rmm_\dX\times\nu^i$ and $\tmu_0^i=\rmm_{X_0}\times\nu^i$.\end{lemma}
\begin{proof}As $\dtmu$ is equal to $\rmm_\dX\times\nu$ by \eqref{InductiveEq}, the statement about $\dtmu^i$ is just Corollary \ref{ProdErgo}.(2).
Projecting to $\tX_0$, the second claim follows.\end{proof}

Since each $\tmu^i$ has positive weight in the decomposition of $\tmu$, every $\tmu$-null set is also a $\tmu^i$-null set. Hence the $\sigma$-algebras $\cA$ in the proof of Lemma \ref{SameLeaf} is also $U$-subordinate with respect to $\tmu^i$, and similarly $\Phi_*\cA$ is $U$-subordinate to $\tmu_\cU^i$. So, the same argument yields

\begin{lemma}\label{SameLeafI}For $\tmu^i$-almost every $\tx$, $(\tmu^i)_\tx^U=(\tmu_\cU^i)_{\Phi(\tx)}^U$.\end{lemma} 

$\tmu^i$ inherits the isometric property of $\tmu$ as well.

\begin{lemma}\label{IsomI}$\tmu^i$ has isometric support along $V$.\end{lemma}
\begin{proof}Let $\cA$ be an arbitrary $V$-subordinate $\sigma$-algebra with respect to $\tmu$, then it is also $V$-subordinate with respect to $\tmu^i$. And for $\tmu^i$-almost every $\tx$, $(\tmu^i)_\tx^\cA$ is bounded by $N\tmu_\tx^\cA$.  By Proposition \ref{LeafMeas}.(3), this implies that for $\tmu^i$-almost every $\tx$, $\supp(\tmu^i)_\tx^V\subset\supp\tmu_\tx^V$, which is contained in $V_0.\tx$\end{proof}

Using the lemmas above, we are ready to link entropies in the two different spaces $\tX$ and $\tX_\cU$.

\begin{lemma}\label{EntLoss}Let $\epsilon>0$ and $\bfm\in\Sigma$ be as in Lemma \ref{AlmostFreeze}. Then $$h_{\tmu_\cU^i}(\talpha_\cU^\bfm|\cB_{Y\times Y_\cU})>h_{\tmu^i}(\talpha^\bfm|\cB_Y)-(\dim V)\epsilon|\bfm|.$$ \end{lemma}

\begin{proof}By Proposition \ref{EntContri}.(2) and Lemma \ref{SameLeafI}, $h_{\tmu_\cU^i}(\talpha_\cU^\bfm, U)=h_{\tmu^i}(\talpha^\bfm, U)$. It follows that  \begin{equation}\label{EntLossEq0}h_{\tmu_\cU^i}(\talpha_\cU^\bfm|\cB_{Y\times Y_\cU})\geq h_{\tmu^i}(\talpha^\bfm, U).\end{equation} 

On the other hand, as $\tmu^i$ has isometric support along $V$ and $(\gov,\gou)$ is admissible. By Lemma \ref{ProdLemma}.(3),   \begin{equation}\label{EntLossEq1}h_{\tmu^i}(\talpha^\bfm|\cB_Y)=h_{\tmu^i}(\talpha^\bfm, V)+h_{\tmu^i}(\talpha^\bfm, U).\end{equation}

Comparing \eqref{EntLossEq0} and \eqref{EntLossEq1}, it is clear that to prove the lemma one only needs \begin{equation}\label{EntLossEq2}h_{\tmu^i}(\talpha^\bfm, V)<(\dim V)\epsilon|\bfm|.\end{equation} 

Using Proposition \ref{EntContri}.(3), we see $h_\tmu(\talpha^\bfm, V)\leq\log\big|\det\alpha^\bfm|_\gov\big|$. By the choice of $\bfm$ in Lemma \ref{AlmostFreeze}, all eigenvalues of $\alpha^\bfm|_\gov$ are less than $\exp(\epsilon|\bfm|)$. Hence  \eqref{EntLossEq2} is verified. \end{proof}

\begin{lemma}\label{EntBase}If $\bfm\in\Sigma$ and $\theta>0$ are as in Lemma \ref{AlmostFreeze}, then $h_{\tmu_0^i}(\talpha^\bfm|\cB_Y)\geq\theta|\bfm|$.\end{lemma}
\begin{proof}
$X_0$ inherits from $X$ the property that, for any finite index subgroup $\Sigma'$, there is no $\alpha|_{\Sigma'}$-equivariant algebraic factor on which the action is virtually cyclic. In the Lyapunov decomposition on $X_0$, by Corollary \ref{NonPropLya}, not all the Lyapunov exponents are proportional to $\chi$. Thanks to the choice of $\bfp$, these exponents do not all vanish at $\bfp$. Furthermore, since the action is by nilmanifold automorphisms, which are volume preserving, the sum of the exponents is identically zero. Hence there must be exponents of both positive and negative values at $\bfp$. In particular, there is at least one exponent $\chi'$ from $\gou$ that also appears in the Lyapunov decomposition on $X_0$. By Lemma \ref{AlmostFreeze},  $\chi'(\bfm)>\theta|\bfm|$.

$h_{\tmu_0^i}(\talpha^\bfm|\cB_Y)$ is at least  the entropy contribution $h_{\tmu_0^i}(\talpha^\bfm, G_0^{\chi'})$ of the Lyapunov subgroup $V_0^{\chi'}\subset G_0$. However, by Lemma \ref{DXUnifI}, $\tmu_0^i$ is the product of the Lebesgue measure $\rmm_{X_0}$ on $X_0$ with some measure $\nu^i$ on $Y$. In particular, the leafwise measure $(\tmu_0^i)_{\tx_0}^{ V_0^{\chi'}}$ is the Haar measure for almost every $\tx_0\in\tX_0$. It follows that  $h_{\tmu_0^i}(\talpha^\bfm, V_0^{\chi'})\geq \dim V_0^{\chi'}\cdot\chi'(\bfm)\geq\theta|\bfm|$. The proof is complete.\end{proof}

\begin{corollary}\label{EntBasePositive}If $(\dim V)\epsilon<\theta$ then $h_{\tmu_{0,\cU}^i}(\talpha_\cU^\bfm|\cB_{Y\times Y_\cU})>0$.\end{corollary}
\begin{proof} We have the following diagram that interwines actions $\talpha|_\Sigma$ on the left column and $\talpha_\cU|_\Sigma$ on the right column:
\begin{equation}\begin{CD}
(\tX,\tmu^i) @>>> (\tX_\cU,\tmu_\cU^i)\\
@VVV @VVV\\
(\tX_0,\tmu_0^i) @. (\tX_{0,\cU},\tmu_{0,\cU}^i)
\end{CD}\end{equation}

If we identify factor maps between dynamical systems with inclusions between action-invariant $\sigma$-algebras, then the diagram is equivalent to:
\begin{equation}\begin{CD}
(\tX,\cB_X\vee\cB_Y,\tmu^i) @>>> (\tX, \cB_X\vee\cB_Y\vee\cU,\tmu^i)\\
@VVV @VVV\\
(\tX,\cB_{X_0}\vee\cB_Y,\tmu^i) @. (\tX,\cB_{X_0}\vee\cB_Y\vee\cU,\tmu^i)
\end{CD}\end{equation}
Here by $\cB_{X_0}$, $\cB_X$ and $\cB_Y$ are all viewed as $\sigma$-subalgebras of $\cB_\tX=\cB_X\vee\cB_Y$. Note that in the upper-right corner of the diagram, $\cB_X\vee\cB_Y\vee\cU$ is actually just $\cB_X\times\cB_Y$, which corresponds to the fact that $\Phi$ is an isomorphism between $\tmu$ and $\tmu_\cU$.

Using Abramov-Rokhlin formula, what was proved in Lemmas \ref{EntLoss} and \ref{EntBase} can be respectively written as:
\begin{equation}h_{\tmu^i}(\talpha^\bfm|\cB_Y\vee\cU)>h_{\tmu^i}(\talpha^\bfm|\cB_Y)-(\dim V)\epsilon|\bfm|;\end{equation}
\begin{equation}h_{\tmu^i}(\talpha^\bfm|\cB_Y)-h_{\tmu^i}(\talpha^\bfm|\cB_{X_0}\vee\cB_Y)>\theta|\bfm|.\end{equation}

Adding the two inequalities yields
\begin{equation}h_{\tmu^i}(\talpha^\bfm|\cB_Y\vee\cU)-h_{\tmu^i}(\talpha^\bfm|\cB_{X_0}\vee\cB_Y)>\big(\theta-(\dim V)\epsilon\big)|\bfm|>0.\end{equation}

Again by Abramov-Rokhlin formula,
\begin{equation}\begin{split}h_{\tmu_0^i}(\talpha_\cU^\bfm|\cB_{Y\times Y_\cU})
=&h_{\tmu^i}(\talpha^\bfm|\cB_Y\vee \cU)-h_{\tmu^i}(\talpha^\bfm|\cB_{X_0}\vee\cB_Y\vee\cU)\\
\geq&h_{\tmu^i}(\talpha^\bfm|\cB_Y\vee\cU)-h_{\tmu^i}(\talpha^\bfm|\cB_{X_0}\vee\cB_Y)>0,\end{split}\end{equation} which is the corollary.\end{proof}

\subsection{Getting invariance}

We can now conclude the proof of Proposition \ref{Ent0Fiber}.

\begin{proof}[Proof of Proposition \ref{Ent0Fiber}]

Recall that, in Lemma \ref{AlmostFreeze} the constant $\theta$ is independent of the choice of $\epsilon$. Hence by choosing $\epsilon<\frac\theta{\dim X}$, Corollary \ref{EntBasePositive} becomes applicable to all non-trivial $\alpha|_\Sigma$-equivariant algebraic factors $X_0$ of $\dX$, where $
\Sigma$ is any finite index subgroup. Any such factor, being also non-trivial factors of $X$, is equipped with an induced action with rank at least $2$ by assumption (2) of Theorem \ref{MeasureInductive}, and positive entropy with respect to (possibly some integer multiple of) $\bfm$ by the corollary. Therefore the assumptions in Theorem \ref{MeasureInductive} hold for $(\talpha_\cU, \dtX_\cU,\dtmu_\cU)$, where $\dtX_\cU=\dtX\times Y_\cU=\dX\times Y\times Y_\cU$ and $\dtmu_\cU$ is the projection of $\dmu_\cU$.

As $\dim\dX<\dim X$, by inductive hypothesis $\dtmu_\cU$ is the Haar measure $\rmm_\dX$ in the $\dX$ component. In particular, \begin{equation}h_{\dtmu_\cU}(\talpha_\cU^\bfm|\cB_{Y\times Y_\cU})=h_{\rmm_\dX}(\alpha^\bfm)=h_{\dtmu}(\talpha^\bfm|\cB_Y).\end{equation}

On one hand, by assumption (1) in Proposition \ref{Ent0Fiber}, \begin{equation}\label{Ent0FiberEq1}\begin{split}h_\tmu(\talpha^\bfm|\cB_Y)=&h_{\dtmu}(\talpha^\bfm|\cB_Y)+h_\tmu(\talpha^\bfm|\cB_\dtX)=h_{\rmm_\dX}(\alpha^\bfm)+0\\
=&h_{\rmm_\dX}(\alpha^\bfm).\end{split}\end{equation}

On the other hand, \begin{equation}\label{Ent0FiberEq2}\begin{split}h_\tmu(\talpha^\bfm|\cB_Y)=&h_{\tmu_\cU}(\talpha_\cU^\bfm|\cB_Y)\geq h_{\tmu_\cU}(\talpha_\cU^\bfm|\cB_{Y\times Y_\cU})\\
\geq & h_{\dtmu_\cU}(\talpha_\cU^\bfm|\cB_{Y\times Y_\cU})=h_{\rmm_\dX}(\alpha^\bfm).\end{split}\end{equation} Here the first equality is because $\Phi$ is a measure-theoretic isomorphism between $(\tX, \tmu)$ and $(\tX_\cU, \tmu_\cU)$ preserving the $Y$ component. The last inequality is because entropy cannot increase on a  factor.

Comparing with \eqref{Ent0FiberEq1}, we see that both inequalities in \eqref{Ent0FiberEq2} must be equalities. In particular, $h_\tmu(\talpha^\bfm|\cB_Y)=h_{\tmu_\cU}(\talpha_\cU^\bfm|\cB_{Y\times Y_\cU})$.

As $\tmu$ and $\tmu_\cU$ both have isometric supports along $V$, by Lemma \ref{ProdLemma} the equation above can be rewritten as \begin{equation}h_\tmu(\talpha^\bfm,V)+h_\tmu(\talpha^\bfm,U)=h_{\tmu_\cU}(\talpha_\cU^\bfm,V)+h_{\tmu_\cU}(\talpha_\cU^\bfm,U).\end{equation} But by \eqref{EntLossEq1}, the two contributions along $U$ coincide, thus so do the two contributions along $V$. Recall that, by our assumption, $V$ has positive entropy contribution $h_\tmu(\talpha^\bfn,V)$ for some $\bfn$ where $V\subset G_\bfn^u$. By Lemma \ref{EntPosForAll}, $h_{\tmu_\cU}(\talpha_\cU^\bfm,V)=h_\tmu(\talpha^\bfm,V)>0$.

By Proposition \ref{EntContri}.(1), the leafwise measure $(\tmu_\cU)_{\tx_\cU}^V$ is not the atomic mass $\delta_e$ at identity for a positive portion of $\tx_\cU\in\tX_\cU$ with respect to $\tmu_\cU$. By Lemma \ref{LeafAction}, this property is preserved by the action $\talpha_\cU$, hence must hold $\tmu_\cU$-almost everywhere because of ergodicity. Or equivalently, $(\tmu_\cU)_{\Phi(\tx)}^V(\{e\})=0$ almost everywhere with respect to $\tmu$. Let $\cA_\cU$ be a $V$-subordinate $\sigma$-algebra on $\tX_\cU$, then $(\tmu_\cU)_{\Phi(\tx)}^{\cA_\cU}(\{\Phi(\tx)\})=0$ for $\tmu$-almost every $\tx$. 

On the other hand, by the construction given in Corollary \ref{LeafTribu}, if $\tx, \tx'\in\tX'$ and $[\tx]_\cU=[\tx']_\cU$, then $\tmu_\tx^V\simeq\tmu_{\tx'}^V$ where $\tX'\subset\tX$ is some fixed set of full measure. Furthermore, we may require that $\tX'$ is contained in the set $\tX^*$ in Proposition \ref{LeafMeas}.(2).

Since $\tmu_\cU\big(\Phi(\tX')\big)=1$, $\tmu_{\Phi(\tx)}^{\cA_\cU}\big(\Phi(X')\big)=1$ for almost every $\tx$. Therefore, for almost every $\tx$ (which may be required to belong to $\tX'$), there exists $\tx'_\cU\in[\Phi(\tx)]_{\cA_\cU}\backslash\{\Phi(\tx)\}$ with $\tx'_\cU\in\Phi(\tX')$. 

Fix $\tx'\in\tX'$ with $\Phi(\tx')=\tx'_\cU$. Then \begin{equation}\big(\tx',\phi(\tx')\big)=\Phi(\tx')=v.\Phi(\tx)=\big(v.\tx,\phi(\tx)\big)\end{equation} for some non-trivial $v\in V$. So $\tx$ and $\tx'$ satisfy: \begin{enumerate}\item both belong to $\tX'$; \item $\phi(\tx')=\phi(\tx)$, in other words, $[\tx']_\cU=[\tx]_\cU$;  \item $\tx'=v.\tx$.\end{enumerate} The first two properties imply $\tmu_\tx^V\simeq\tmu_{\tx'}^V$, and the last one yields $\tmu_\tx^V\simeq(\tau_v)_*\tmu_{\tx'}^V$ by Proposition \ref{LeafMeas}.(2).  Therefore, $\tmu_\tx^V\simeq(\tau_v)_*\tmu_\tx^V$, as Proposition \ref{Ent0Fiber} claims.
\end{proof}

\section{Consequence of translational invariance}\label{SecRatner}

This section is devoted to establish:

\begin{proposition}\label{TIMeas} In the settings of Theorem \ref{MeasureInductive}, if Property \ref{TransInvProp} holds, then there is a finite index subgroup $\Sigma\subset\bZ^r$, such that in the ergodic decomposition $\tmu=\frac1N\sum_{i=1}^N\tmu^i$ of $\tmu$ under the restricted action $\alpha|_\Sigma$, for each ergodic component $\tmu^i$, there is a non-trivial normal rational subgroup $L^i\subset G$ that is invariant under $\alpha|_\Sigma$, such that $\tmu^i$ is invariant under left translations by $L^i$.
\end{proposition}

To prove Proposition \ref{TIMeas}, we begin with an argument from \cite{BQ13}*{\S8}, then use facts about algebraic groups.

\subsection{Decomposition into Ratner measures}

Let $\tX$, $\talpha$, $\tmu$ and $\bfn$ be as in \S\ref{MeasureInductive}, and $V^{[\chi]}$ be the coarse Lyapunov subgroup in Property \ref{TransInvProp}. As $\chi\neq 0$, we can fix $\bfn\in\bZ^r$ with $\chi(\bfn)>0$, or in other words ,$V^{[\chi]}\subset G_\bfn^u$. Again, we denote $V^{[\chi]}$ by $V$ and $\gov^{[\chi]}$ by $\gov$ for simplicity.

Define a subgroup of $V$ at $\tmu$-almost every $\tx\in\tX$ by \begin{equation}V_\tx=\{v\in V: (\tau_v)_*\tmu_\tx^V\simeq \tmu_\tx^V\},\end{equation} and let $\gov_\tx=\{\log v: v\in V_\tx\}\subset\gov.$ In other words, $V_\tx$ is the stabilizer of $\tmu_\tx^V$ in $\cM_1^0(V)$ with respect to right translations. In particular, $V_\tx$ is a closed subgroup. It is non-trivial by Property \ref{TransInvProp}, which we assume throughout this section.

It follows from the equivariance property in Lemma \ref{LeafAction} that \begin{equation}\label{LeafStabEq}\alpha^\bfn.V_\tx=V_{\talpha^\bfn.\tx},  \text{ for }\tmu\text{-a.e. }x.\end{equation}

\begin{lemma}\label{LeafStabConn} For $\tmu$-almost every $\tx\in\tX$, $V_\tx$ is connected.\end{lemma}
\begin{proof} Assume the opposite, then $\gov_\tx$ is disconnected, and denote by $\gov_{\tx,0}$ its identity component. The function $l(\tx):=\inf_{w\in\gov_\tx\backslash\gov_{\tx,0}}\|w\|$ is positive at a $\tmu$-positive portion of $\tx$.  In particular, there is some $M>0$ such that $\tX_M=\{\tx: M<l(\tx)\leq 2M\}$ has positive measure. Moreover, by our assumption, $l(\tx)<\infty$ almost everywhere. By \eqref{LeafStabEq},  \begin{equation}\label{LeafStabConnEq1}l(\talpha^{k\bfn} x)=\inf_{w\in \gov_\tx\backslash\gov_{\tx,0}}\|\alpha^{k\bfn}.w\|.\end{equation} 

As $V\subset G_\bfn^u$,
$\|\alpha^{-k\bfn}\|\rightarrow 0,\text{ as }k\rightarrow\infty$. Thus for almost all $\tx$ with $l(\tx)>0$, $\frac{l(\alpha^{-k\bfn}.\tx)}{l(\tx)}$ decays uniformly to $0$ as $k$ grows. In particular, for sufficiently large $k$, $\talpha^{-k\bfn}.\tX_M$ is disjoint from $\tX_M$, which violates Poincar\'e's recurrence theorem. So $V_\tx$ must be connected.\end{proof}

Therefore $V_\tx=\exp\gov_\tx$ with $\gov_\tx$ being a Lie subalgebra in $\gov$. 

We emphasize that Property \ref{TransInvProp} is only invariance in a proportional sense and is weaker than actual invariance $\tmu_\tx^V=(\tau_V)_*\tmu_\tx^V$. We now show that the two are in fact equivalent.

\begin{lemma}For $\tmu$-almost every $\tx$ and all $v\in V_\tx$, $(\tau_v)_*\tmu_\tx^V=\tmu_\tx^V$.\end{lemma}
\begin{proof}By definition of the relation $\simeq$, there is a continuous function $\hat\rho_\tx:V_\tx\mapsto\bR_+$ such that for almost every $\tx$ and $v\in V_\tx$, $(\tau_v)_*\tmu_\tx^V=\hat\rho_\tx(v)\tmu_\tx^V$. One see easily that $\hat\rho_\tx$ is a group morphism into the multiplicative group $\bR_+$. Let $\rho_\tx:\gov_\tx\mapsto\bR$ be the Lie algebra morphism given by the derivative of $\hat\rho_\tx$. 

Assume $\rho_\tx$ is non-trivial at a positive portion of points. Then at these $\tx$, $\rho_\tx$ is a non-trivial linear map with $\|\rho_\tx\|>0$. Similar to the proof of Lemma \ref{LeafStabConn}, there is $M$ such that $\tX_M :=\{\tx: M<\|\rho_\tx\|\leq 2M\}$ has positive measure.

On the otherhand, \eqref{LeafStabEq} implies $\rho_{\alpha^{k\bfn}.\tx}=\rho_\tx\circ\alpha^{-k\bfn}$. Using the same argument in the proof of Lemma  \ref{LeafStabConn}, we see that $\frac{\left\|\rho_{\alpha^{k\bfn}.\tx}\right\|}{\|\rho_\tx\|}$ decays uniformly for almost all $\tx$ as $k\rightarrow\infty$. This implies $\alpha^{k\bfn}.\tX_M$ becomes disjoint from $\tX_M$ for all large $k$ and again contradicts Poincar\'e's recurrence theorem. 

Hence at almost every point $\rho_\tx\equiv 0$ and $\hat\rho_\tx\equiv 1$. In other words, $(\tau_v)_*\tmu_\tx^V=\tmu_\tx^V$.\end{proof}

Desintegrating $\tmu$ with respect to the measurable map $\tx\mapsto\gov_\tx$ from $\tX$ to the Grassmanian of $\gov$, we get a decomposition 
\begin{equation}\label{StabDesintEq}\tmu=\int\tmu_\tx\di\tmu(x)\end{equation}

\begin{lemma}\label{LocalTransInv}For $\tmu$-almost every $\tx$, the probability measure $\tmu_\tx$ on $\tX$ is invariant under left translations by $V_\tx$.
\end{lemma}
\begin{proof}This follows from \cite{BQ11}*{Prop. 4.3}\footnote{In the statement of \cite{BQ11}*{Prop. 4.3}, the acting group, which we denote by $V$ here, is $\bR^d$. But as in \cite{BQ13}, the switch to a nilpotent Lie group $V$ does not impose any difference, thanks to the unimodularity of nilpotent Lie groups.}.\end{proof}

Equation \eqref{StabDesintEq} and Lemma \ref{LocalTransInv} allow us to desintegrate $\tmu$ into ergodic $V_\tx$-invariant components.

\begin{definition} Let $\cE(X)$ be the collection of Ratner measures on $X$. These are all probability measures on $X$ which can be written as the unique $H$-invariant probability measure on a compact orbit $H.x$ of some connected closed subgroup $H\subset G$.

Define $\cE(\tX):=\{\xi\times\delta_y: \xi\in\cE(X), y\in Y\}$, which is the collection of uniform probability measures on compact orbits in $\tX$.
\end{definition}

\begin{proposition}\label{RatnerDecomp}\begin{enumerate}\item $\tmu$ admits a decomposition
$$\tmu=\int_{\tX}\txi_\tx\di\tmu(\tx),$$
where $\tx\mapsto\txi_\tx$ is a measurable map from $\tX$ to $\cE(\tX)$ that is defined for $\tmu$-almost every $\tx$. In addition, $\tx\in\supp\txi_\tx$.
\item $\txi_{\alpha^\bfm.\tx}=\alpha_*^\bfm\txi_\tx$ for $\tmu$-almost every $\tx$.
\end{enumerate}
\end{proposition}
\begin{proof}(1) By \eqref{StabDesintEq}, it suffices to prove the claim for $\tmu_\tx$ instead of $\tmu$. However, by Lemma \ref{LocalTransInv}, $\tmu_\tx$ can be desintegrated into ergodic invariant components under left translations by $V_\tx$. Since $V_\tx$ is a connected nilpotent Lie group that acts only on the $X$ component of $\tX=X\times Y$, by Ratner's Theorem \cite{R91}, each of these components belongs to the family $\cE(\tX)$.

(2) It follows from \eqref{LeafStabEq} that the desintegration of $\tmu$ into $\tmu_\tx$, as well as the decomposition of $\tmu_\tx$ into $V_\tx$-ergodic components, are equivariant under the group action.
\end{proof}

If $\xi\in\cE(X)$, then $g_*\xi\in\cE(X)$ as well for all $g\in G$, where $g_*$ is the left translation pushforward by $g$. This defines a left $G$-action on $\cE(X)$. 

Hereafter, $x_0$ stands for the origin in $G/\Gamma$. 

\begin{remark}\label{RatnerRep}As the uniform measure on a compact orbit $L.gx_0$ of some closed subgroup $L\subset G$ is a left translate of $g^{-1}Lg.x_0$, every class in $G\backslash\cE(X)$ can be represented by the uniform measure on the compact orbit $L.x_0$ of a rational subgroup $L$.
\end{remark}

Recall a rational subgroup $L$ is connected and its Lie algebra is a rational subspace in $\gog$ with respect to the lattice $\exp{-1}(\Gamma)$. Due to the fact that there are only countably many rational subspaces in a real vector space with a $\bQ$-structure, the family of rational subgroups are countable. Thanks to the bijection between rational subgroups $L$ and the class in $G\backslash\cE(X)$ represented by $L.x_0$, one has the following lemma.

\begin{lemma}\label{RatnerCountable} $G\backslash \cE(X)$  is countable.\end{lemma}

This is a special case of a general theorem fact, due to Ratner, and other authors, on general homogeneous spaces.

On $\cE(\tX)$, we define an equivalence relation $\sim$: $\xi\times\delta_y\sim\xi'\times\delta_{y'}$ if and only if $\xi'=g_*\xi$ for some $g\in G$. Then $\cE(\tX)/\sim$ is identified naturally with $G\backslash \cE(X)$. 

Notice for $\xi\times\delta_y\in\cE(\tX)$,  $\talpha_*^\bfm(\xi\times\delta_y)=\alpha_*^\bfm\xi\times\delta_{\beta^\bfm y}$ is also in $\cE(\tX)$. Moreover, $\alpha_*^\bfm(g_*\xi)=(\alpha_*^\bfm g)_*(\alpha_*^\bfm\xi)$. Hence $\talpha_*$ induces a $\bZ^r$-action on $\cE(\tX)/\sim$, which we still denote by $\talpha_*$.

Consider the map $P:x\mapsto [\txi_\tx]$, where $[\txi_\tx]$ is the class containing $\txi_\tx$ in $\cE(X)/\sim$, and the pushforward  $P_*\tmu$ on the countable space $\cE(X)/\sim$. By Proposition \ref{RatnerDecomp}.(2), $P_*\tmu$ is $\talpha_*$-invariant and ergodic as $\tmu$ is. However, as $\cE(X)/\sim$ is countable by Lemma \ref{RatnerCountable}, $P_*\tmu$ must be an atomic measure uniformly distributed on finitely many points. In light of Remark \ref{RatnerRep}, this yields

\begin{lemma}\label{RatnerFiniteDecomp} There is a finite index subgroup $\Sigma\subset\bZ^r$, with respect to which $\tmu$ has an ergodic decomposition $\frac1N\sum_{i=1}^N\tmu^i$, so that for each $i$, there exist a rational subgroup $L^i\subset G$, and a decomposition
$$\tmu^i=\int_\tX\txi_\tx^i\di\tmu^i(x),$$
such that:\begin{enumerate}
\item $\tmu^i$ is $\talpha|_\Sigma$-invariant and ergodic, and the $\tmu^i$'s are permuted transitively by $\talpha$;
\item For $\tmu^i$-almost every $\tx\in\tX$, $\tx\in\supp\txi_\tx^i$. And there exists $g_\tx\in G$ and $y_\tx\in Y$ such that $\txi_\tx^i=(g_\tx)_*\xi^i\times\delta_{y_\tx}$, where $\xi^i$ is the unique $L^i$-invariant probability measure on $L^i. x_0$.
\end{enumerate}
\end{lemma}
Since $\tx\in\supp\txi_\tx^i$, $y_\tx$ is just the $y$-component of $\tx$.

As $\tmu$ has large projection on factors, we expect the same is true for $\tmu^i$ because the $\tmu^i$'s form a finite decomposition. 

Recall the commutator subgroup $[G,G]\subset G$ is normal, connected and invariant under automorphisms of $G$. It is also rational (see \cite{CG90}*{Cor. 5.22}). 

Denote the maximal abelian factor $X/[G,G]$ by $X_\ab$, and $\tX/[G,G]=X_\ab\times Y$ by $\tX_\ab$. They are $\alpha$-equivariant algebraic factors respectively of $X$ and $\tX$.

Let $\tmu_\ab$ and $\tmu_\ab^i$ respectively be the projections of $\tmu$ and $\tmu^i$ to $\tX_\ab$.

\begin{lemma}\label{AbUnif}If $G$ is not abelian. Then the projection of each $\tmu^i$ to $\tX_\ab$ can be written as $\rmm_{X_\ab}\times\nu^i$, where $\nu^i$ is some ergodic $\beta|_\Sigma$-invariant probability measure on $Y$. \end{lemma}

\begin{proof}Take any finite index subgroup $\Sigma_0$, and any $\alpha|_{\Sigma_0}$-equivariant algebraic factor $X_0$ of $X_\ab$. As $X_0$ is also a factor of $X$, by assumptions in Theorem \ref{MeasureInductive} the induced action on $X_0$ is not virtually cyclic, and the projection $\tmu_0$ of $\tmu$ to $\tX_0=X_0\times Y$ has positive conditional entropy $h_{\tmu_0}(\talpha^\bfn|\cB_Y)$ for some $\bfn\in\Sigma_0$.  As $G$ is non-abelian, $\dim X_\ab=\dim X-\dim[G,G]<\dim X$. So by inductive hypothesis, $\tmu_\ab=\rmm_{X_\ab}\times\nu$, where $\nu$ is the projection of $\tmu$ on $Y$.

As $\tmu^i$ is $\talpha|_\Sigma$-ergodic, so is $\tmu_\ab^i$. So $\tmu_\ab^i$ is one of the $\talpha|_\Sigma$-ergodic components of $\rmm_{X_\ab}\times\nu$. The lemma follows from Corollary \ref{ProdErgo}.\end{proof}

\subsection{Reducing support to the normalizer}

From now on we fix a component $\tmu^i$ from Lemma \ref{RatnerFiniteDecomp}. In this part, we prove  $\tmu^i$ is invariant under left translations by a non-trivial subgroup of $G$ and its support is characterized by the normalizer of this subgroup up to a compactly supported factor.

For $\bfm\in\Sigma$, $\alpha^\bfm$ sends $\xi_\tx$ to $(\alpha^\bfm g_\tx)_*(\alpha_*^\bfm\xi^i)$, which must be in the same $\cE(\tX)/\sim$ equivalence class as $\xi^i$. Thus $\alpha_*^\bfm\xi^i\sim\xi^i$.

On the other hand, $\alpha_*^\bfm\xi^i$ is the uniform probability measure on the compact homogeneous subspace $(\alpha^\bfm.L^i).x_0$.
Therefore, the group $\alpha^\bfm.L^i$ must coincide with $L^i$ for all $\bfm\in\Sigma$, in other words, $L^i$ is $\alpha|_\Sigma$-invariant.

Let $\hL^i$ be the stabilizer in $G$ of $\xi^i$. Then $\hL^i$ is $\alpha|_\Sigma$-invariant and closed in $G$. Furthermore $L^i\lhd\hL^i$ and $L^i$ is the identity component of $\hL^i$.

Notice that given $\tx$, $g_\tx$ is uniquely defined up to right translation by $\hL^i$. Hence  $G/\hL^i$ can be naturally identified with the equivalence class $[\xi^i]\subset \cE(X)$ with respect to the $G$-action. As $\xi^i$ is $\alpha|_\Sigma$-invariant, $\Sigma$ acts on $[\xi^i]$ by $\alpha_*$. This is identified with the natural $\Sigma$-action on $G/\hL^i$ by $\alpha|_\Sigma$.

Let $d=\dim L^i$. The $\bZ^r$-action $\alpha$ naturally gives rise to an action $\wedge^d\alpha:\bZ^r\curvearrowright \wedge^d\gog$. Define a base point $\theta_0\in\wedge^d\gog$ by fixing a volume form on the Lie algebra $\gol^i$ of $L^i$. Then $\theta_0$ is preserved by $(\wedge^d\alpha)|_\Sigma$ up to orientation. By replacing $\Sigma$ with a subgroup of index $2$ (and hence doubling the cardinality $N$ in Lemma \ref{RatnerFiniteDecomp}) if necessary, we may assume $\theta_0$ is actually preserved. 

Given $x$, consider the pushforward $\theta_\tx\in\wedge^d\gog$ of $\theta$ by $\wedge^d\Ad_{g_\tx}$, which is in the $d$-th exterior product space of the Lie subalgebra $\Ad_{g_\tx}\gol^i$. Because $\hL^i$ normalizes $L^i$ and $G$ is nilpotent, $\Ad_g$ fixes $\theta_0$ if $g\in\hL^i$. So $\theta_\tx$ is well defined, as $g_\tx$ is unique up to right translation by $\hL^i$.

For any $\bfm\in\Sigma$, one can choose $g_{\talpha^\bfm.\tx}=\alpha^\bfm.g_\tx$. Then for $\bfm\in\Sigma$, using the invariance of $\theta_0$ under $\wedge^d\alpha^\bfm$, we have
\begin{equation}\begin{split}\theta_{\alpha^\bfm.\tx}=&(\wedge^d\Ad_{\alpha^\bfm g_\tx})_*\theta_0=(\wedge^d\alpha^\bfm \circ\wedge^d\Ad_{g_\tx}\circ\wedge^d\alpha^{-\bfm}).\theta_0\\
=&(\wedge^d\alpha^\bfm \circ\wedge^d\Ad_{g_\tx}).\theta_0=(\wedge^d\alpha^\bfm).\theta_\tx.\end{split}\end{equation}

Hence the maps \begin{equation}\label{ProjMaps}\begin{array}{ccccc}\tx&\mapsto& g_\tx\hL^i&\mapsto &\theta_\tx\\
                  \tX&\mapsto& G/\hL^i&\mapsto&\wedge^d\gog
                 \end{array}\end{equation}
interwines the actions $\talpha|_{\Sigma}$ on $\tX$, the induced one on $G/\hL^i$ and $(\wedge^d\alpha)|_{\Sigma}$ on $\wedge^d\gog$.

Recall that any simply connected nilpotent Lie group $G$ is an algebraic group, and its adjoint representation is algebraic. Consider two $G$-actions by regular maps on algebraic varieties: one is the representation $\wedge^d\Ad$ of $G$ on the affine variety $\wedge^d\gog$; the other is its projectivization, which we denoted by $\wedge^d\Ad$ indifferently, on the projective variety $\bP(\wedge^d\gog)$. Let $\pi_\bP:\theta\mapsto[\theta]$ be the projection from $\wedge^d\gog$ to $\bP(\wedge^d\gog)$. Notice that $\pi_\bP$ commutes the two actions.

$[\theta_0]$ is the normalized volume form on $\gol^i$. 

The stabilizers of both $\theta_0$ and $[\theta_0]$ in $G$ are $N_G(L^i)$. For $[\theta_0]$ this is obvious. For $\theta_0$, this is because, since the ambient group is nilpotent, any element normalizing $L^i$ also preserves the invariant volume on it.  Denote respectively by $\cO$ and $\cO_\bP$ the $G$-orbits of $\theta_0$ and $[\theta_0]$, both of which are set-theoretically identified with $G/N_G(L^i)$. Then because the stabilizers are the same and $\pi_\bP$ interchanges the two actions, $\pi_\bP$ is a bijection between $\cO$ and $\cO_\bP$.

It is more than that: $\pi_\bP$ is an algebraic morphism from $\cO$ to $\cO_\bP$. For this purpose we need to explain first $\cO$ and $\cO_\bP$ are varieties.

A theorem of Rosenlicht \cite{R61}*{Thm. 2} says all orbits of an algebraic action by a unipotent algebraic group on an affine variety are closed. In our case, this says $\cO$ is a Zariski closed subset of the affine variety $\wedge^d\gog$. In particular, it is also closed in the Hausdorff sense.

On the other hand, $\cO_\bP$ is a constructible subset in $\bP(\wedge^d\gog)$, and, with the natural Zariski topology, is a quasi-projective variety isomorphic to the homogeneous space $G/N_G(L^i)$. In fact, one can use $\cO_\bP$ as a definition of the quotient variety  $G/N_G(L^i)$, see \cite{H75}*{\S 12}.

Let $N_\ab$ be the projection of $N_G(L^i)$ to the maximal abelian quotient group $G_\ab=G/[G,G]$. It is a general fact about nilpotent Lie groups that $N_G(L^i)$ is connected, rational and closed in $G$.  Therefore $N_\ab$ remains a connected rational Lie subgroup (with respect to the lattice $\Gamma_\ab$, which is the projection of $\Gamma$).  Furthermore, $G_\ab$, $N_\ab$ and $G_\ab/N_\ab$ are all connected torsion-free abelian groups, or equivalently, vector spaces over $\bR$.

From now on, denote by $\pi$ the natural projection from $G$ to $G_\ab/N_\ab$.

There is a canonical quotient map $\pi_{\bP,\ab}$ between homogeneous spaces from $G/N_G(L^i)$ to $G_\ab/N_\ab\cong G/N_G(L^i)[G,G]$, which is a regular morphism and projects the coset $g N_G(L^i)$ to $\pi(g)=gN_G(L^i)[G,G]$. As the Zariski structure of $G/N_G(L^i)$ is intrinsic and does not depend on the model $\cO_\bP$, we can view $\pi_{\bP,\ab}$ as a regular map from $\cO_\bP$ to $G_\ab/N_\ab$, that sends $\big[(\wedge^d\Ad_g).\theta_0\big]$ to the natural projection $\pi(g)$ of $g$.

Define a new map $\iota=\pi_{\bP,\ab}\circ\pi_\bP$ from $\cO$ to $G_\ab/N_\ab$, then from the discussion above, we know immediately $\iota$ has the following properties:

\begin{lemma}\label{ContiProj}\begin{enumerate}
\item $\iota\big((\wedge^d\Ad_g).\theta_0\big)=\pi(g)$ is the natural projection of $g\in G$ in $G_\ab/N_\ab$.
\item $\iota$ is regular in Zariski topology, and hence continuous in Hausdorff topology.
\end{enumerate}\end{lemma} 

Finally,  construct another map $\psi$ from $\tX$ to $G_\ab/N_\ab$ by \begin{equation}\label{ModuloNEq}\psi:\tx\mapsto\pi(g_\tx).\end{equation} Again, since $g_\tx$ is defined up to right translation by $\hL^i\subset N_G(L^i)\subset\ker\pi$, $\psi$ is well defined.

The relations between the maps we are discussing are clarified in the commutative diagrams below:

\begin{equation}\begin{diagram}
&&\cO&&\\
&\ruTo^{g\mapsto(\wedge^d\Ad_g).\theta_0}&\dTo_{\pi_\bP}&\luTo^{\theta}&\\
G&\rTo^{g\mapsto gN_G(L^i)}&\large\substack{\cO_\bP\\=G/N_G(L^i)}&\lTo^{\tx\mapsto g_\tx N_G(L^i)}&\tX\\
&\rdTo_\pi&\dTo_{\pi_{\bP,\ab}}&\ldTo_{\psi}&\\
&&G_\ab/N_\ab&&
\end{diagram}\end{equation}

\begin{lemma}\label{CpctModN}$\supp\psi_*\tmu^i$ is compact in $G_\ab/N_\ab$.\end{lemma}

We need the following fact from linear algebra to prove the lemma.

\begin{lemma}\label{CpctLinear}Suppose $W$ is a real vector space, $\phi$ is a $\bZ^r$-action on $W$ by linear isomorphisms, and $\rho$ is an ergodic $\phi$-invariant probability measure on $W$. Then there exist a $\phi$-invariant subspace $W_0$ of $W$, and a non-degenerate inner product $q_0$ on $W_0$, such that:\begin{enumerate}
\item The restriction of $\phi$ to $W_0$ acts as a subgroup of $\SO(W_0,q_0)$;
\item $\rho$ is supported on some sphere $\{w\in W_0: \|w\|_{q_0}=R\}$ inside $W_0$.\end{enumerate}\end{lemma}

\begin{proof}
By Poincare's recurrence theorem, for any $\bfm\in\bZ^r$, $\phi$ can only be supported on the non-wandering points of $\phi^\bfm$. 

As the action is commutative, we can decompose $W\otimes_\bR\bC$ into a direct sum $\bigoplus_{j=1}^JW_j$ of common generalized eigenspaces for $\phi^\bfm$, with eigenvalue $\theta_j^\bfm\in\bC^\times$ for $\bfm\in\Sigma$ on $W_j$. Suppose $q=\bigoplus_j q_j$ is a non-wandering point in $W$. We claim that each $W_j$-component $w_j$ is non-zero only if $w_j$ is a common eigenvector, with all eigenvalues $\theta_j^\bfm$ from the unit circle $\{|\theta|=1\}$.

In fact, if $w_j\neq 0$ and for some $\bfm$, \begin{itemize}
\item either $w_j$ is not an eigenvector of the Jordan block $\phi^\bfm|_{W_j}$,
\item or $|\theta_j^\bfm|>1$,
\end{itemize}
then $\phi^{k\bfm}.w_j\rightarrow\infty$ as $k$ grows, which prevents $w_j$, and thus $w$, from being non-wandering. Since $w_j$ should be non-wandering for $\phi^{-\bfm}$ as well,  $|\theta_j^\bfm|<1$ is also an obstruction. This shows the claim.

Therefore the set of non-wandering points are characterized by the direct sum $W_0^\bC$ of the common eigenspaces of $\phi$ with eigenvalues only from the unit circle. Take $W_0=W_0^\bC\cap W$, and $q_0=\sum_{j: W_j\subset W_0^\bC}|w_j|^2$. Then part (1) is satisfied.

Given part (1), any point $q\in\supp\rho$ must belong to $W_0$, and thus has its orbit under $\phi$ contained in a sphere. Since $\rho$ is ergodic, it is supported on the closure of a single orbit, which is sufficient to complete the proof.\end{proof}

\begin{proof}[Proof of Lemma \ref{CpctModN}] Let $\theta_*\tmu^i$ be the pushforward by $\theta:\tx\mapsto \theta_\tx$ on $\wedge^d\gog$. Then $\theta_*\tmu^i$ is invariant and ergodic under $(\wedge^d\alpha)|_{\Sigma}$. By Lemma \ref{CpctLinear}, $\theta_*\tmu^i$ is compactly supported in $\wedge^\gog$. On the other hand, by Rosenlicht's theorem, $\cO$ is Zariski closed, and hence Hausdorff closed as well, in $\wedge^\gog$. Because $\psi$ maps into $\cO$ and  the support of $\psi_*\tmu^i$ is actually a compact subset of $\cO$.

Moreover, one can easily check that $\psi=\iota\circ\theta$. Because $\iota:\cO\mapsto G_\ab/N_\ab$ is continuous by Lemma \ref{ContiProj}, $\psi_*\tmu^i=\iota_*(\theta_*\tmu^i)$ is compactly supported. \end{proof}

\subsection{Normalizer is full} Provided Lemma \ref{CpctModN}, we show that \begin{lemma}\label{FullN}$N_G(L^i)$ is equal to $G$.\end{lemma}

Note that, if $G$ is abelian then the statement trivially holds. Therefore we focus on the non-abelian case, in which case we show that:

\begin{lemma}\label{AbFullN}If $G$ is not abelian, then $N_\ab=G_\ab$.\end{lemma}

\begin{proof} As $[G,G]$ is invariant under all automorphisms of $G$ and $N_G(L^i)$ is $\alpha|_\Sigma$ invariant, there is a natural induced action $\alpha|_\Sigma$ on $G_\ab/N_\ab$. For a $\tmu^i$-generic point $\tx\in\tX$ and $\bfm\in\Sigma$, one can choose $g_{\talpha^\bfm.\tx}=\alpha^\bfm.g_\tx$. Thus \begin{equation}\psi(\talpha^\bfm.\tx)=\pi(\alpha^\bfm.g_\tx)=
\alpha^\bfm.\pi(g_\tx)=\alpha^\bfm.\psi(\tx).\end{equation}

It follows that $\psi_*\tmu^i$ is an ergodic $\alpha|_\Sigma$-invariant probability measure on $G_\ab/N_\ab$.  However, as we remarked before, $G_\ab/N_\ab$ is just a vector space. And the $\alpha^\bfm$'s, descending from automorphisms of $G$, act as linear
isomorphisms on $G_\ab/N_\ab$. So by Lemma \ref{CpctLinear}, there is a linear subspace $W_0$ of $G_\ab/N_\ab$, equipped with a non-degenerate inner-product $q_0$, such that $\psi_*\tmu^i$ is supported on a single sphere $S=\{w\in W_0: \|w\|_{q_0}=R\}$ with respect to the norm $q_0$. If $G_\ab/N_\ab$ is not trivial, then $\dim S<\dim(G_\ab/N_\ab)$.
 
 Recall that $N_\ab$ is rational with respect to the projection $\Gamma_\ab$ of $\Gamma$ to $G_\ab$. Let $x_{0,ab}=\Gamma_\ab$ be the origin in the torus $G_\ab/\Gamma_\ab$. Then $N_\ab x_{0,ab}$ is a compact orbit, and hence must be a subtorus $T$ in the torus $X_\ab=G_\ab/\Gamma_\ab$. Especially, $T$ is a connected closed subgroup of $X_\ab$ and is invariant under $\alpha|_\Sigma$. 
 
 Consider the quotient $X_\ab/T$, it is a lower dimensional torus and can also be written as $G_\ab/(N_\ab\Gamma_\ab)$. The diagram below, with all unmarked arrows representing natural projections, is commutative in a $\tmu^i$-almost everywhere sense:
 \begin{equation}\label{FullNDiagram}\begin{diagram}
 \tX&\rTo&\tX_\ab&\rTo&X_\ab\\
 \dTo^{\psi}&&&&\dTo(0,1)\\
 G_\ab/N_\ab&\rTo&G_\ab/(N_\ab\Gamma_\ab)&\rEq&X_\ab/T
 \end{diagram}\end{equation}
 
Indeed, for $\tmu^i$-almost every point $\tx$, choose any $g_\tx$ that satisfies Lemma \ref{RatnerFiniteDecomp}. Then $\psi(\tx)$ is the natural projection $\pi(g_\tx)$ of $g_\tx$ from $G$ to $G_\ab/N_\ab$, and hence the image of $\tx$ through $G_\ab/N_\ab$ to $X_\ab/T$ is the natural projection of $\pi(g_\tx)$.

On the other hand,  by  Lemma \ref{RatnerFiniteDecomp}, $\tx$ can be written as $(x,y_\tx)$, where $x\in\supp(g_\tx)_*\xi^i=g_\tx.(L^i.x_0)\subset g_\tx.N_G(L^i).x_0$. Therfore the projection of $\tx$ to $\tX_\ab$ is inside $g_\tx.N_\ab x_{0,ab}\times\{y_\tx\}$. When one further projects to $X_\ab$, the image is inside $g_\tx.N_\ab x_{0,ab}$, which is just $\pi(g_\tx).T$. Therefore, the projection through the upper-right path in the diagram also sends $\tx$ to the natural projection of $\pi_(g_\tx)$ in $X_\ab/T$. So the diagram \eqref{FullNDiagram} commutes.

Now call the map from the diagram by $\Pi: \tX\mapsto X_\ab/T$. On one hand, by Lemma \ref{AbUnif}, the pushforward of $\tmu^i$ to $X_\ab$ is $\rmm_{X_\ab}$, and therefore $\Pi_*\tmu^i$ is the uniform probability measure on the quotient torus $X_\ab/T$.   

Whereas, when one takes the other path, unless $G_\ab=N_\ab$, $\psi_*\tmu_i$ is supported on a lower dimensional sphere $S$ inside the vector space $G_\ab/N_\ab$. Moreover, $X_\ab/T=G_\ab/(N_\ab\Gamma_\ab)=(G_\ab/N_\ab)\Big/\big(\Gamma_\ab/(\Gamma_\ab\cap N_\ab)\big)$ is a quotient torus of $G_\ab/N_\ab$ by a lattice. The further projection of the compact sphere $S$ to  $X_\ab/T$ has the same dimension as $S$, which is less than $\dim X_\ab/T$.   Thus $\Pi_*\tmu_i$, which is the pushforward of $\psi_*\tmu_i$ and  is supported on the projection of $S$, cannot be the uniform measure.

This contradiction forces $N_\ab$ to be equal to $G_\ab$.\end{proof}

\begin{proof}[Proof of Lemma \ref{FullN}] As remarked earlier, a proof is required only if $G$ is not abelian. Let $G_0=G$, $G_1=[G,G]$, $G_2=[G, G_1]$, $\cdots$, $G_k=\{e\}$ be the lower central series of $G$. And denote by $N_j$ the projection of $N_G(L^i)$ to the quotient group $G/G_j$. We claim that $N_j=G/G_j$ for all $j\geq 1$.

For $j=1$ this is the content of Lemma \ref{AbFullN}. Assuming $j\geq 2$ and the claim holds for $j-1$, we prove by induction.

Because $G/G_{j-1}=(G/G_j)\Big/(G_{j-1}/G_j)$ and the projection of $N_j$ to $G/G_{j-1}$ is $N_{j-1}=G/G_{j-1}$, it suffices to show that $N_j$ contains $G_{j-1}/G_j$. 

As $G_{j-1}=[G, G_{j-2}]$, for this purpose we only need that for all $h\in G$ and $h'\in G_{j-2}$, $N_G(L^i)$ has non-trivial intersection with $[h,h']G_j$. Since $N_G(L^i)$ has full projection modulo $G_{j-1}$, we can pick elements $g\in hG_{j-1}$ and $g'\in h'G_{j-1}$ from $N_G(L^i)$. Then $[g,g']$ is congruent to $[h,h']$ modulo $[G_{j-1},G_{j-1}]\subset G_j$. This proves the claim.

 The lemma is the $j=k$ case of the claim.\end{proof}

The implication from Lemma \ref{FullN} to Proposition \ref{TIMeas} is straightforward.
\begin{proof}[Proof of Proposition \ref{TIMeas}] By Lemma \ref{FullN}, $L^i \lhd G$. In Lemma \ref{RatnerFiniteDecomp}, for $\tmu^i$-almost every $\tx$, the measure component $\txi_\tx^i$ at $\tx$ equals $(g_\tx)_*\xi^i\times\delta_{y_\tx}$. Note $(g_\tx)_*\xi^i$ is $g_\tx L^i g_\tx^{-1}$-invariant, but $g_\tx L^i g_\tx^{-1}=L^i$ as $L^i$ is normal. So almost every component $\txi_\tx^i$ is $L^i$-invariant. By the decomposition in Lemma \ref{RatnerFiniteDecomp}, $\tmu^i$ is $L^i$-invariant. Recall that $L^i$ is a non-trivial rational subgroup, so Proposition \ref{TIMeas} follows.
\end{proof}

\section{Example of non-rigidity}\label{SecHeisen}
\subsection{The Heisenberg nilmanifold}
Write $(\bfx,\bfy,z)$ for the coordinates on $\bR^{13}=\bR^6\oplus\bR^6\oplus\bR$. Define the Lie bracket by
\begin{equation}\label{HeisenBracketEq}\big[(\bfx,\bfy,z),(\bfx',\bfy',z)\big]=\big(\bfzero,\bfzero,2\bfx^\T\bfy'-2(\bfx')^\T\bfy\big).\end{equation}
This turns $\bR^{13}$ into the 13-dimensional Heisenberg Lie algebra, which is two-step nilpotent. 

Denote this Lie algebra by $\gog$ and let $G=\exp\gog$ be the corresponding Heisenberg Lie group. $\exp$ is a diffeomorphism between $\gog$ and $G$. We denote $\exp(\bfx,\bfy,z)\in G$ indifferently by $(\bfx,\bfy,z)$. Then $G$ is set-theoretically $\bR^{13}$, with identity $(\bfzero,\bfzero,0)$. Using Baker-Campbell-Hausdorff formula, the group rule is
\begin{equation}\label{HeisenRuleEq}\begin{split}&(\bfx,\bfy,z)\cdot (\bfx',\bfy',z)\\=&(\bfx,\bfy,z)+(\bfx',\bfy',z)+\frac12\big[(\bfx,\bfy,z),(\bfx',\bfy',z)\big]\\
=&(\bfx+\bfx',\bfy+\bfy',z+z'+\bfx^\T\bfy'-(\bfx')^\T\bfy).\end{split}\end{equation}
In particular, $(\bfx,\bfy,z)^{-1}=(-\bfx,-\bfy,-z)$.

 The integer points $\Gamma=\{(\bfx,\bfy,z)\in G: \bfx,\bfy\in\bZ^6, z\in\bZ\}$ is a cocompact lattice in $G$. The resulting quotient $X=G/\Gamma$  is the 13-dimensional Heisenberg nilmanifold.
 
Let $\{\bfx\}\in[-\frac12,\frac12)^6$ and $\{z\}\in[-\frac12,\frac12)$ be the unique representatives that are respectively congruent to $\bfx\in\bR^6$ and $z\in\bR$,  modulo $\bZ^6$ and $\bZ$.  Write $[\bfx]=\bfx-\{\bfx\}$ and $[z]=z-\{z\}$ for the integer parts.

There is a unique integer part\footnote{Depending on the context, the symbol should not cause confusion with the Lie bracket.} $[(\bfx,\bfy,z)]\in\Gamma$ for every $(\bfx,\bfy,z)\in G$, such that the difference $\{(\bfx,\bfy,z)\}=(\bfx,\bfy,z)[(\bfx,\bfy,z)]^{-1}$ lies in $[-\frac12,\frac12)^{13}$. With simple calculations, one find
\begin{equation}\label{HeisenIntEq}\begin{split}[(\bfx,\bfy,z)]=&\Big([\bfx],[\bfy],\big[z-\bfx^\T[\bfy]+[\bfx]^\T\bfy\big]\Big)\\
=&\Big([\bfx],[\bfy],\big[z+\bfx^\T\{\bfy\}-\{\bfx\}^\T\bfy\big]\Big);\end{split}\end{equation}
\begin{equation}\label{HeisenResEq}\begin{split}\{(\bfx,\bfy,z)\}=&\Big(\{\bfx\},\{\bfy\},\big\{z-\bfx^\T[\bfy]+[\bfx]^\T\bfy\big\}\Big)\\
=&\Big(\{\bfx\},\{\bfy\},\big\{z+\bfx^\T\{\bfy\}-\{\bfx\}^\T\bfy\big\}\Big).\end{split}\end{equation}
So $[-\frac12,\frac12)^{13}\subset G$ is a fundamental domain of $\Gamma$, and $\{(\bfx,\bfy,z)\}$ is the representative of $(\bfx,\bfy,z)$ in this domain.
 
\subsection{The action} We now equip $X$ with an action by automorphisms. 

We make use of the following partially hyperbolic irreducible $\bZ^2$ action by $6$-dimensional toral automorphisms discovered by S. Katok. The concrete construction can be found in \cite{KN11}*{Example 2.2.20}.

There are two commuting matrices $A,B\in\SL(6,\bZ)$, which can be diagonalized over $\bC$ simutaneouly, respectively into $\diag{\zeta_1^A,\cdots,\zeta_6^A}$ and $\diag{\zeta_1^B,\cdots,\zeta_6^B}$, such that:
\begin{enumerate}
\item $A$ is irreducible over $\bQ$, that is, its characteristic polynomial is irreducible;
\item $\zeta_1^A=\overline{\zeta_2^A}$ and $\zeta_1^B=\overline{\zeta_2^B}$ are conjugate pairs of imaginary eigenvalues, and both pairs have absolute value $1$;
\item All the other eigenvalues $\zeta_3^A,\cdots,\zeta_6^A$, $\zeta_3^B,\cdots,\zeta_6^B$ are real;
\item $\zeta_3^A,\cdots,\zeta_6^A$ have distinct absolute values;
\item $A$ and $B$ are multiplicatively independent, i.e., $A^{n_1}B^{n_2}\neq\Id$ for all non-trivial integer pairs $(n_1,n_2)$.
\end{enumerate}

Therefore there is a two-dimensional subspace $W\subset\bR^6$ on which $A$ and $B$ acts as twisted isometries. More precisely, $W$ is preserved by both $A$ and $B$, and after a change of basis on $W$, $A$ and $B$ respectively act on $W$ by rotations of angles $\arg\zeta_1^A$ and $\arg\zeta_1^B$. Since $\zeta_1^A$ is not a root of unity, these angles are both irrational. 
 
 Let $\beta$ be the $\bZ^2$-action on $\bT^6=\bR^6/\bZ^6$ generated by the multiplications by $A$ and $B$. Namely, $\beta^{(n_1,n_2)}$ is the integer matrix $A^{n_1}B^{n_2}$. Let $\hbeta:\bZ^2\curvearrowright\bR^6$ be the action in which $\bfn$ acts by $\hbeta^\bfn=(\beta^{-\bfn})^\T$.
  
 Define an action $\alpha$ on $G$ by
 \begin{equation}\label{HeisenActionEq}\alpha^\bfn.(\bfx,\bfy,z)=(\beta^\bfn\bfx,\hbeta^\bfn\bfy,z).\end{equation}
 
From \eqref{HeisenBracketEq}, we get
\begin{equation}\begin{split}&\big[\alpha^\bfn.(\bfx,\bfy,z),\alpha^\bfn.(\bfx',\bfy',z')\big]\\
=&\Big(\bfzero,\bfzero,2\bfx^\T(\beta^\bfn)^\T(\beta^{-\bfn})^\T\bfy'-2(\bfx')^\T(\beta^\bfn)^T(\beta^{-\bfn})^T\bfy\big)\Big)\\ 
=&\big(\bfzero,\bfzero,2\bfx^\T\bfy'-2(\bfx')^\T\bfy)\big)=\big[(\bfx,\bfy,z),(\bfx',\bfy',z)\big].
\end{split}\end{equation}

Because the action $\alpha$ preserves the linear structure as well as the Lie bracket $[\cdot,\cdot]$, it also preserves the group rules \eqref{HeisenRuleEq}. Thus $\alpha$ consists of group automorphisms of $G$.

Moreover, since $\alpha$ acts by integer matrices, it leaves the lattice $\Gamma$ invariant. So $\alpha$ descends to an action on $X$ by nilmanifold automorphisms, which we still denote by $\alpha$.

Notice that, the center of $G$ is $[G,G]=\{(\bfzero,\bfzero,z)\}$. The abelianization $G_\ab=G/[G,G]$ is $\bR^6\oplus\bR^6$, parametrized by $(\bfx,\bfy)$. The maximal abelian factor $X_\ab$ is just $\bT^6\times\bT^6$, and the action $\alpha$ projects to the product action $\beta\times\hbeta$ on $X_\ab$.

\subsection{The invariant measure} As remarked earlier, $\beta$ acts on the two dimensional plane $W\subset\bR^6$ by rotations, modulo a change of basis. Let $S_0\subset W$ be a non-trivial circle centered at the origin of small radius, that is preserved by these rotations. We may assume \begin{equation}\label{SmallCircleEq}S_0\subset[-\frac12,\frac12)^6\end{equation} with respect to the original metric of $\bR^6$. There is a unqiue rotation-invariant uniform probability measure $\rmm_{S_0}$ on $S_0$. The projection of $S_0$ from $\bR^6$ into $\bT^6$ is injective because of the small radius. Let $S$ and $\rmm_S$ be respectively the projections of $S_0$ and $\rmm_{S_0}$.  

\begin{lemma}\label{CircleErgo} $\rmm_S$ is $\beta$-invariant. Furthermore, for any finite index subgroup $\Sigma\subset\bZ^2$, $\rmm_S$ is ergodic under $\beta|_\Sigma$.
\end{lemma}
\begin{proof}As $\rmm_S$ is the $\beta$-equivariant projection of $\rmm_{S_0}$, it suffices to prove the same properties for $\rmm_{S_0}$. The invariance is obvious. Observe that $\beta|_\Sigma$ contains a non-trivial power of $A$, which acts by an irrational rotations as $\arg\zeta_1^A$ is irrational. The ergodicity follows.\end{proof}

\begin{lemma}\label{TotalIrrT6} The actions $\beta$ and $\hbeta:\bZ^2\curvearrowright\bT^6$ is totally irreducible and not virtually cyclic.\end{lemma}
\begin{proof} $\beta$ is generated by $A$ and $B$, while $\hbeta$ is generated by $(A^{-1})^T$ and $(B^{-1})^T$. We prove the lemma for $\beta$, and the $\hbeta$-part is similary. 

$\beta$ is not virtually cyclic since $A$ and $B$ are multiplicatively independent.

To show that $\beta$ is totally irreducible, notice that every finite index subaction contains a non-trivial power of $A$. It suffices to show that for $m\neq 0$, $A^m$ is irreducible over $\bQ$.

Suppose $A^m$ is not irreducible. That means there are repeated roots among the eigenvalues $(\zeta_1^A)^m, \cdots, (\zeta_6^A)^m$ of $A^m$. Since $\zeta_3^A,\cdots,\zeta_6^A$ are irrational real numbers of distinct absolute values, and $\zeta_1^A$ and $\zeta_2^A$ are conjugate imaginary numbers from the unit circle, this may happen only if $(\zeta_1^A)^m=(\zeta_2^A)^m$, which is not the case because $\zeta_1^A$ and $\zeta_2^A$ are not roots of unity. So $A^m$ must be irreducible.\end{proof}

\begin{lemma}\label{CircleProdErgo}The measure $\rmm_S\times\rmm_{\bT^6}$ is $\beta\times\hbeta$-invariant and is ergodic under any restriction $(\beta\times\hbeta)|_\Sigma$ to a finite index subgroup.\end{lemma}
\begin{proof}The invariance is clear; we prove the ergodicity.

Lemma \ref{TotalIrrT6} contains that the action $\hbeta$ on $\bT^6$ has no virtually cyclic factor. By Corollary \ref{ProdErgo}, any ergodic component of $\rmm_{S}\times\rmm_{\bT^6}$ with respect to $(\beta\times\hbeta)|_\Sigma$ can be written as $\nu\times\rmm_{\bT^6}$ where $\nu$ is an $\beta|_\Sigma$-ergodic component of $\rmm_S$. But $\nu=\rmm_S$ by Lemma \ref{CircleErgo},  completing the proof.\end{proof}

Write respectively $\overline{(\bfx,\bfy)}$ and $\overline{(\bfx,\bfy,z)}$ for the points represented by $(\bfx,\bfy)$ and $(\bfx,\bfy,z)$ in $X_\ab=\bT^6\times\bT^6$ and $X$. Let $\pi:\overline{(\bfx,\bfy,z)}\mapsto\overline{(\bfx,\bfy)}$ be the projection from $X$ to $X_\ab$.

Notice that every point in the $\beta\times\hbeta$-invariant set $S\times\bT^6\subset X_\ab$ is uniquely represented by some point from $S_0\times[-\frac12,\frac12)^6$. Define a map $\psi: S\times\bT^6\mapsto X$ by:
\begin{equation}\label{SectionMapEq}\psi\big(\overline{(\bfx,\bfy)}\big)=\overline{(\bfx,\bfy,\bfx^\T\bfy)}, \forall \bfx\in S_0,\forall\bfy\in[-\frac12,\frac12)^6.\end{equation}
Then $\psi$ is a piecewise continuous section map over $S\times\bT^6$, as \begin{equation}\label{ProjSectIdEq}\pi\circ\psi=\Id|_{S\times\bT^6}.\end{equation}

\begin{lemma}For all $\bfn\in\bZ^2$, $\psi\circ(\beta^\bfn\times\hbeta^\bfn)=\alpha^\bfn\circ\psi$.\end{lemma}
\begin{proof}Given $(\bfx,\bfy)\in S_0\times\bT^6$, we calculate separately the image of $\overline{(\bfx,\bfy)}$ under both maps. The left hand side gives
\begin{equation}\label{InvSectionEq1}\begin{split}
&\psi\big((\beta^\bfn\times\hbeta^\bfn).\overline{(\bfx,\bfy)}\big)\\
=\ &\psi\Big(\overline{\big(\beta^\bfn\bfx,(\beta^{-\bfn})^\T\bfy\big)}\Big)\\
=\ &\psi\Big(\overline{\big(\beta^\bfn\bfx,\{(\beta^{-\bfn})^\T\bfy\}\big)}\Big)\\
\overset{\eqref{SectionMapEq}}=&\ \overline{\Big(\beta^\bfn\bfx,\big\{(\beta^{-\bfn})^\T\bfy\big\},\bfx^\T(\beta^\bfn)^\T\big\{(\beta^{-\bfn})^\T\bfy)\big\}\Big)}.
\end{split}\end{equation} Here one can apply the definition of $\psi$ in the last equality because $\beta^\bfn\bfx\in S_0$ and $\big\{(\beta^{-\bfn})^\bfy\big\}\in[-\frac12,\frac12)^6$.

On the other hand, the image by the map on the right hand side is
\begin{equation}\label{InvSectionEq2}\begin{split}
&\alpha^\bfn.\psi\big(\overline{(\bfx,\bfy)}\big)\\
=\ &\alpha^\bfn.\overline{(\bfx,\bfy,\bfx^\T\bfy)}=\overline{(\beta^\bfn\bfx,(\beta^{-\bfn})^\T\bfy,\bfx^\T\bfy)}\\
\overset{\eqref{HeisenResEq}}=&\ \overline{\Big(\beta^\bfn\bfx,\big\{(\beta^{-\bfn})^\T\bfy\big\},\bfx^\T\bfy+\bfx^\T(\beta^\bfn)^\T\big\{(\beta^{-\bfn})^\T\bfy\big\}-\{\beta^\bfn\bfx\}^\T(\beta^{-\bfn})^\T\bfy\Big)}\\
\end{split}\end{equation} However, notice that $\{\beta^\bfn\bfx\}=\beta^\bfn\bfx$,
because $\beta^\bfn\bfx\in[-\frac12,\frac12)^6$ by \eqref{SmallCircleEq}.  We deduce that
$\{\beta^\bfn\bfx\}^\T(\beta^{-\bfn})^\T\bfy=(\beta^\bfn\bfx)^T(\beta^{-\bfn})^\T\bfy=\bfx^\T\bfy$.
Therefore, $\eqref{InvSectionEq1}=\eqref{InvSectionEq2}$ and the two maps do coincide.\end{proof} 

Therefore, over the subset $S\times\bT^6$, $\psi$ interwines $\alpha$ with its projection $\beta\times\hbeta$ to $X_\ab$. 

\begin{proof}[Proof of Theorem \ref{Heisenberg}] Let $\mu=\psi_*(\rmm_S\times\rmm_{\bT^6})$. We verify that $\mu$ satisfies the requirements of Theorem \ref{Heisenberg}.

Since $\psi$ commutes $\alpha$ and $\beta\times\hbeta$ and $\rmm_S\times\rmm_{\bT^6}$ is $\beta\times\hbeta$-invariant, $\mu$ is $\alpha$-invariant.

By \eqref{ProjSectIdEq}, $\pi_*\mu=\rmm_S\times\rmm_{\bT^6}$, it further projects to $\rmm_{\bT^6}$ on the second copy of $\bT^6$, on which $\alpha$ projects to $\hbeta$. By Lemma \ref{TotalIrrT6}, the action $\hbeta$ on $\bT^6$ is totally irreducible and not virtually cyclic. Property (1) from Theorem \ref{Heisenberg} is hence verified. In addition, since $\beta$ is also totally irreducible and not virtually cylic,  $\alpha$ does not have any virtually cyclic factor. 

Moreover, by restricting both $\alpha$ and $\beta\times\hbeta$ to an arbitrary finite index subgroup $\Sigma\subset\bZ^2$, one deduce from Lemma \ref{CircleProdErgo} that $\mu$ is ergodic under $\alpha|_\Sigma$. This proves the second part of the theorem.

It remains to check part (3) of the theorem. Suppose, in order to derive a contradiction, that for some subgroup $\Sigma\subset\bZ^2$ of finite index, $H\subset G$ is an $\alpha|_\Sigma$-invariant connected closed subgroup such that $\mu$ desintegrates into $H$-invariant components supported by compact orbits of the left translation action by $H$. In particular, $\mu$ is $H$-invariant, and $\mu$-almost every point has a compact $H$-orbit.

Consider first the projection $H_\ab$ of $H$ to the maximal abelian factor $\bR^6\oplus\bR^6$. The left translation by $H$ must preserve the projection $\rmm_S\times\rmm_{\bT^6}$. But inside the first $\bT^6$ component, the circle measure $\rmm_S$ is not invariant under any translations except those by integer vectors. Since $H$ is connected, this implies that $H_\ab$  lives in $\{\bfzero\}\times\bR^6$.

As $H$ is $\alpha|_\Sigma$-invariant, $H_\ab$, while regarded as a subgroup of $\bR^6$,  is $\hbeta|_\Sigma$-invariant. Moreover, since $H$ has compact orbit for almost every point. $H_\ab$ has compact orbit for almost every point in $X_\ab$, or equivalently, for every point on the second $\bT^6$ component. 

The $H_\ab$-orbit of the origin in $\bT^6$ has to be a $\hbeta|_\Sigma$-invariant subtorus. Whereas the $\hbeta$-action on $\bT^6$ is totally irreducible, forcing this $H_\ab$-orbit to be the full torus $\bT^6$. Thus $H_\ab$ is $\bR^6$.

In our coordinates, which identify the Lie group with the Lie algebra, $H$ is just an $\alpha|_\Sigma$-invariant subspace of $\bR^6\oplus\bR$ which has full projection on $\bR^6$. In addition, $\alpha$ acts linearly by $\hbeta\oplus\Id$ on $\bR^6\oplus\bR$. In particular, for some element $\bfn$ from the finite index subgroup $\Sigma$, $\alpha^\bfn$ acts by $A^m\oplus\Id$ where $m\neq 0$. Due to the fact that all eigenvalues of  $A$ are not roots of unity, all eigenvalues of $A^m\oplus\Id$ on the subspace $\bR^6$ are different from $1$. Therefore, $H$ must be either $\bR^6$ or $\bR^6\oplus\bR$. 

Because $\mu$ is supported on the section $\psi(S\times\bT^6)$, which intersects every fiber of $\pi$ at most once, it is not invariant under translations by the $\bR$ subgroup that corresponds to the $z$-direction. Thus $H$ cannot be $\bR^6\oplus\bR$. 

Therefore $H$ must be the second $\bR^6$ component in $G$, more precisely $H=\{(\bfzero,\bfv,0):\bfv\in\bR^6\}$.

However, it is a direct computation to check which points of $X$ have compact orbits under this subgroup. The $H$-orbit of a point $\overline{(\bfx,\bfy,z)}$ is compact if and only if $\bfx$ is rational.

Observe that there are only countably many rational points in the $\bT^6$ parametrized by $\bfx$, while the measure $\rmm_S$ is an absolutely continuous measure on a twisted circle. So $\rmm_S$-almost every point is irrational. Therefore, with respect to the measure $\mu=\psi_*(\rmm_S\times\rmm_{\bT^6})$, a generic point does not have compact $H$-orbit, which is the contradiction we want. \end{proof}

\end{document}